\documentclass{article}

\usepackage[a4paper,
 total={140mm,237mm},
 left=35mm,
 top=30mm,
 ]{geometry}
\usepackage[utf8]{inputenc}
\usepackage{mathpazo}
\usepackage[hidelinks]{hyperref}
\usepackage{doi}
\usepackage{relsize}

\usepackage[T1]{fontenc}
\usepackage{amssymb, amsmath}
\usepackage[dvipsnames]{xcolor}
\usepackage[english]{babel}
\usepackage{amsthm}
\usepackage{enumitem}
\usepackage{tikz}
\usetikzlibrary{arrows}
\usetikzlibrary{calc}
\usepackage{bm}
\usepackage{csquotes}
\usepackage[shortcuts]{extdash}
\usepackage{xcolor}
\usepackage{authblk}

\usepackage[format=plain,
            labelfont=it,
            textfont=it]{caption}
\usepgflibrary{shapes.geometric}
\usepackage{float}
\allowdisplaybreaks

 \hypersetup{breaklinks=true, %splits links across lines
        }

\usepackage[maxnames=20,backend=biber, url=false, doi=false, eprint=false]{biblatex}

\usepackage{subfiles}

\AtBeginBibliography{\small}
 
\newtheorem{theorem}{Theorem}[section]
\newtheorem{innercustomthm}{Theorem}
\newenvironment{customthm}[1]
  {\renewcommand\theinnercustomthm{#1}\innercustomthm}
  {\endinnercustomthm}

\newtheorem{corollary}[theorem]{Corollary}

\newcommand{\ignore}[1]{}

\newtheorem{lemma}[theorem]{Lemma}
\newtheorem{prop}[theorem]{Proposition}
\newtheorem{fact}[theorem]{Fact}
\theoremstyle{remark}
\newtheorem{remark}[theorem]{Remark}
\newtheorem{notation}[theorem]{Notation}

\theoremstyle{definition}

\newtheorem{definition}[theorem]{Definition}

\newtheorem{innercustomq}{Question}
\newenvironment{customq}[1]
  {\renewcommand\theinnercustomq{#1}\innercustomq}
  {\endinnercustomq}

\newtheorem{prob}[theorem]{Problem}

\def\acts{\curvearrowright}

\title{Minimal operations over permutation groups\footnote{An extended abstract of this paper is accepted for publication by the 40th Annual ACM/IEEE Symposium on Logic in Computer Science
(LICS)~\cite{Binarysymm}.}}

\author[1]{Paolo Marimon}
\author[2]{Michael Pinsker}
\affil[1]{Institut f\"{u}r Diskrete Mathematik \& Geometrie, TU Wien. Vienna, Austria.}
\affil[2]{Institut f\"{u}r Diskrete Mathematik \& Geometrie, TU Wien. Vienna, Austria.}

\begin{document}

\maketitle
\begin{abstract} 
We classify the possible types of minimal operations above an arbitrary permutation group. Above the trivial group, a theorem of Rosenberg yields that there are five types of minimal operations. We show that above any non-trivial permutation group there are at most four such types. Indeed, except above Boolean groups acting freely on a set, there are only three. In particular, this is the case for oligomorphic permutation groups, for which we improve a result of Bodirsky and Chen by showing one of the types in their classification does not exist. Building on these results, we answer three questions of Bodirsky that were previously open.
\end{abstract}
\textbf{Primary MSC classes:} 08A40, 20B05, 08A70.\\
\textbf{Secondary MSC classes:} 03C05, 20B07, 20M20.

\tableofcontents

\section{Introduction}
\subsection{Motivation}\label{subsect:motivation}

One of the aims of universal algebra is  classifying algebras up to term equivalence~\cite{szendrei1986clones}: term equivalence gives a language-independent way of studying most important properties of an algebra such as subalgebras, congruences, endomorphisms, etc. Indeed, much of universal algebra focuses on obtaining structural information about an algebra from the existence of a given kind of term operation. This perspective naturally gives rise to the notion of a function clone: a set of operations from some finite power of the domain into itself closed under composition and projections (Definition~\ref{def:clone}). Algebras up to term equivalence correspond to clones: we can associate to each algebra the clone of its term operations, term equivalent algebras give rise to the same clone, and every clone is associated to the algebra with function symbols for all of its operations.\\

The clones on a given domain $B$ form a lattice with respect to inclusion. Given a transformation monoid $\mathcal{T}$ on $B$, the clones whose unary operations equal $\mathcal{T}$ form an interval in this lattice~\cite[Proposition 3.1]{szendrei1986clones}, called the monoidal interval of $\mathcal{T}$, in which the clone $\langle \mathcal{T}\rangle$ generated by the transformation monoid and projections is the smallest element: it consists of all functions obtained from functions in $\mathcal{T}$ by adding dummy variables. Since monoidal intervals partition the lattice of clones, the task of classifying the clones on a given domain can be split into their study. Universal algebra distinguishes itself from group and semigroup theory by its focus on higher arity operations. This might be one reason why most research has been on idempotent algebras and clones, i.e. the interval lying above the trivial monoid, where all unary operations in the clones are trivial. Moving to  monoidal intervals for larger monoids $\mathcal{T}$, it becomes easier for an operation to generate another modulo $\mathcal{T}$. Hence, overall, large monoids tend to have monoidal intervals which are smaller and easier to describe\footnote{Note that it is not the case that the larger the monoid, the smaller the associated monoidal interval. For example, for $B$ finite, there are permutation groups whose monoidal interval consists of a single element~\cite{palfy1982contributions, kearnes2001collapsing}, whilst the monoidal interval above the full monoid $\mathcal{T}_B$  has $(|B|+1)$-many elements~\cite{burle1967classes}.} than that of the trivial monoid~\cite{burle1967classes, Haddad_Rosenberg_1994, machida2007minimal, PINSKER200859}, and semigroup theory plays a stronger role in their study ~\cite{krokhin1995monoid, PINSKER200859}.\\

\ignore{
Moving to  monoidal intervals for larger monoids $\mathcal{T}$, one would expect the interval to shrink since the larger the monoid, the easier it is for a given operation to generate another modulo $\mathcal{T}$. The way this works is not straightforward, and semigroup theory plays a stronger role in the picture~\cite{krokhin1995monoid, PINSKER200859}: for $B$ finite there are permutation groups whose monoidal interval consists of a single element~\cite{palfy1982contributions, kearnes2001collapsing}, whilst the monoidal interval above the full monoid $\mathcal{T}_B$  has $(|B|+1)$-many elements~\cite{burle1967classes}. Nevertheless, for large monoids it is often possible to describe their monoidal intervals~\cite{burle1967classes, Haddad_Rosenberg_1994, machida2007minimal, PINSKER200859}.}

It is natural to ask what is the minimal amount of structure that can be found in a clone of a given monoidal interval if that clone contains an essential operation, i.e. an  operation which depends on more than just one variable (and hence the clone is not just the disguised semigroup $\langle \mathcal{T}\rangle$). This gives rise to the definition of a minimal clone above $\langle \mathcal{T}\rangle$: a clone $\mathcal{C}\supsetneq\langle\mathcal{T}\rangle$ with no other clone sandwiched in between them. It is easy to see that any minimal clone above $\langle \mathcal{T}\rangle$ is generated by a single operation together with $\langle \mathcal{T}\rangle$, which, when chosen to be of minimal arity, is called a minimal operation above $\langle \mathcal{T}\rangle$. Any unary operation not in $\mathcal{T}$ generates together with $\langle\mathcal{T}\rangle$ a clone outside the monoidal interval of $\mathcal{T}$ all whose operations are essentially unary, i.e.\ depend on at most one variable  (Definition~\ref{def:essential}). If on the other hand a minimal operation above $\langle\mathcal{T}\rangle$ is of arity $>1$, it has to be essential,  and the clone it generates together with $\langle\mathcal{T}\rangle$ lies in the monoidal interval of $\mathcal{T}$. Hence, minimal elements in the monoidal interval above $\langle \mathcal{T}\rangle$ correspond to minimal clones above $\langle \mathcal{T}\rangle$ whose minimal operations are non-unary.  Over a finite set, every clone $\mathcal{D}\supsetneq{\langle\mathcal{T}\rangle}$ contains a minimal clone above ${\langle\mathcal{T}\rangle}$. While this is in general  false over an infinite set (cf.~\cite{csakany2005minimal}), it still holds true in many important situations, e.g. if $\mathcal{T}$ is the automorphism group of an $\omega$-categorical structure in a finite language (see below,~\cite{BodChen}). \\

The question of understanding the minimal clones above the clone $\langle \mathrm{Id} \rangle$  generated by the trivial monoid, which is the clone $\mathcal{P}_B$ consisting of all  projections on the domain $B$, has a long history in universal algebra. A classical theorem of Rosenberg~\cite{fivetypes} shows that over a finite set $B$  there are only five types of minimal operations:

\begin{theorem}[Five Types Theorem~\cite{fivetypes}]\label{thm:firstheorem} Let $B$ be finite and let $f$ be a minimal operation above $\langle \mathrm{Id} \rangle$. Then $f$ is of one of the following types:
\begin{enumerate}
    \item a unary operation;
    \item a binary operation;
    \item a ternary majority operation;
    \item a ternary minority operation of the form $x+y+z$ in a Boolean group $(B,+)$;
    \item a $k$-ary semiprojection for some $k\geq 3$.
\end{enumerate}
\end{theorem}
We describe these operations in Definition~\ref{def:mainops}. Whilst splitting the possible minimal operations into five types, Rosenberg does not give necessary and sufficient conditions for an operation of these types to be minimal except for the cases of a unary operation and of a minority. Hence, current research, e.g.~\cite{brady2023coarse, machida2024orderly}, focuses on describing in more detail the possible behaviours of the remaining three cases. From Post~\cite{postlattice} and Cs\'{a}k\'{a}ny~\cite{csakany1983all}, whose work precedes that of Rosenberg, we know all of the minimal clones on a two-element and a three-element domain. However, even the problem of describing all minimal clones on a four-element domain is still open. We refer the reader to~\cite{szendrei2024ivo} for a recent review of this topic, which discusses progress of~\cite{brady2023coarse} on the case of binary operations. Other review papers on minimal operations are~\cite{quackenbush1995survey} and~\cite{csakany2005minimal}. In a sense, a full classification of the minimal clones on a finite domain seems to be harder than that of the maximal clones, for which we have a complete description by Rosenberg~\cite{rosenberg1965structure, rosenberg1970uber}.\\

It has been observed in~\cite{BodCSP} that basically following the same proof, a weakening of Theorem~\ref{thm:firstheorem} which still features five types of minimal operations can be obtained for arbitrary transformation monoids $\mathcal T$ which do not contain any constant operations (Theorem~\ref{fivetypes}). However, for $\mathcal{T}$ a non-trivial permutation group, one might expect stronger statements to hold. This is, in particular,  suggested both by a long history of results showing that the monoidal interval of many  permutation groups collapses into a single element~\cite{Haddad_Rosenberg_1994,palfy1984unary, palfy1982contributions, kearnes2001collapsing, krokhin1995monoid}, and by a result of Bodirsky and Chen (Theorem~\ref{thm:fourtypesbodchen} below) who obtained a stronger classification with only four types above certain infinite permutation groups. Function clones whose unary part forms a permutation group are of particular importance since on the one hand, they correspond to relational structures which are cores (in the sense of~\cite{hell1992core, bodirsky2007cores}, see below), and since on the other hand, they promise connections  between group theory and universal algebra.\\

We investigate the problem of classifying minimal operations above permutation groups both over a finite and an infinite domain. Our main interest lies in clones that arise as the higher-arity symmetries (i.e. polymorphisms) of some relational structure. Analogously to the case of groups, where, over arbitrary domains, automorphism groups of relational structures correspond to those permutation groups which are closed in the pointwise convergence topology, polymorphism clones of relational structures correspond to  topologically closed clones. Hence, we will focus on clones closed with respect to the pointwise convergence topology (Definition~\ref{def:closed}), also known as locally closed clones; this topology trivialises on a finite domain. Consequently, for  $G$ a permutation group acting on $B$, we investigate minimality above $\overline{\langle G\rangle}$ in the lattice of closed clones (Definitions~\ref{def:minimalc} and~\ref{def:minimalop}), where $\overline{\langle G\rangle}$ is the closed clone generated by $G$: it consists of all (essentially unary) functions arising by the addition of  dummy variables to a unary function 
which on each finite subset of $B$ agrees with some element of $G$. We remark that focusing on closed clones only increases the generality of our results, basically  since the notion of generation it gives rise to is more complex than generation just by composition and projections. In particular, our methods also classify minimal operations above permutation groups in the lattice of all clones on a given domain (cf. Remark~\ref{rem:oknoclosure}). Meanwhile, our focus on closed clones allows us to obtain several results for infinite-domain constraint satisfaction problems (CSPs).\\

Given a relational structure $B$, its constraint satisfaction problem $\mathrm{CSP}(B)$ is the computational problem  of deciding for a finite structure $A$ (in the same relational signature) whether it maps homomorphically into $B$. On a finite domain, a deep theory using tools from universal algebra has been  developed, building on the insight that the richer the clone of polymorphisms of $B$ is, the easier  $\mathrm{CSP}(B)$ is~\cite{JBK}. This culminated in the Bulatov~\cite{bulatov2017dichotomy} and Zhuk~\cite{zhuk2020proof} proof of the Feder and Vardi conjecture~\cite{feder1993monotone}: assuming P$\neq$NP, all constraint satisfaction problems on a finite domain $B$ are either in P or NP-complete, with the dividing line between the two classes being whether the polymorphisms of $B$ satisfy any non-trivial height~1 identity (which yields that $\mathrm{CSP}(B)$ is in P). Many of the universal-algebraic methods  also lift to the context of $B$ being a countable structure whose automorphism group is {oligomorphic}, i.e. has finitely many orbits on $n$-tuples for each $n$~\cite{BodCSP}. These are also known as ${\omega}${-categorical} structures, since their first-order theory has a unique countable model up to isomorphism~\cite[Theorem 7.3.1]{fathodges}. The main research programme on constraint satisfaction problems on an infinite domain aims at answering a question of Bodirsky and Pinsker of whether a complexity dichotomy holds for CSPs of a natural class of $\omega$-categorical structures (i.e. first-order reducts of finitely bounded homogeneous structures, cf.~\cite{BPP-projective-homomorphisms,irrelevant,BodCSP, infinitesheep}).\\

When studying $\mathrm{CSP}(B)$ in the finite or oligomorphic setting, if the polymorphism clone of $B$, $\mathrm{Pol}(B)$, is essentially unary (and with no constant operations), then $\mathrm{CSP}(B)$ is NP-hard~\cite{JBK,Topo-Birk}. Hence, we may focus on the case of $\mathrm{Pol}(B)$ containing some essential operation of arity $>1$. Moreover, for the purposes of studying $\mathrm{CSP}$s, it is sufficient to consider the case of $B$ being a (model-complete) {core}, i.e. such that the closure of its automorphism group in the pointwise convergence topology $\overline{\mathrm{Aut}(B)}$ corresponds to its endomorphism monoid $\mathrm{End}(B)$~\cite{bodirsky2007cores}. In general, since $B$ is in a finite language, $\mathrm{Pol}(B)$ contains some  operation of arity $>1$ minimal above its essentially unary part~\cite{BodChen, JBK}, which, if $B$ is a core, is $\overline{\langle\mathrm{Aut}(B)\rangle}$.\\

Indeed, classifying the minimal operations above $\overline{\langle\mathrm{Aut}(B)\rangle}$ is a part of the traditional "bottom-up" approach to a complexity dichotomy for CSPs: first classify the minimal operations above $\overline{\langle \mathrm{Aut}(B)\rangle}$, and then investigate %structure the remaining part of the argument by considering 
 which functions may live above different minimal operations and what identities they may satisfy. This was the approach of Jeavons~\cite{jeavons1998algebraic} in his algebraic proof of Schaefer's complexity dichotomy for constraint satisfaction problems on a two-element domain~\cite{schaefer1978complexity}. Moreover, this strategy was heavily used in an oligomorphic setting~\cite{bodirsky2010complexity, bodirsky2015schaefer, kompatscher2018complexity, bodirsky2019constraint}, with the minimal operations over the automorphism group of the random graph being classified in an own paper,~\cite{minrandom},  for this purpose. Other modern techniques~\cite{Smoothapp, mottet2023order, feller2024algebraic, bitter2024generalized}, while avoiding this precise approach,  also require arguments to find, given the existence of some polymorphism which is not essentially unary, the existence of a binary one.\\

Motivated by these applications, Bodirsky and Chen~\cite{BodChen} previously obtained a Rosenberg-style type classification of minimal operations above oligomorphic permutation groups:
\begin{theorem}[Four types, oligomorphic case, ~\cite{BodChen, BodCSP}] \label{thm:fourtypesbodchen} Let $G\acts B$ be an oligomorphic permutation group on a countably infinite $B$. Let $f$ be minimal above $\overline{\langle G\rangle}$. Then, $f$ is of one of the following four types:
    \begin{enumerate}
    \item a unary operation;
    \item a binary operation;
    \item a ternary quasi-majority operation;
    \item a $k$-ary quasi-semiprojection for some $3\leq k\leq 2r-s$, where $r$ is the number of $G$-orbitals (orbits under the componentwise action of $G$ on pairs) and $s$ is the number of $G$-orbits.
\end{enumerate}
\end{theorem}
Again, these operations are defined in Definition~\ref{def:mainops}. In this paper, we study the fully general context of $G\acts B$ being a non-trivial group acting faithfully on $B$, where $B$ is a set (not necessarily finite or countable). Nevertheless, our results will also offer a surprising  improvement of Theorem~\ref{thm:fourtypesbodchen}. 

\subsection{Main results}

\subsubsection{The minimal functions theorem}

We obtain the following:

\begin{customthm}{A}\label{maintheorem} Let $G\acts B$ be a non-trivial group acting faithfully on $B$ with $s$ many orbits (where $s$ is possibly infinite). Let $f$ be a minimal operation above $\overline{\langle G \rangle}$. Then, $f$ is of one of the following types:
\begin{enumerate}
    \item a unary operation;
    \item a binary operation;
    \item \label{cas:min} a ternary quasi-minority operation of the form $\alpha\mathfrak{q}$ for $\alpha\in G$, where 
    \begin{itemize}
        \item $G$ is a Boolean group acting freely on $B$;
        \item the operation $\mathfrak{q}$ is a $G$-invariant Boolean Steiner $3$-quasigroup.
    \end{itemize}
    \item \label{cas:semi} a $k$-ary orbit\-/semiprojection for $3\leq k \leq s$. 
\end{enumerate}
    Moreover, in case (3), such operations exist if and only if $s=2^n$ for some $n\in\mathbb{N}$ or is infinite. In case (4), if $G$ is finite or oligomorphic, we have that a minimal $k$-ary orbit semiprojection exists for each $2\leq k\leq s$.
\end{customthm}
The definitions of orbit\-/semiprojection and $G$-invariant Boolean Steiner $3$-quasigroup are given in Definitions~\ref{def:orbitsemi} and~\ref{def:boolgtwisted}, and they correspond to strengthenings of being a quasi\-/semiprojection, and being a minority of the form $x+y+z$ in a Boolean group respectively. We actually obtain stronger results classifying operations satisfying weaker conditions than minimality. We outline these in Subsection~\ref{sub:main}.\\

A surprising feature of Theorem~\ref{maintheorem} is that (quasi-)majorities, item 3 in Theorems~\ref{thm:firstheorem} and~\ref{thm:fourtypesbodchen}, cannot exist as minimal above any non-trivial permutation group. Moreover, whenever $G$ is not a Boolean group acting freely on $B$ with $2^n$ or infinitely many orbits, there are only three possible minimal operations. Since the action of an oligomorphic permutation group is never free, our results improve on Theorem~\ref{thm:fourtypesbodchen} from~\cite{BodChen} in two ways: (i) we show that minimal quasi-majorities cannot exist, and (ii)  for the only remaining type of operations of arity $>2$ (i.e. the quasi-semiprojections), we specify the behaviour on the orbits and bound their arity by the number of orbits rather than the number of orbitals. More generally, no analogue of Theorem~\ref{maintheorem} was known even on a finite domain, where it constitutes progress for the study of the monoidal intervals of permutation groups and their collapse (cf. Claim 3.6 in~\cite{kearnes2001collapsing}). Finally, it should have implications for the study of CSPs both of non-rigid finite structures and of $\omega$-categorical structures.

\subsubsection{Binary essential operations and three answers to questions of Bodirsky}

Theorem~\ref{maintheorem} has various consequences for finding low-arity essential operations in polymorphism clones of non-rigid structures, which we explore in Section~\ref{consequences}. In particular, some of our original motivation for studying minimal operations above permutation groups comes from the following question of Bodirsky:
\begin{customq}{1}({Question 14.2.6  (24) in~\cite{BodCSP}})\label{Bodq} Does every countable $\omega$-categorical model-complete core with an essential polymorphism also have a binary essential polymorphism?
\end{customq}
 Whilst we give a negative answer to Question~\ref{Bodq} in the fully general setting (Corollary~\ref{cor:counterex}), Theorem~\ref{maintheorem} (or more precisely Theorem~\ref{thmc}) yields a positive answer whenever $G\acts B$ is transitive (Corollary~\ref{niceres}). Moreover, since orbit\-/semiprojections only satisfy trivial identities, we get the following surprising result:
\begin{customthm}{B}(Theorem~\ref{thm:findbin})\label{mainthmb}
 Let $G\acts B$ be non-trivial such that $G$ is not a Boolean group acting freely on $B$. Let $\mathcal{C}$ be a closed clone belonging to the  monoidal interval of  $\overline{\langle G \rangle}$. Suppose that $\mathcal{C}$ contains no binary essential operation. Then, $\mathcal{C}$ has a uniformly continuous clone 
 homomorphism to $\mathcal{P}_{\{0,1\}}$, the clone of projections on the set $\{0,1\}$.
\end{customthm}
In a finite or $\omega$-categorical context, a uniformly continuous 
 clone homomorphism to $\mathcal{P}_{\{0,1\}}$ is known to imply NP-hardness~\cite{Topo-Birk}, meaning that for the purposes of CSPs, in these cases one may assume there is some binary essential operation in $\mathrm{Pol}(B)$ as long as $\mathrm{Aut}(B)$ is non-trivial and not a Boolean group acting freely on $B$.\\

Previous literature on infinite-domain CSPs also developed techniques to find binary essential polymorphisms given the existence of an essential polymorphism~\cite{BodKara, minrandom, Smoothapp}. These rely on $G\acts B$ having a particular property, the orbital extension property (Definition~\ref{def:OEP}), for which there were no known primitive oligomorphic counterexamples. In Subsection~\ref{sub:OEP}, we give examples of primitive finitely bounded homogeneous structures without the orbital extension property answering negatively another question of Bodirsky:

\begin{customq}{2}(Question 14.2.1 (2) in {\cite{BodCSP}})\label{qOEP}
    Does every primitive oligomorphic permutation group have the orbital extension property?
\end{customq}

 Hence, our results on finding binary essential operations expand the reach of pre-existing techniques. Our examples come from Cherlin's recent classification of homogeneous $2$-multitournaments~\cite{Cherlinclass}.\\

Often, arguments for the tractability of $\mathrm{CSP}$s rely on finding a binary injective polymorphism~\cite{BodKara, bodirsky2010complexity, minrandom, Smoothapp, bodirsky2020complexity}. Since in Theorem~\ref{mainthmb} we found general reasons for why binary essential polymorphisms appear in tractable $\mathrm{CSP}$s of $\omega$-categorical structures, it is reasonable to investigate when we can also get a binary injective polymorphism. We conclude the paper with Section~\ref{sec:datalog}, where we answer negatively another question of Bodirsky, which asks whether binary injective polymorphisms can always be found when $\mathrm{CSP}(B)$ meets sufficiently strong tractability assumptions:

\begin{customq}{3}(Question 14.2.6 (27)  in {\cite{BodCSP}})\label{q:datalog}
    Does every $\omega$-categorical structure without algebraicity that can be solved
by Datalog also have a binary injective polymorphism?
\end{customq}

We give a counterexample to Question~\ref{q:datalog} which is a reduct of a finitely bounded homogeneous unary structure, i.e. it lies within the scope of above-mentioned dichotomy conjecture of Bodirsky and Pinsker. \\ 

\textbf{Acknowledgements:} The first author would like to thank Jakub Rydval and Carl-Fredrik Nyberg Brodda for helpful discussions. He would also like to thank the participants to a reading group on Bodirsky's book~\cite{BodCSP}. Funded by the European Union (ERC, POCOCOP, 101071674). Views and opinions expressed are however those of the authors only and do not necessarily reflect those of the European Union or the European Research Council Executive Agency. Neither the European Union nor the granting authority can be held responsible for them. 
This research was funded in whole or in part by the Austrian Science Fund (FWF) [I 5948]. For the purpose
of Open Access, the authors have applied a CC BY public copyright licence to any Author Accepted Manuscript (AAM) version arising from this submission.
\subsection{Preliminaries}\label{sub:prelim}
In this subsection we give some basic definitions and results relevant to the rest of the paper. We follow the discussion of~\cite{szendrei1986clones} and~\cite{BodChen}. We refer the reader to these references for the relevant proofs.\\

Let $B$ denote a set. For $n\in\mathbb{N}$, $\mathcal{O}^{(n)}$ denotes the set $B^{{B}^n}$ of functions  $B^n\to B$. We write
\[\mathcal{O}:=\bigcup_{n\in\mathbb{N}} \mathcal{O}^{(n)}.\]
We equip $\mathcal{O}^{(n)}$ with the product topology and then $\mathcal{O}$ with the sum topology, where $B$ was endowed with the discrete topology.

\begin{definition}\label{def:clone} Let $B$ be a set. A \textbf{function clone} over $B$ is a set $\mathcal{C}\subseteq \mathcal{O}$ such that
\begin{itemize}
    \item $\mathcal{C}$ contains all projections: for each $1\leq i\leq k\in\mathbb{N}$, $\mathcal{C}$ contains the $k$-ary projection to the $i$th coordinate $\pi_{i}^k\in\mathcal{O}^{(k)}$, given by $(x_1, \dots, x_k)\mapsto x_i$;
    \item $\mathcal{C}$ is closed under composition: for all $f\in\mathcal{C}\cap\mathcal{O}^{(n)}$ and all $g_1, \dots, g_n\in \mathcal{C}\cap\mathcal{O}^{(m)}$, $f(g_1, \dots, g_n)$, given by
    \[(x_1, \dots, x_m)\mapsto f(g_1(x_1, \dots, x_m), \dots, g_n(x_1, \dots, x_m)),\]
    is in $\mathcal{C}\cap\mathcal{O}^{(m)}$.
\end{itemize}
\end{definition}

\begin{definition}\label{def:closed} Given $\mathcal{S}\subseteq \mathcal{O}$, $\langle \mathcal{S}\rangle$ denotes the smallest function clone containing $\mathcal{S}$. It consists of all functions in $\mathcal{O}$ which can be written as a term function using functions from $\mathcal{S}$ and projections. Meanwhile, $\overline{\mathcal{S}}$ denotes the closure of $\mathcal{S}$ in $\mathcal{O}$ with respect to the topology we described. We call a function clone $\mathcal{C}$ \textbf{closed}, when it is closed with respect to this  topology. We have that $f\in\overline{\mathcal{S}}$ if for each finite $A\subseteq B$, there is some $g\in\mathcal{S}$ such that $g|_{A}=f|_{A}$. It is easy to see that for $\mathcal{S}\subseteq\mathcal{O}$, $\overline{\langle \mathcal{S}\rangle}$ is the smallest closed function clone containing $\mathcal{S}$. We say that $\mathcal{S}$ \textbf{locally generates} $g$ if $g\in\overline{\langle \mathcal{S}\rangle}$. Frequently, we will just say that $\mathcal{S}$ generates $g$ since we will always work with local generation. 
\end{definition}

\begin{definition}\label{def:pol}
Let $B$ be a relational structure. 
 We say that $f:B^n\to B$ is a \textbf{polymorphism} of $B$ if it preserves all relations of $B$: for any 
 such relation $R$,
\[\text{ if } \begin{pmatrix}
    a^1_1\\
    \vdots\\
    a_1^k
\end{pmatrix}, \dots, 
\begin{pmatrix}
    a_n^1\\
    \vdots\\
    a_n^k
\end{pmatrix}
\in R \text{, then }
\begin{pmatrix}
    f(a_1^1, \dots, a_n^1)\\
    \vdots\\
    f(a_1^k, \dots, a_n^k)
\end{pmatrix}\in R.\]
    
    We call the set of polymorphisms of $B$ the \textbf{polymorphism clone of} $B$, and denote it by $\mathrm{Pol}(B)$. 
\end{definition}

Below, we list some basic facts about closed clones. Their proofs can be found in~\cite{szendrei1986clones}:

\begin{fact} The closed subclones of $\mathcal{O}$ ordered by inclusion form a lattice. 
\end{fact}

\begin{fact}[{\cite{rosenberg1999locally}}] \label{fact:correspondence} The closed subclones of $\mathcal{O}$ on a set $B$  correspond to the polymorphism clones of relational structures on $B$.
\end{fact}

\begin{definition}\label{def:minimalc} Let $\mathcal{D}\supsetneq \mathcal{C}$ be closed subclones of $\mathcal{O}$. We say that $\mathcal{D}$ is \textbf{minimal above} $\mathcal{C}$ if there is no closed clone $\mathcal{E}$ such that 
\[\mathcal{C}\subsetneq \mathcal{E}\subsetneq \mathcal{D}.\]
\end{definition}

\begin{definition}\label{def:minimalop} Let $\mathcal{C}$ be a closed subclone of $\mathcal{O}$. We say that an operation $f\in\mathcal{O}\setminus\mathcal{C}$ is \textbf{minimal above } $\mathcal{C}$ if $\overline{\langle \mathcal C\cup\{f\}\rangle}$ is minimal above $\mathcal{C}$ and every operation in $\overline{\langle \mathcal C\cup\{f\}\rangle}\setminus \mathcal{C}$ has arity greater than or equal to that of $f$.
\end{definition}

\begin{fact} Let $\mathcal{C}\subseteq\mathcal{D}$ be closed subclones of $\mathcal{O}$. Then, $\mathcal{D}$ is minimal above $\mathcal{C}$ if and only if there is some operation $f$ on $B$ which is minimal above $\mathcal{C}$ and such that $\overline{\langle \mathcal{C}\cup\{f\}\rangle}=\mathcal{D}$. 
\end{fact}

\begin{notation} We use the symbol '$\approx$' to denote identities which hold universally. For example, instead of
\[\forall x,y\in B \ f(x,y)=f(y,x),\]
we write
\[f(x,y)\approx f(y,x).\]
\end{notation}

\begin{definition} \label{def:clonehom} Let $\mathcal{C}$ and $\mathcal{D}$ be clones with domains $C$ and $D$ respectively. A map $\eta:\mathcal{C}\to\mathcal{D}$ is a \textbf{clone homomorphism} if it preserves arities and universally quantified identities. In the special case where the domain $D$ of $\mathcal{D}$ is finite, we say that $\eta$ is \textbf{uniformly continuous} if there exists a finite $A\subseteq C$ such that $f_{\upharpoonright A}=g_{\upharpoonright A}$ implies $\eta(f)=\eta(g)$ for all $f,g\in \mathcal{C}$; see~\cite{uniformbirkhoff} for more information on the uniform structure justifying this terminology.
\end{definition}

\begin{definition}\label{def:mainops} We define some types of operations relevant to the rest of the paper by the identities that they satisfy:
\begin{itemize}
    \item a \textbf{ternary quasi-majority} is a ternary operation $m$ such that 
\[m(x,x,y)\approx m(x,y,x)\approx m(y,x,x)\approx m(x,x,x);\]
\item a \textbf{quasi-Malcev} operation is a ternary operation $M$ such that
\[M(x,y,y)\approx M(y,y,x)\approx M(x,x,x);\]
    \item a \textbf{ternary quasi-minority} is a ternary operation $\mathfrak{m}$ such that 
\[\mathfrak{m}(x,x,y)\approx \mathfrak{m}(x,y,x)\approx \mathfrak{m}(y,x,x)\approx \mathfrak{m}(y,y,y);\]
\item a \textbf{quasi-semiprojection} is a $k$-ary operation $f$ such that there is an $i\in\{1, \dots, k\}$ and a unary operation $g$ such that whenever $(a_1, \dots, a_k)$ is a non-injective tuple from $B$,
\[f(a_1, \dots, a_k)= g(a_i).\]
\end{itemize}
For each of the operations above, we remove the prefix "quasi" when the operation is \textbf{idempotent}, i.e. satisfies $f(x, \dots, x)\approx x$. For example, a ternary majority is a ternary quasi-majority $m$ which also satisfies $m(x,x,x)\approx x$; in the case of a semiprojection idempotency  implies $g(x)\approx x
$.
\end{definition}

\begin{definition}\label{def:essential} We say that a $k$-ary operation $f$ is \textbf{essentially unary} if there is a unary operation $g$ and $1\leq i\leq k$ such that
\[f(x_1, \dots, x_k)\approx g(x_i).\]
We say that $f$ is \textbf{essential} if it is not essentially unary.  A clone $\mathcal{C}$ is essentially unary if all of its operations are essentially unary. It is essential if it has an essential operation.
\end{definition}

\subsection{Classifying minimal operations via almost minimal operations}\label{sub:main}
Our path towards  Theorem~\ref{maintheorem}  consists in first classifying operations satisfying weaker versions of minimality and then refining the types of operations obtained from this classification as we approach minimality. The main notion that we investigate is that of almost minimality (Definition~\ref{def:almostmin}), which captures that $f$ does not locally generate anything new of strictly lower arity together with $\overline{\langle G\rangle}$ (which is clearly implied by minimality). Already looking at this notion, for $G\acts B$ not being a Boolean group acting freely on $B$, we get just three types of almost minimal operations. In the case of Boolean groups, after classifying almost minimal operations, we need the  stronger  notion  of  strictly almost minimal operations (Definition~\ref{def:semimin}): functions such that every operation of lower or equal arity they locally generate with $\overline{\langle G\rangle}$ locally generates them back with $\overline{\langle G\rangle}$ (but an operation of larger   arity might not). Strict almost minimality is therefore situated between almost minimality and minimality. Finally, once strictly almost minimal operations are understood, we move to minimal operations by adding the requirement that also operations of strictly larger arity which are locally generated by $f$ together with $\overline{\langle G\rangle}$ locally generate $f$ back with $\overline{\langle G\rangle}$.\\

\begin{notation}
   Let $G$ be a group acting on a set $B$. By $G\acts B$ we mean that $G$ acts faithfully on $B$. Moreover, throughout this paper, we assume that $G$ is non-trivial.
\end{notation}

\begin{definition} A \textbf{Boolean group} is a group where every element except the identity has order $2$; in other words, an elementary Abelian 2-group.
\end{definition}

\begin{definition} We say that the action $G\acts B$ is \textbf{free} if no non-trivial element in $G$ fixes a point of $B$.
\end{definition}

\subsubsection{Almost minimal operations}

\begin{definition}\label{def:almostmin} Let $G\acts B$. The $k$-ary operation $f\in\mathcal{O}\setminus\overline{\langle G\rangle}$ is \textbf{almost minimal} if $\overline{\langle G \cup\{f\}\rangle}\cap \mathcal{O}^{(r)}=\overline{\langle G\rangle}\cap \mathcal{O}^{(r)}$ for each $r<k$. 
\end{definition}

\begin{remark} There is no harm in similarly defining a $k$-ary operation being almost minimal above an arbitrary closed clone $\mathcal{D}$ (rather than $\overline{\langle G\rangle}$). Moreover, note that if $f$ is minimal above $\mathcal{D}$, then it is also almost minimal above it. 
\end{remark}

\begin{remark}\label{alwaysunary} Unary functions in $\mathcal{O}^{(1)}\setminus \overline{\langle G\rangle}$ are vacuously almost minimal.  
\end{remark}

Whilst minimal operations do not  exist for every action $G\acts B$ over an infinite domain $B$ (cf.~\cite{csakany2005minimal}), it is easy to see that almost minimal operations do:
%exist over any domain $B$ and any permutation group $G\acts B$:

\begin{lemma}\label{lem:existencealmost} Let $G\acts B$. Suppose that $\mathcal{C}$ is a closed function clone strictly containing $\overline{\langle G\rangle}$. Then, $\mathcal{C}$ contains some almost minimal operation.
\end{lemma}
\begin{proof}
Let $f\in\mathcal{C}\setminus\overline{\langle{G}\rangle}$ be of minimal arity. Any function that  $f$ locally generates with $G$ belongs to $\mathcal C$; hence if such a function is of strictly smaller arity than $f$, then it also belongs to $\overline{\langle{G}\rangle}$, by the minimality of the arity of $f$. Thus,  $f$ is almost minimal. 
    %We prove that if $\mathcal{C}\supseteq\overline{\langle G\rangle}$ has no almost minimal operation, then $\mathcal{C}=\overline{\langle G\rangle}$.  
     %Suppose that $f$ is of minimal arity such that $f\in\mathcal{C}\setminus\overline{\langle{G}\rangle}$. By Remark~\ref{alwaysunary}, $f$ has arity $>1$. Since $f$ is not almost minimal, it generates some $g$ of strictly smaller arity in $\mathcal{C}\setminus\overline{\langle{G}\rangle}$. But this contradicts the choice of $f$ as being of minimal arity. Hence, $\mathcal{C}=\overline{\langle{G}\rangle}$ as desired.
\end{proof}

As mentioned above, in this paper we also classify the almost minimal operations above permutation groups. We begin from the observation that the following version of Theorem~\ref{thm:firstheorem} in a non-idempotent context from~\cite[Theorem 6.1.42]{BodCSP} (recall the discussion after Theorem~\ref{thm:firstheorem}) only uses the assumption of almost minimality:
%(see Remark~\ref{rem:onlyalmostmin} below):

\begin{theorem}[Five types theorem over essentially unary clones,~\cite{BodCSP}]\label{fivetypes} Let $\mathcal{T}$ be a transformation monoid on $B$ without constant operations and let $f$ be minimal above $\overline{\langle\mathcal{T}\rangle}$. Then, up to permuting its variables, $f$ is of one of the following five types:
\begin{enumerate}
    \item a unary operation;
    \item a binary operation;
    \item a ternary quasi-majority operation;
    \item a quasi-Malcev operation;
    \item a $k$-ary quasi-semiprojection for some $k\geq 3$.
\end{enumerate}
\end{theorem}

\begin{remark}\label{rem:onlyalmostmin} Inspecting the proof of Theorem~\ref{fivetypes} in~\cite[\S 6.1.8]{BodCSP}, one can see that the only property of minimal operations being used is almost minimality. In particular, Theorem~\ref{fivetypes} actually gives a classification into five types of the almost minimal operations above a transformation monoid (not containing any constant operations). The proof of Theorem~\ref{fivetypes} follows the same structure as the first part of the standard proof of Rosenberg's Theorem~\ref{thm:firstheorem} (cf.~\cite{fivetypes, csakany2005minimal}).
\end{remark}

We classify almost minimal operations above permutation groups $G\acts B$ using Theorem~\ref{fivetypes} for $\mathcal{T}=G$ as a starting point. Firstly, in Lemma~\ref{semiorb}, we prove that for any permutation group $G\acts B$, a $k$-ary quasi-semiprojection almost minimal above $\overline{\langle G\rangle}$ satisfies the stronger property of being an orbit\-/semiprojection:

\begin{definition}\label{def:orbitsemi} Let $G\acts B$. We say that the $k$-ary operation $f$ on $B$ is an \textbf{orbit\-/semiprojection} with respect to the $i$-th variable for $i\in\{1, \dots, k\}$ if there is a unary operation $g\in\overline{G}$ such that for any tuple $(a_1, \dots, a_k)$ where at least two of the $a_j$ lie in the same $G$-orbit, 
\[f(a_1, \dots, a_k)=g(a_i).\]
\end{definition}

Our classification of almost minimal operations above non-trivial permutation groups splits into three cases. The first, considered in Section~\ref{sec:threetypes}, is that of $G\acts B$ not being a Boolean group acting freely on $B$. In this case, we are able to show that neither quasi-Malcev operations nor quasi-majorities from Theorem~\ref{fivetypes} can be almost minimal above $\overline{\langle G\rangle}$, yielding the following:

\begin{customthm}{C}[Theorem~\ref{thm:threetypes}, Three Types Theorem] \label{thmc}
Let $G\acts B$ be such that $G$ is not a Boolean group acting freely on $B$. Let $s$ be the (possibly infinite) number of orbits of $G$ on $B$. Let $f$ be an almost minimal operation above $\overline{\langle G \rangle}$. Then, $f$ is of one of the following types:
\begin{enumerate}
    \item a unary operation;
    \item a binary operation;
    \item a $k$-ary orbit\-/semiprojection for $3\leq k \leq s$.
\end{enumerate}
\end{customthm}

In Proposition~\ref{prop:exproj} we note that for any $G\acts B$ (including Boolean groups acting freely), almost minimal $k$-ary orbit\-/semi-projections exist for all finite $2\leq k\leq s$.\\

We can also classify the almost minimal functions above the clone locally generated by a Boolean group acting freely on a set. In this context, freeness of $G\acts B$ yields that $\overline{\langle G\rangle}=\langle G\rangle$ (Remark~\ref{rem:free}), and hence we shall write only $\langle G\rangle$. For $|G|>2$, a new type of almost minimal function is possible, which we call a $G$-quasi-minority:

\begin{definition}\label{deforbitmin}
   Let $G\acts B$. We say that a ternary operation  $\mathfrak{m}$ is a $\bm{G}$\textbf{-quasi-minority} if  for all $\beta \in G$, 
\begin{equation}\label{orbitmindef}
  \mathfrak{m}(y, x, \beta x)\approx \mathfrak{m}(x, \beta x, y)\approx \mathfrak{m}(x, y, \beta x)\approx \mathfrak{m}(\beta y, \beta y, \beta y).  
\end{equation}
We say that a $G$-quasi-minority is a $\bm{G}$\textbf{-minority} if it is idempotent. 
\end{definition}

\begin{remark}
For a $G$-quasi-minority $\mathfrak{m}$ to make sense as a possible almost minimal function above $\overline{\langle G\rangle}$, we really need $G$ to be a Boolean group acting freely on $B$. Firstly, condition (\ref{orbitmindef}) implies that for all $\beta\in G$,
\[\mathfrak{m}(x,x,x)\approx \mathfrak{m}(\beta^{-1}x, x, \beta x)\approx \mathfrak{m}(\beta^2 x, \beta^2 x, \beta^2 x).\]
Since all elements in $\overline{\langle G\rangle}\cap\mathcal{O}^{(1)}$ are injective functions, if $\mathfrak m$ is almost minimal above $\overline{\langle G\rangle}$, we must have that the map  $x\mapsto \mathfrak{m}(x,x,x)$ (which is generated by $\mathfrak{m}$) is injective. This yields that $\beta^2=1$ for all $\beta\in G$, as otherwise $\beta^2$ would move some element by faithfulness; hence $G$ is Boolean. Secondly, if $\alpha\in G$ fixes a point $a\in B$, then for all $b\in B$, by condition (\ref{orbitmindef}),
\[\mathfrak{m}(b,b,b)=\mathfrak{m}(b,a, a)=\mathfrak{m}(b,a, \alpha a)=\mathfrak{m}(\alpha b, \alpha b, \alpha b).\]
By injectivity of $x\mapsto \mathfrak{m}(x,x,x)$, we must have that $\alpha b= b$ for all $b\in B$, which, since $G\acts B$ is faithful, implies that $\alpha$ is the identity in $G$. Hence, $G\acts B$ must be free.\\

Conversely, whenever $G\acts B$ is a Boolean group acting freely on $B$, then  almost minimal $G$-quasi-minorities always exist, as we show in Proposition~\ref{exmin}.
\end{remark}

For the case of $G\acts B$ being the free action of a Boolean group of size $>2$, we get the following classification of almost minimal operations:

\begin{customthm}{D}[Theorem~\ref{Boolcase}, Boolean case]\label{thmd} Let $G\acts B$ be a Boolean group acting freely on $B$ with $s$-many orbits (where $s$ is possibly infinite) and $|G|>2$. Let $f$ be an almost minimal operation above $\langle G \rangle$. Then,  $f$ is of one of the following types:
\begin{enumerate}
    \item a unary operation;
    \item a binary operation;
    \item a ternary $G$-quasi-minority;
    \item a $k$-ary orbit\-/semiprojection for $3\leq k \leq s$.
\end{enumerate} 
\end{customthm}

In the only remaining case of the group $\mathbb{Z}_2$ acting freely on $B$, there are two more types of almost minimal operations; these actually generate each other.

\begin{definition}\label{oddmajdef} Let $G\acts B$. Then, we say that $m(x,y,z)$ is an \textbf{odd majority} if $m$ is a quasi-majority such that for all $\gamma \in G\setminus\{1\}$,
\begin{equation}\label{oddmaj}
    m(y,x, \gamma x) \approx m(x,\gamma x, y)\approx m(x,y,\gamma x)\approx m(y,y,y).
\end{equation}

\end{definition}

\begin{definition}\label{oddmaldef} Let $G\acts B$. The ternary operation $M$ is an \textbf{odd Malcev} if $M$ is a quasi-Malcev such that for all $\gamma\in G\setminus\{1\}$,
\begin{align}
M(x,y,x) \approx M(x,x,x) \label{oddm1}\\
M(y, \gamma x, x)\approx M(x, \gamma x, y)\approx M(x,x,x) \label{oddm2}\\
    M(x,y,\gamma x)\approx M(\gamma y, \gamma y, \gamma y). \label{oddm3}
\end{align}
\end{definition}

\begin{customthm}{E}[Theorem~\ref{Z2case}, $\mathbb{Z}_2 $ case]\label{thme} Let $\mathbb{Z}_2\acts B$ freely with $s$-many orbits (where $s$ is possibly infinite). Let $f$ be an almost minimal operation above ${\langle \mathbb{Z}_2 \rangle}$. Then,  $f$ is of one of the following types:
\begin{enumerate}
    \item a unary operation;
    \item a ternary $G$-quasi-minority;
    \item an odd majority;
    \item an odd Malcev (up to permutation of variables);
    \item a $k$-ary orbit\-/semiprojection for $2\leq k \leq s$.
\end{enumerate} 
\end{customthm}
Odd majorities and odd Malcev operations above $\mathbb{Z}_2$ are in a sense dual of each other, since an odd majority $m(x,y,z)$ always generates an odd Malcev $m(x,\gamma y, z)$ for  $\gamma\in\mathbb{Z}_2\setminus \{0\}$ and vice versa. In Proposition~\ref{prop:optimalZ2} we prove that almost minimal odd majorities and odd Malcev operations always exist over $\langle \mathbb{Z}_2\rangle$.

\subsubsection{Back to minimality}

We can actually prove that there are no minimal odd majorities or odd Malcevs (Corollary~\ref{nominmaj}). This allows us to treat uniformly all Boolean groups acting freely on a set when classifying minimal operations above $\langle G\rangle$ (as opposed to almost minimal operations). Moreover, we are able to prove that minimal $G$-quasi-minorities exhibit considerable structure:

\begin{definition}\label{def:boolgtwisted} A $\bm{G}$\textbf{-invariant Boolean Steiner }$\bm{3}$\textbf{-quasigroup} is a symmetric ternary minority operation $\mathfrak{q}$ also satisfying the following conditions:
\begin{align}
    \mathfrak{q}(x, y, \mathfrak{q}(x,y,z))&\approx z\; ; \tag{SQS} \label{eq:sqslaw}\\
    \mathfrak{q}(x, y, \mathfrak{q}(z, y, w))&\approx \mathfrak{q}(x, z, w) \; ; \tag{Bool} \label{boolaw}\\
    \text{for all }\alpha, \beta, \gamma\in G, \mathfrak{q}(\alpha x, \beta y, \gamma z)&\approx \alpha\beta\gamma \mathfrak{q}(x, y, z) \;. \tag{Inv} \label{eq:twistid}
\end{align}
\end{definition}

As we discuss more in Section~\ref{sec:mintwisted}, idempotent symmetric minorities satisfying (\ref{eq:sqslaw}) are called Steiner $3$-quasigroups or Steiner Skeins in the literature and have been studied in universal algebra and design theory~\cite{quackenbush1975algebraic, armanious1980algebraische, lindner1978steiner}. When they further satisfy  (\ref{boolaw}), they are called Boolean Steiner $3$-quasigroups, since they are known to be of the form $\mathfrak{q}(x,y,z)=x+y+z$ in a Boolean group. The condition (\ref{eq:twistid}) is novel and adds further constraints on the behaviour of $\mathfrak{q}$: it must induce another Boolean Steiner $3$-quasigroup on the orbits which is related to $\mathfrak{q}$ in a particular way. Indeed, we can characterise $G$-invariant Boolean Steiner $3$-quasigroups in much detail, and we do so in Subsections~\ref{sub:chartwist} and~\ref{sub:counting}. In particular, in the latter subsection we give a detailed enough description of them that we are able to count their exact number for each free action of a Boolean group on a finite set (Corollary~\ref{cor:counting}). All of this yields the following theorem:

\begin{customthm}{F}[Theorem~\ref{Boolminclass}, minimal operations, Boolean case] \label{thmf} Let $G\acts B$ be a non-trivial Boolean group acting freely on $B$ with $s$-many orbits (where $s$ is possibly infinite). Let $f$ be a minimal operation above $\langle G \rangle$. Then,  $f$ is of one of the following types:
\begin{enumerate}
    \item a unary operation;
    \item a binary operation;
    \item a ternary minority of the form $\alpha \mathfrak{q}$ where $\mathfrak{q}$ is a $G$-invariant Boolean Steiner $3$-quasigroup and $\alpha\in G$;
    \item a $k$-ary orbit\-/semiprojection for $3\leq k \leq s$.
\end{enumerate} 
Furthermore, all $G$-invariant Boolean Steiner $3$-quasigroups on $B$ are minimal, and they exist if and only if $s=2^n$ for some $n\in\mathbb{N}$ or is infinite.
\end{customthm}

Theorems~\ref{thmc} and~\ref{thmf} together give us a classification of the types of minimal operations above (the clone locally generated by) any  non-trivial permutation group, with the latter theorem giving necessary and sufficient conditions for the existence of the minimal operation of the minority type. We would like to also have necessary and sufficient conditions for the existence of minimal orbit\-/semiprojections in a given arity $k$. Above the trivial group, P\'{a}lfy proved that on a domain of size $s$ there are minimal semiprojections for each size $2\leq k\leq s$~\cite{palfy1986arity}. A similar construction also yields necessary and sufficient conditions for a $k$-ary orbit\-/semiprojection to exist in the settings most relevant to us:

\begin{customthm}{G}(Theorem~\ref{thm:Palfy}, P\'{a}lfy's Theorem for orbit\-/semiprojections) \label{thmg} Let $G\acts B$ with $s$-many orbits be finite or oligomorphic. Then, for all finite $1<k\leq s$, there is a $k$-ary orbit\-/semiprojection minimal above $\overline{\langle G\rangle}$.
\end{customthm}

The various results mentioned in this section are then summarised in Theorem~\ref{maintheorem}:

\begin{proof}[Proof of Theorem~\ref{maintheorem}] Since minimal operations are almost minimal, Theorem~\ref{thmc} already tells us the best possible classification of minimal operations above $\overline{\langle G\rangle}$ when $G$ is not Boolean acting freely on $B$. The latter case is dealt with in Theorem~\ref{thmf}. Finally, Theorem~\ref{thmg} gives necessary and sufficient conditions for the existence of minimal orbit\-/semiprojections in a finite or oligomorphic context.
\end{proof}

\begin{remark}[Minimality above permutation groups in the lattice of all clones]\label{rem:oknoclosure} Our results also yield a classification of minimal and almost minimal clones above the clone $\langle G \rangle$ generated by a permutation group, when considering minimality in the lattice of all clones rather than that of closed clones. In this context, for a set of operations $\mathcal{S}$, we work with the notion of generation given by $\langle \mathcal{S} \rangle$, rather than $\overline{\langle \mathcal{S} \rangle}$ in the definitions of minimal clone (\ref{def:minimalc}), minimal operation (\ref{def:minimalop}), and almost minimal operation (\ref{def:almostmin}). With these (different) notions of minimality, Theorems~\ref{thmc},~\ref{thmd},~\ref{thme},~\ref{thmf} are all true when replacing $\overline{\langle G \rangle}$ by the clone $\langle \mathcal{T}\rangle$ generated by any transformation monoid $\mathcal{T}$ such that $G\subseteq \mathcal{T}\subseteq\overline{G}$. This can be seen by noting that at every step where we exclude the existence of a certain type of minimal (or almost minimal) operation, we do so by proving that if an operation $f$ of that type exists, then it generates (just by composition with itself and group elements and without local closure) another operation which does not locally generate it back. For Theorem~\ref{thmg}, we actually get the better result that, for all $G\acts B$ with $s$-many orbits (without further assumptions) and all finite $2\leq k\leq s$, there are minimal (in the lattice of all clones) $k$-ary orbit\-/semiprojections above $\langle G\rangle$ (Remark~\ref{rem:minimalityalways}). This yields a version of Theorem~\ref{maintheorem} for minimality above $\langle G\rangle$ in the lattice of all clones where the possible types of minimal operations are exactly those described in Theorem~\ref{maintheorem}, and with necessary and sufficient conditions for the existence of those of arity $>2$. Firstly, for $G\acts B$ a Boolean group acting freely on $B$, a $G$-invariant Boolean Steiner $3$-quasigroup is always minimal above $\langle G\rangle$ both in the interval of all clones and in the interval of closed clones above $\langle G\rangle$ (Lemma~\ref{lem:whatgislike}). Hence, such an operation exists and is minimal above $\langle G\rangle$ if and only if the number $s$ of $G$-orbits is either $2^n$ or infinite. Secondly, as mentioned above, there is a minimal (in the lattice of all clones) $k$-ary orbit\-/semiprojection above $\langle G\rangle$ if and only if $2\leq k\leq s$.
\end{remark}

\section{A three types theorem when \texorpdfstring{$G\acts B$}{G} is not the free action of a Boolean group}\label{sec:threetypes}
The main result of this section, proven in Subsection~\ref{sub:three}, is Theorem~\ref{thm:threetypes}, which classifies almost minimal operations above $\overline{\langle G\rangle}$ for $G$ not being a Boolean group acting freely on $B$. We use as a starting point Theorem~\ref{fivetypes} and explore in more detail the consequences of almost minimality for each of the types of operations of arity $>2$. We begin with Lemma~\ref{semiorb}, which, with no assumption on $G$, shows that almost minimal quasi-semiprojections  are orbit\-/semiprojections. Then, Lemma~\ref{notfree} and~\ref{notBool} show that we need $G\acts B$ to be the free action of a Boolean group in order for almost minimal quasi-Malcev operations to possibly exist. We will see this is a sufficient condition for them to exist in Section~\ref{sec:boolalmost}. Thus, we proceed by considering the case of a quasi-majority and prove that these can also only exist as almost minimal if $G$ is a Boolean group acting freely on $B$ in Lemma~\ref{nomaj}.  Indeed, later in Lemma~\ref{Bnomaj}, we show that almost minimal quasi-majorities only exist for the free action of $\mathbb{Z}_2$. Thus, Theorem~\ref{thm:threetypes} follows from putting together these results.\\

We then show in Subsection~\ref{sub:exminorb}  that $k$-ary orbit-semiprojections always exist as almost minimal above $\overline{\langle G\rangle}$ as long as $G\acts B$ is not transitive. Later in Section~\ref{sec:palfy}, we will give sufficient conditions for their existence as minimal operations above $\overline{\langle G\rangle}$.

\subsection{The three types theorem}\label{sub:three}

\begin{lemma}\label{semiorb} Let $G\acts B$. For $k\geq 3$, let $f$ be a $k$-ary quasi-semiprojection which is almost minimal above $\overline{\langle G \rangle}$. Then, $f$ is an orbit\-/semiprojection.
\end{lemma}
\begin{proof}
We may say without loss of generality that there is a unary operation $g$ such that whenever $|\{a_1, \dots, a_k\}|<k$, 
\[f(a_1, \dots, a_k)=g(a_1).\]
Let $\alpha\in G$. For all $i<j$, the function $h_{i,j, \alpha}$ defined by 
\[
(x_1,\ldots,x_k)\mapsto f(x_1,\ldots,x_{j-1},\alpha x_i,x_{j+1},\ldots,x_k)
\]
does not depend on $x_j$, hence it depends on only one variable by the almost minimality of $f$. Moreover,  setting for any $\ell\notin\{ 1,j\}$ the variable $x_\ell$  equal to $x_1$ and applying  $h_{i,j,\alpha}$ yields $g(x_1)$, by the property on $f$ assumed above. Hence, the variable $h_{i,j,\alpha}$ depends on is $x_1$, and 
\[
h_{i,j,\alpha}(x_1,\ldots,x_k)\approx g(x_1)\; .
\]
 Whenever $(a_1,\ldots,a_k)$ is any tuple where $a_i, a_j$ belong to the same orbit, then $a_j=\alpha a_i$ for some $\alpha \in G$. By the above, $f$ will return $g(a_1)$ on this tuple; hence, $f$ is an orbit\-/semiprojection onto $x_1$ as witnessed by $g$.
\end{proof}

\begin{lemma}\label{notfree} Let $G\acts B$ be such that the action of $G$ on $B$ is not free. Then, no almost minimal function over $\overline{\langle G \rangle}$ can be a quasi-Malcev operation. 
\end{lemma}
\begin{proof}
Since the action of $G$ on $B$ is not free, there is some non-trivial $\alpha\in G$ and distinct $a,b,c\in B$ such that $\alpha(a)=a, \alpha(b)=c$. Suppose by contradiction that $M(x,y,z)$ is almost minimal and quasi-Malcev, and consider $h(x,y)=M(x,\alpha x, y)$.
If $h(x,y)$ depends on the first argument, we have that
\[M(a,a,a)=h(a,a)=h(a,b)=M(a,a,b)=M(b,b,b)\;,\]
contradicting injectivity of $M(x,x,x)\in\overline{\langle G\rangle}$. 
Similarly, if $h(x,y)$ depends on the second argument,
\[M(c,c,c)=M(a,a,c)=h(a,c)=h(b,c)=M(b,c,c)=M(b,b,b)\; ,\]
again contradicting injectivity of $M(x,x,x)\in\overline{\langle G\rangle}$. Thus, $h(x,y)$ depends on both arguments, contradicting the almost minimality of $M$. 
\end{proof}

\begin{lemma}\label{notBool} Let $G\acts B$ with $G$ not Boolean. Then, no almost minimal function over $\overline{\langle G \rangle}$ can be a quasi-Malcev operation. 
\end{lemma}
\begin{proof} Let $\alpha\in G$ have order $\geq 3$. Since $\alpha^2\neq 1$, there is $b\in B$ such that $b, \alpha(b)=c,$ and $\alpha^2(b)=d$ are all distinct. Suppose that $M(x,y,z)$ is almost minimal and quasi-Malcev. Consider $h(x,y)=M(x,y,\alpha x)$. We have that 
\[h(c,d)=M(c,d,d)=M(c,c,c)=M(b,b,c)=h(b,b).\]
By almost minimality, $h\in\overline{\langle G\rangle}$, and so it equals a unary injective function of either its first or its second argument. However, since $b,c,d$ are distinct,  this yields a contradiction.
\end{proof}

    \begin{lemma}\label{majtoMal}
     Let $G\acts B$. Suppose that $m(x,y,z)$ is a quasi-majority operation which is almost minimal above $\overline{\langle G \rangle}$. Then, for any $\beta\in G\setminus \{1\}$, 
\[h(x,y,z):=m(x,\beta y, z)\]
is a quasi-Malcev operation such that
\begin{equation}\label{weirdid}
  h(x,x,x)\approx m(x,x,x).  
\end{equation}
\end{lemma}
\begin{proof} Note that $h$ is also almost minimal above $\overline{\langle G \rangle}$ since $m(x,y,z)=h(x,\beta^{-1} y, z)$. Hence, the binary function given by $l(x,y):=h(x,x,y)$ must be in $\overline{\langle G \rangle}$ and so essentially unary. Note that $l(x,y)$ cannot depend on the first variable. Otherwise,
\begin{align*}
    m(x,x,x) &\approx m(x, \beta x, x)\approx h(x,x,x) \approx l(x,x)\\ \approx l(x,\beta x)&\approx 
 h(x,x,\beta  x)\approx m(x, \beta  x, \beta x)\approx m(\beta x,\beta x,\beta x),
\end{align*}

which contradicts injectivity of $m(x,x,x)\in \overline{\langle G \rangle}$. Hence, $l(x,y)$ depends on the second variable yielding
\[h(x,x,y)\approx l(x,y)\approx l(y,y)\approx  
 h(y,y,y)\approx m(y,\beta y, y)\approx m(y,y,y).\]
From this equation we obtain that $h$ satisfies condition~(\ref{weirdid}). By symmetry, the same argument applies for $h(y,x,x)$, yielding that 
\[h(x,x,y)\approx h(y,y,y)\approx h(y,x,x),\]
and so that $h$ is quasi-Malcev.
\end{proof}

%\begin{remark}\label{alsol} Note that by symmetry, from Lemma~\ref{majtoMal}, if $m$ is a quasi-majority almost minimal above $\overline{\langle G\rangle}$ and $\beta\in G\setminus \{1\}, l(x,y,z)=m(x,z,\beta y)$ is also quasi-Malcev and such that $l(x,x,x)\approx m(x,x,x)$.
%\end{remark}

\begin{lemma}\label{nomaj} Let $G\acts B$. Suppose no almost minimal function over $\overline{\langle G \rangle}$ is quasi-Malcev. Then, no almost minimal function over $\overline{\langle G \rangle}$ can be a ternary quasi-majority operation. 
\end{lemma}
\begin{proof}
Suppose that $m$ is a ternary almost minimal quasi-majority above $\overline{\langle G \rangle}$. Then, for $\beta\in G\setminus \{1\}$, $m(x,\beta y, z)$ is quasi-Malcev and almost minimal above $\overline{\langle G \rangle}$ by Lemma~\ref{majtoMal}. Since we are assuming there are no quasi-Malcev operators almost minimal above $\overline{\langle G \rangle}$, this yields a contradiction.
\end{proof}

\begin{theorem}[Thereom~\ref{thmc}, Three types theorem]\label{thm:threetypes} Let $G\acts B$ be such that $G$ is not a Boolean group acting freely on $B$. Let $s$ be the (possibly infinite) number of orbits of $G$ on $B$. Let $f$ be an almost minimal operation above $\overline{\langle G \rangle}$. Then, $f$ is of one of the following types:
\begin{enumerate}
    \item a unary operation;
    \item a binary operation;
    \item a $k$-ary orbit\-/semiprojection for $3\leq k \leq s$.
\end{enumerate}
\end{theorem}
\begin{proof} We need to consider the five possibilities for $f$ from Rosenberg's five types theorem (Theorem~\ref{fivetypes}). Lemma~\ref{notBool} and Lemma~\ref{notfree} tell us that $f$ cannot be a quasi-Malcev. Moreover, $f$ cannot be a quasi-majority from Lemma~\ref{nomaj}. Finally, if $f$ is a quasi-semiprojection it must be an orbit\-/semiprojection from Lemma~\ref{semiorb}. Since $f$ is an orbit\-/semiprojection, its arity must be $\leq s$, where $s$ is the number of orbits of $G$.
\end{proof}

Comparing Theorem~\ref{thm:threetypes} to Theorem~\ref{fivetypes}, we excluded the existence of certain kinds of almost minimal operations. It is worth discussing whether and when the three types we obtain to can exist as almost minimal or minimal. Firstly, by definition, any unary operation is almost minimal above $\overline{\langle G\rangle}$, and any constant operation is minimal above it. For binary operations, the situation is more complex. If $G\acts B$ is not transitive or satisfies some strong forms of imprimitivity, there are binary almost minimal (and minimal) operations above $\overline{\langle G\rangle}$~\cite{palfy1982contributions}. However, if $G\acts B$ satisfies some strong forms of primitivity, there are no binary almost minimal operations above  $\overline{\langle G\rangle}$~\cite{palfy1982contributions, kearnes2001collapsing}. In the next subsection we show that almost minimal $k$-ary orbit\-/semiprojections always exist for all $2\leq k\leq s$, where $s$ is the number of $G$-orbits. In Section~\ref{sec:palfy}, we show this is the case also for minimal orbit\-/semiprojections in the cases that interest us.

\subsection{Orbit-semiprojections} \label{sub:exminorb}
In this subsection, we note that essential $k$-ary orbit\-/semiprojections exist and are almost minimal for all $2\leq k\leq s$. Indeed, it is easy to prove that essential $k$-ary orbit\-/semiprojections and various other operations introduced in Subsection~\ref{sub:main} are always almost minimal. We prove this in Lemma~\ref{lem:collapse}. From this observation, it also follows that there are almost minimal $k$-ary orbit\-/semiprojections for all $2\leq k\leq s$ (Proposition~\ref{prop:exproj}).

\begin{definition} Let $G\acts B$. A \textbf{weak orbit\-/semiprojection} is a $k$-ary operation $f$ for $k\geq 2$ such that for any distinct $ i<j\leq k$, there is some $s(i,j)\leq k$ and $\alpha_{i,j}\in \overline{G}$ such that whenever $(a_1, \dots, a_k)$ is a $k$-tuple where $a_i$ and $a_j$ are in the same $G$-orbit, then 
\[f(a_1, \dots, a_k)=\alpha_{i,j}(a_{s(i,j)}).\]
\end{definition}

\begin{remark}\label{allweak} Let $G\acts B$. Suppose that $f$ is an orbit\-/semiprojection, a $G$-quasi-minority, an odd majority, or an odd Malcev. Then, $f$ is a weak orbit\-/semiprojection. This can be seen from the definitions of these operations.
\end{remark}

Below we prove that essential weak orbit\-/semiprojections are almost minimal, thus showing that $G$-quasi-minorities, odd majorities, odd Malcev operations, and essential orbit\-/semiprojections are always almost minimal. 

\begin{lemma}\label{lem:collapse} Let $G\acts B$. Let $f:B^k\to B$ be an essential weak orbit\-/semiprojection. Then, $f$ is almost minimal above $\overline{\langle G\rangle}$.
\end{lemma}
\begin{proof} First, observe that since $f$ is  essential, it is not in $\overline{\langle G\rangle}$. Let $r< k$. We show by induction on terms that all $r$-ary functions in $\langle G \cup \{f\}\rangle$ are essentially unary and arise from a function in $\overline{G}$ by adding dummy variables; taking the topological closure then yields the lemma. For the base case, the elements of $G$ and the projections satisfy the claim. In the induction step, it is sufficient to consider $f(g_1,\dots ,g_k)$, where $g_1,\dots ,g_k$ are $r$-ary and satisfy our claim. Since all of the $g_i$ are essentially unary and $r<k$, two of them must depend on the same variable and so there are $1\leq i<j\leq k$ such that $g_i(x_1, \dots, x_{r})$ and $g_j(x_1, \dots, x_r)$ always lie in the same $G$-orbit. In particular, since $f$ is a $k$-ary weak orbit\-/semiprojection, there is $\alpha_{i,j}\in \overline{G}$ and $s(i,j)\leq k$ such that
  \[f(g_1, \dots, g_k)(x_1, \dots x_{r})=\alpha_{i,j}g_{s(i,j)}(x_1, \dots x_{r}),\]
Since by the inductive hypothesis, $g_{s(i,j)}$ is essentially unary and obtained from a function in $\overline G$ by adding dummy variables, so is $f(g_1, \dots, g_k)$. This completes the proof.
\end{proof}

\begin{prop}\label{prop:exproj} Let $G\acts B$, where $s$ is the number of orbits of $G$ on $B$. Then, for every $2\leq k\leq s$ there is a $k$-ary orbit\-/semiprojection which is almost minimal over $\overline{\langle G \rangle}$. 
\end{prop}
\begin{proof}
    We construct a $k$-ary quasi-semiprojection $f$ as follows. Let $(O_r|r<s)$ enumerate the $G$-orbits of $B$. Then, given $(a_1, \dots, a_k)\in B^k$, if any two of the $a_i$ are in the same orbit, $f(a_1, \dots, a_k)=a_1$. Otherwise, all of the $a_i$ belong to different orbits, and we let $f(a_1, \dots, a_k)=a_j$, where $a_j$ is the element appearing in the orbit with the largest index in our ordering. By construction, $f$ is a $k$-ary orbit\-/semiprojection which is essential. So almost minimality follows from Lemma~\ref{lem:collapse}.
\end{proof}

\begin{remark}\label{rem:binarysuff} From Theorem~\ref{thm:threetypes} and Proposition~\ref{prop:exproj}, we can easily recover the following fact mentioned in~\cite[Remark 7]{palfy1982contributions}: for $|B|>2$, a permutation group $G\acts B$ is collapsing (i.e., its monoidal interval consists of a single element) if and only if there is no binary essential operation in $\mathrm{St}(G)$, the maximal element in the monoidal interval of $G$.\footnote{As usual (cf. Remark~\ref{rem:oknoclosure}), this statement is true both working in the lattice of all clones on $B$ and in the lattice of closed clones on $B$. We give the proof below in the context of the lattice of closed clones to maintain a consistent notation.} We can prove this fact with our methods as follows: the left-to-right direction is trivial. Hence, we prove that if $G\acts B$ is not collapsing, then there is a binary essential operation in $\mathrm{St}(G)$. If $G\acts B$ is not transitive, then there is a binary orbit\-/semiprojection almost minimal over $\overline{\langle G \rangle}$ by Proposition~\ref{prop:exproj}. This will be contained in $\mathrm{St}(G)$. If $G\acts B$ is transitive, but not the free action of a Boolean group, Theorem~\ref{thm:threetypes} implies that $G$ is collapsing if and only if there is a binary almost minimal operation above $G$; and if there is a binary almost minimal operation above $G$, this is in $\mathrm{St}(G)$. The only remaining case is that of $G\acts B$ being the regular (i.e., transitive and free) action of a Boolean group. Whenever $G$ acts regularly on $B$ and $|B|>2$, it is easy to construct a binary almost minimal operation above $\langle G\rangle$ (in this case $\langle G\rangle=\overline{\langle G\rangle}$, cf. Remark~\ref{rem:free})~\cite[Proposition 6]{palfy1982contributions}: take any $\alpha\in G\setminus\{1\}$. Define $f_\alpha\in\mathcal{O}^{(2)}$ as follows:
\[f_\alpha(x,y)=\begin{cases}
    x &\text{if } \alpha x=y\; ;\\
    y &\text{otherwise}\;.
\end{cases}\]
Since $G\acts B$ is regular, it is easy to see $f_\alpha$ is almost minimal over $\langle G\rangle$, completing the proof. Interestingly, it is also true for a transformation monoid $\mathcal{T}$ on a finite set $|B|>2$ that $\mathcal{T}$ is collapsing  if and only if there is no binary essential operation in $\mathrm{St}(\mathcal{T})$, the maximal element in the monoidal interval of $\mathcal{T}$~\cite{grabowski1997binary}. We thank Andrei Krokhin for drawing our attention to~\cite{grabowski1997binary}.
\end{remark}

We end this section with the following proposition which  points out that an orbit\-/semiprojection can only generate other orbit\-/semiprojections in its same arity. 
This is important later in Section~\ref{sec:mintwisted}, since it implies that if an operation which is not an orbit\-/semiprojection generates an essential orbit\-/semiprojection, then it cannot be minimal (since it cannot be generated back by something it generates). Note, however,  that a $k$-ary orbit\-/semiprojection always generates operations which are not orbit\-/semiprojections in higher arity, for example, by adding dummy variables.

\begin{prop}\label{orbitgen} Let $G\acts B$ and suppose that $f$ is a $k$-ary orbit\-/semiprojection. Then, every function in $\overline{\langle G\cup \{f\}\rangle}\cap \mathcal{O}^{(k)}$ is an %a $k$-ary 
orbit\-/semiprojection.
\end{prop}
\begin{proof} Similarly to Lemma~\ref{lem:collapse}, we first show the proposition holds for ${\langle G\cup \{f\}\rangle}\cap \mathcal{O}^{(k)}$ by induction on terms. The base case is trivial (recalling that essentially unary operations are non-essential orbit\-/semiprojections). Suppose that $g_1, \dots, g_k$ are $k$-ary orbit\-/semiprojections (we count among these also essentially unary operations). Now, suppose that $(a_1, \dots, a_k)$ is a tuple with two elements in the same orbit. We know that each of $g_1, \dots, g_k$ depend on precisely one variable whenever this is the case. In particular, at least two of them will depend on the same variable, say $g_i$ and $g_j$ for $i<j$. Suppose that $f$ depends on the $l$th variable whenever two entries are in the same orbit and corresponds to $\alpha\in G$. Then, 
    \[f(g_1, \dots, g_k)(a_1, \dots, a_k)=\alpha g_l(a_1, \dots, a_k),\]
    and so $f(g_1, \dots, g_k)$ is also an orbit\-/semiprojection. In particular, by induction on terms, we get that all operations in ${\langle G\cup \{f\}\rangle}\cap \mathcal{O}^{(k)}$ are orbit\-/semiprojections. We can conclude this is the case for all operations in $\overline{\langle G\cup \{f\}\rangle}\cap \mathcal{O}^{(k)}$ since any operation in the local closure of a set of orbit\-/semiprojections is also an orbit\-/semiprojection.
\end{proof}

\section{Consequences of the three types theorem and oligomorphic clones}\label{consequences}
In this section, we explore  consequences of the three types Theorem~\ref{thm:threetypes} in an oligomorphic context, and more generally for clones whose unary operations correspond to $\overline{G}$ for $G\acts B$ non-trivial and not being the free action of a Boolean group on $B$.

\subsection{Essential polymorphisms in \texorpdfstring{$\omega$}{omega}-categorical structures}\label{sub:essentialpolys}
We start by looking  at Question~\ref{Bodq} of Bodirsky, which asks whether an $\omega$-categorical model-complete core with an essential polymorphism also has a binary essential polymorphism. Whilst we give a negative answer to it in Corollary~\ref{cor:counterex}, we prove in Corollary~\ref{niceres} that the answer is positive for $\omega$-categorical structures with at most two orbits, and thus in particular for the important class of transitive structures. In   Subsection~\ref{whyeasybin},  we then are going to find weaker assumptions that imply the existence of binary essential polymorphisms. %The results in this subsection are direct consequences of Theorem~\ref{thm:threetypes}.

\begin{definition}\label{def:coreclone}
   For $G\acts B$, we say that a closed function clone $\mathcal{C}$ is a \textbf{core clone} with respect to $G$ if it belongs to the monoidal interval of $\overline{G}$, that is $\mathcal{C}\cap\mathcal{O}^{(1)}=\overline{G}$. 
 We simply say that $\mathcal{C}$ is a core clone if it is a core clone with respect to some $G\acts B$.
\end{definition}

\begin{corollary}\label{cor:generalcor} Let $G\acts B$ not be the free action of a Boolean group and have $\leq 2$ orbits. Let $\mathcal{C}$ be an essential core clone with respect to $G$. Then,  $\mathcal{C}$ contains a binary essential polymorphism.
\end{corollary}
\begin{proof} By assumption, $\mathcal{C}$ is essential with $\mathcal{C}\cap\mathcal{O}^{(1)}=\overline{G}$. Hence, by Lemma~\ref{lem:existencealmost} it has an operation $f$ of arity $\geq 2$ and almost minimal above $\overline{\langle G\rangle}$. By the three types theorem (Theorem~\ref{thm:threetypes}), since $s\leq 2$, $f$ has to be binary. By almost minimality, $f$ is essential.
\end{proof}

\begin{remark} No oligomorphic permutation group $G\acts B$ %is a Boolean group acting 
acts freely, and hence, Theorem~\ref{thm:threetypes} applies to this context. To see this, %note that the action of an oligomorphic permutation group is never free. 
take any  element $a\in B$. The action of the stabilizer of $a$, $G_a$ on $B$ still has finitely many orbits by oligomorphicity. Hence, there are two elements $b,c\in B$ such that $(a,b)$ and $(a,c)$ lie in the same $G$-orbit. In particular, this means that there is a non-identity group element  fixing $a$.
\end{remark}

\begin{definition}\label{mccores}  We say that a countable $\omega$-categorical structure $B$ is a \textbf{model-complete core} if $\mathrm{Pol}(B)$ is a core clone.
\end{definition}

\begin{remark}
If a countable $\omega$-categorical structure $B$ is a model-complete core, then $\mathrm{Pol}(B)$ is a core clone with respect to $\mathrm{Aut}(B)$ (and with respect to any permutation group dense in $\mathrm{Aut}(B)$).
\end{remark}

\begin{remark}\label{rem:mincor} Let $B$ be a countable $\omega$-categorical structure with automorphism group $G$. For $f$ almost minimal above $\overline{\langle G \rangle}$, $\mathcal{C}=\overline{\langle G \cup \{f\}\rangle}$ is the polymorphism clone of a model-complete core with automorphism group $G$. This follows from Fact~\ref{fact:correspondence}.
%, c.f. Theorem 6.1.13~\cite{BodCSP}). 
%In particular, this is the case for $\mathcal{C}=\overline{\langle G \cup \{f\}\rangle}$ for $f$ almost minimal above $\overline{\langle G \rangle}$.
\end{remark}

\begin{corollary}\label{niceres} Let $B$ be an $\omega$-categorical countable model-complete core such that $\mathrm{Aut}(B)$ has $\leq 2$ orbits. Then, if $\mathrm{Pol}(B)$ has an essential polymorphism, it also has a binary essential polymorphism. 
\end{corollary}
\begin{proof} Follows from Corollary~\ref{cor:generalcor} and the definition of an $\omega$-categorical model-complete core.
\end{proof}

If $B$ is $\omega$-categorical in a finite relational language, the binary essential operation in Corollary~\ref{niceres} can also be chosen to be minimal due to the following fact:

\begin{fact}[Existence of minimal clones, oligomorphic case,~\cite{BodChen}]\label{existencemin} Let $B$ be a countable $\omega$-categorical structure in a finite relational language. Let $\mathcal{D}$ be a  closed clone properly containing $\overline{\langle \mathrm{Aut}(B)\rangle}$. Then, $\mathcal{D}$ contains a closed clone $\mathcal{C}$ which is minimal above $\overline{\langle \mathrm{Aut}(B)\rangle}$.
\end{fact}

\begin{corollary}\label{cor:counterex} Let $B$ be a countable $\omega$-categorical structure such that $\mathrm{Aut}(B)$ has $s\geq 3$ orbits on $B$. Then, for each $3\leq k\leq s$ there is a model-complete core with the same automorphism group as $B$ and such that it has a $k$-ary essential polymorphism and no essential polymorphisms of lower arity.
\end{corollary}
\begin{proof} By Proposition~\ref{prop:exproj}, for each $3\leq k\leq s$, there is a closed clone $\mathcal{C}$ above $\overline{\langle\mathrm{Aut}(B)\rangle}$ generated by an essential (by almost minimality) $k$-ary orbit\-/semiprojection and $\overline{\langle\mathrm{Aut}(B)\rangle}$. By Remark~\ref{rem:mincor}, $\mathcal{C}$ is the polymorphism clone of a model-complete core with automorphism group $\mathrm{Aut}(B)$.
\end{proof}

Corollary~\ref{cor:counterex} gives a negative answer to 
Question~\ref{Bodq}. However, Corollary~\ref{niceres} yields that Question~\ref{Bodq} has a positive answer whenever $B$ is an $\omega$-categorical model-complete core whose automorphism group has at most two orbits. Moreover we will now see  in Theorem~\ref{thm:findbin} of Section~\ref{whyeasybin} that for $B$ an $\omega$-categorical model-complete core whose CSP is not NP-hard we can always find a binary essential operation in $\mathrm{Pol}(B)$; hence, for what concerns applications to CSPs of $\omega$-categorical structures, one could say the answer to Question~\ref{Bodq} is positive as well.

\subsection{Why easy problems lie above binary polymorphisms}\label{whyeasybin}
In this subsection, we show Theorem~\ref{thm:findbin}, which states that as long as $\mathcal{C}$ is a core clone with respect to a suitable permutation group $G\acts B$ without a uniformly continuous clone homomorphism to the clone $\mathcal{P}_{\{0,1\}}$ of projections on a two-element set (in the sense of Definition~\ref{def:clone}), it will have a binary essential polymorphism. This is especially helpful for the study of $\mathrm{CSP}$s of $\omega$-categorical structures since $\mathrm{Pol}(B)$ having such a clone homomorphism implies hardness of the CSP~\cite{Topo-Birk}. Moreover, it is an interesting disanalogy with the case of idempotent clones (i.e. core clones above the trivial group), where on any domain there are clones with no essential binary operations but no clone homomorphism to $\mathcal{P}_{\{0,1\}}$  (e.g. any clone generated by a  ternary majority or minority operation).\\

We begin with Lemma~\ref{thm:trivialidentities}, proving that the closed clone generated by $G$ and its orbit semiprojections has a uniformly continuous clone homomorphism to $\mathcal{P}_{\{0,1\}}$. %We remark  that for a clone homomorphism $\eta:\mathcal{C}\to\mathcal{D}$ where the domain $D$ of $\mathcal{D}$ is finite, uniform continuity of $\eta$ is equivalent to there being some finite $A\subseteq C$ such that $f_{\upharpoonright A}=g_{\upharpoonright A}$ implies $\eta(f)=\eta(g)$; see~\cite{uniformbirkhoff} for more information on the uniform structure considered in this context.

\begin{definition}\label{def:homomcont} Let $G\acts B$. Let $\mathcal{S}$ be the closed clone generated by all orbit\-/semiprojections:
\[\mathcal{S}:=\overline{\langle G\cup\{ f\; \vert\; f \text{ is an orbit\-/semiprojection for }G\acts B\}\rangle}.\]
%Recall that  $\mathcal{P}_{\{0,1\}}$ is the clone of projections on a two-element set. 
\end{definition}

%We can prove that the clone $\mathcal{S}$ of orbit\-/semiprojections always satisfies only trivial identities. 
%Indeed, $\mathcal{S}$ has a uniformly continuous homomorphism to $\mathcal{P}_{\{0,1\}}$ 

\begin{lemma}\label{thm:trivialidentities} Let $G\acts B$. There is a uniformly continuous clone homomorphism from %the clone 
 $\mathcal{S}$ 
 %generated by orbit\-/semiprojections and $G$ 
  to $\mathcal{P}_{\{0,1\}}$.%, $\xi: \mathcal{S}\to \mathcal{P}_{\{0,1\}}$.
\end{lemma}
\begin{proof} Since $G\acts B$ is non-trivial, let $C\subseteq B$ be a non-trivial $G$-orbit (i.e. $|C|>1$). Clearly, the map $\rho:\mathcal{S}\to\mathcal{S}_{\upharpoonright C}$ sending each operation $f$ to its restriction $f_{\upharpoonright C}$ to $C$  is a clone homomorphism since each identity satisfied by operations in $\mathcal{S}$ on $B$ will also be satisfied on a restriction of the domain. Note next that any such  restriction $f_{\upharpoonright C}$  is essentially unary: this is clear for orbit\-/semiprojections and elements of $G$, and follows by an easy induction on terms for arbitrary operations in $\mathcal{S}$. Let $\tau:\mathcal{S}_{\upharpoonright C} \to\mathcal{P}_{\{0,1\}}$ send each $k$-ary operation $f_{\upharpoonright C}$ in $\mathcal{S}_{\upharpoonright C}$ to the $k$-ary projection to the $i$th coordinate $\pi_i^k$, where $i$ is the variable on which $f$ depends. This is again easily seen to be a clone homomorphism. Thus, $\xi:=\tau\circ\rho$ is a clone homomorphism $\xi:  \mathcal{S}\to \mathcal{P}_{\{0,1\}}$, which moreover is uniformly continuous: Let $B'=\{c,d\}$ for distinct $c,d\subseteq C$. For any $f,g\in\mathcal{S}$, $f_{\upharpoonright B'}=g_{\upharpoonright B'}$ implies that $f$ and $g$ depend on the same variable, and whence  $\xi(f)=\xi(g)$, yielding uniform continuity.
\end{proof}

\begin{definition} Let $\mathcal{C}$ be a function clone. Let $\xi:\mathcal{C}\cap \mathcal{O}^{\leq (3)}\to \mathcal{P}_{\{0,1\}}$ be a map preserving arities. The \textbf{minor extension of} $\xi$ is the map $\xi':\mathcal{C}\to \mathcal{P}_{\{0,1\}}$ defined as follows:\\

Let $f\in\mathcal{C}$ be an $n$-ary operation. Let $(a_1, \dots, a_n):=a\in\{0,1\}^n$. Write $f_a(x,y)$ for the binary operation induced by $f$ substituting the variable $x_i$ with $x$ if $a_i=0$ and with $y$ otherwise. We then define 
\[\xi'(f)(a):=\xi(f_a)(0,1).\]
\end{definition}

From~\cite[Proposition 6.8]{irrelevant} we know the following: 
\begin{fact}\label{ternarytheom} Let $\mathcal{C}$ be a  function clone. Suppose that $\xi:\mathcal{C}\cap \mathcal{O}^{\leq(3)}\to \mathcal{P}_{\{0,1\}}$ is a partial clone homomorphism (i.e., it preserves arities and identities). Then, the minor extension $\xi':\mathcal{C}\to \mathcal{P}_{\{0,1\}}$ is a clone homomorphism.
\end{fact}

From this it follows, in particular, that  if a clone exhibits  any structure at all in the sense that it has no clone homomorphism to $\mathcal{P}_{\{0,1\}}$ (which is equivalent to the satisfaction of some non-trivial set of identities in the clone), then it contains an essential operation of arity $\leq 3$. As a consequence of our results, we now can prove that  the same assumption (in fact, a weakening thereof considering uniform continuous clone homomorphisms only) implies even the existence of a binary essential operation, provided the clone is a core clone with respect to a group which is not a Boolean group acting freely:

\begin{theorem}[Theorem~\ref{mainthmb}]\label{thm:findbin}
 Let $G\acts B$ be such that $G$ is not a Boolean group acting freely on $B$. Suppose that $\mathcal{C}$ is a core clone with respect to $G$, and that $\mathcal{C}$ has no uniformly continuous clone homomorphism to $\mathcal{P}_{\{0,1\}}$, the clone of projections on a two-element set. Then, $\mathcal{C}$ contains a binary essential operation almost minimal above $\overline{\langle G\rangle}$. 
\end{theorem}
\begin{proof} Note that if $f\in \mathcal{C}\cap\mathcal{O}^{(2)}\supsetneq \overline{\langle G\rangle}\cap\mathcal{O}^{(2)}$, since $\mathcal{C}$ is a core clone with respect to $G$, the operation $f$ must be essential and almost minimal, and the desired conclusion follows. Hence, suppose that $\mathcal{C}\cap\mathcal{O}^{(2)}=\overline{\langle G\rangle}\cap\mathcal{O}^{(2)}$. Then, all ternary operations in $\mathcal{C}\setminus \overline{\langle G\rangle}$ are  almost minimal. In particular, $\mathcal{C}\cap\mathcal{O}^{(3)}$ consists entirely of essentially unary operations and orbit\-/semiprojections.\\

From Lemma~\ref{thm:trivialidentities}, there is a uniformly continuous clone homomorphism from $\mathcal{S}$ to $\mathcal{P}_{\{0,1\}}$ which, when restricted to $\mathcal{C}\cap\mathcal{O}^{\leq (3)}$ yields a map $\xi:\mathcal{C}\cap\mathcal{O}^{\leq (3)} \to\mathcal{P}_{\{0,1\}}$ preserving arities and identities. Hence, from Fact~\ref{ternarytheom}, the minor extension of $\xi$ yields a clone homomorphism $\xi':\mathcal{C}\to\mathcal{P}_{\{0,1\}}$. It is easy to see this is uniformly continuous by observing that given $B'$ a two-element set from a $G$-orbit, if $f$ and $g$ agree on $B'$, then they are sent to the same projection.  
\end{proof}

\begin{remark} Theorem~\ref{thm:findbin} is false whenever its assumption on $G$ fails. % That is, if $G\acts B$ is a Boolean group acting freely on $B$ (including if $G$ is the trivial group), there is a core clone $\mathcal{C}$ with respect to $G$ with no clone homomorphism to $\mathcal{P}_{\{0,1\}}$. In fact, we 
 In fact, will see that if $G\acts B$ is a Boolean group acting freely on $B$ (including if $G$ is the trivial group), then there is always a $G$-quasi-minority $m$ almost minimal above $\overline{\langle G\rangle}$ (Proposition~\ref{exmin}). Hence, $\overline{\langle G\cup\{m\}\rangle}$ is a core clone with respect to $G$ with no binary essential operation (by almost minimality of $m$), but with no clone homomorphism to $\mathcal{P}_{\{0,1\}}$, since $m$ satisfies a set of non-trivial identities. 
 %which cannot be satisfied by a projection.
\end{remark}

Since a uniformly continuous  clone homomorphism to $\mathcal{P}_{\{0,1\}}$  implies NP-hardness in a finite or $\omega$-categorical setting~\cite[Theorem 28]{Topo-Birk}, we can deduce that if $\mathrm{CSP}(B)$ is not NP-hard, there must be some binary essential polymorphism $f\in\mathrm{Pol}(B)$ witnessing this:

\begin{corollary} Let $B$ be a finite or countable $\omega$-categorical  model-complete core. For finite $B$, suppose additionally  that   $\mathrm{Aut}(B)\acts B$ is not the free action of a Boolean group.  Then $\mathrm{CSP}(B)$ is NP-hard   or  $\mathrm{Pol}(B)$ contains a binary essential polymorphism. 
\end{corollary}

%Obviously, if P=NP, the above statement is vacuously true.

\subsection{The orbital extension property and its failure}\label{sub:OEP}

Finding binary essential operations given some essential operation in $\mathrm{Pol}(B)$ for $B$ $\omega$-categorical is useful for several arguments in the context of infinite-domain CSPs, meaning that previous authors already developed methods to do this~\cite{BodKara, Smoothapp, bodirsky2020complexity, minrandom}. The standard technique made use of the orbital extension property.

\begin{definition}\label{def:OEP} We say that $G\acts B$ has the \textbf{orbital extension property} if there is an orbital (i.e. an orbit of the componentwise action $G\acts B^2$) $O$ such that for any $u,v\in B$ there is $z\in B$ such that $(u,z), (v,z)\in O$. 
\end{definition}

In the following, we shall also say that a structure has the orbital extension property if its automorphism group does, and shall proceed likewise with other group-theoretic properties (in particular, transitivity and primitivity). 
Note that having the orbital extension property implies transitivity of $G\acts B$.  A lot of transitive $\omega$-categorical structures, such as the order of the rational numbers or the random graph, have the orbital extension property; see Remark~\ref{rem:examples:oep} below for further examples. 
An example of a transitive $\omega$-categorical structure which does not have the orbital extension property is $K_{\omega, \omega}$, the complete bipartite graph where both parts of the partition are countable. To see this, take $a,b\in K_{\omega, \omega}$ forming an edge (i.e. they are in different blocks of the bipartition). There are two orbitals for pairs of distinct vertices, one, $O_E$ for pairs forming an edge, and one, $O_{\neg E}$ for pairs not forming an edge. However, neither $O_E$ nor  $O_{\neg E}$ can witness the orbital extension property with respect to $(a,b)$ since, due to there being only two sides in the partition, forming an edge with respect to one of $a$ or $b$ implies not forming an edge with the other vertex, and vice-versa. Note, however, that for $k>2$ the  complete $k$-partite graph with every block in the partition being infinite does have the orbital extension property.\\

The following was first observed in a weaker form in~\cite[Lemma 5]{BodKara}, and then strengthened in~\cite[Lemma 40]{minrandom} and~\cite[Proposition 22]{Smoothapp}:

\begin{fact}\label{fact:OEPfact} Let $\mathcal{C}$ be a closed clone with an essential operation and containing a permutation group $G\acts B$ with the orbital extension property. Then, $\mathcal{C}$ contains a binary essential operation.
\end{fact}

Note that for $K_{\omega, \omega}$, the aforementioned example of a transitive structure without the orbital extension property, there is an equivalence relation (given by the two blocks of $K_{\omega, \omega}$) which is invariant under the automorphisms of the structure. Primitivity, i.e.~the absence of an equivalence relation invariant under all automorphisms, clearly implies transitivity, as does the orbital extension property.  It was hitherto not clear whether for $\omega$-categorical structures, primitivity also   implies the orbital extension property:

\begin{customq}{\ref{qOEP}}(Question 14.2.1 (2) in {\cite{BodCSP}})
   Does every primitive oligomorphic permutation group have the orbital extension property?
\end{customq}

In this subsection, we give examples of  primitive $\omega$-categorical structures which fail the orbital extension property. This answers Question~\ref{qOEP}, and shows how Theorem~\ref{thm:findbin} improves on previous methods.  Namely, on one hand,  Theorem~\ref{thm:findbin} makes  Fact~\ref{fact:OEPfact} redundant and covers also non-primitive and even non-transitive structures:  as long as $B$ is an $\omega$-categorical core with $\mathrm{Aut}(B)$ having less than two orbits, if $\mathrm{Pol}(B)$ is essential then it must contain a binary essential operation. Indeed, Theorem~\ref{thm:findbin} even tells us that this will happen as long as $B$ is a model-complete core and $\mathrm{Pol}(B)$ does not have a uniformly continuous clone homomorphism to $\mathcal{P}_{\{0,1\}}$, which for the purposes of studying CSPs is everything one needs. On the other hand, our counterexample shows  that techniques using the orbital extension property cannot be generalised to all primitive $\omega$-categorical model-complete cores, unlike Theorem~\ref{thm:findbin}. We remark that our counterexample even falls within the scope of the infinite-domain CSP conjecture of Bodirsky and Pinsker referred to in the introduction.\\

In order to define our counterexample, we begin by reminding the reader of some basic definitions regarding homogeneous structures. We refer to~\cite{homogeneous} for a survey on the topic, and to~\cite{BodCSP, infinitesheep} for surveys on CSPs of homogeneous structures. 

\begin{definition}
A countable relational structure is \textbf{homogeneous} if every isomorphism between its finite substructures extends to an automorphism of the whole structure. Homogeneous structures in a finite relational language are $\omega$-categorical, and so their automorphism group is oligomophic. For $B$ a relational structure, its \textbf{age}, $\bm{\mathrm{Age}(B)}$, is the class of finite substructures of $B$. The ages of homogeneous structures correspond to classes of finite structures known as Fra\"{i}ss\'{e} classes which satisfy some combinatorial closure properties. In particular, given a Fra\"{i}ss\'{e} class $\mathcal{F}$, there is a (unique up to isomorphism) countable homogeneous structure $B$ whose age is $\mathcal{F}$, which we call its \textbf{Fra\"{i}ss\'{e} limit}. For a relational language $\mathcal{L}$ and a set of finite $\mathcal{L}$-structures $\mathcal{D}$, $\bm{\mathrm{Forb}^{\mathrm{emb}}(\mathcal{D})}$ denotes the class of finite $\mathcal{L}$-structures such that no structure in $\mathcal{D}$ embeds in them. A homogeneous $\mathcal{L}$-structure $B$ is \textbf{finitely bounded} if there is some finite set of $\mathcal{L}$-structures $\mathcal{D}$ such that $\mathrm{Age}(B)=\mathrm{Forb}^{\mathrm{emb}}(\mathcal{D})$. Note that $\mathcal{D}$ can be chosen to be \textbf{minimal} in the sense that for all $A\in\mathcal{D}$ no proper substructure of $A$ is in $\mathcal{D}$.
\end{definition}

\begin{remark}\label{rem:examples:oep} With these notions at hand, we can provide many more examples of $\omega$-categorical structures with the  orbital extension property:
\begin{itemize}
   \item any transitive homogeneous structure with free amalgamation\footnote{A more general statement for readers familiar with model theory: any transitive $\omega$-categorical structure $M$ with an invariant $1$-type over $M$ has the orbital extension property. By a $1$-type over $M$, we mean a maximally consistent set of formulas with parameters from $M$ in a single free variable, and we denote by $S_1(M)$ the space of all such types. The automorphism group $\mathrm{Aut}(M)$ has a natural action on $S_1(M)$ and we say that $p\in S_1(M)$ is invariant if it is fixed by this action. For $p\in S_1(M)$ invariant and $a\in M$, take the restriction $p_{\upharpoonright\{a\}}$ of $p$ to $a$. By $\omega$-categoricity, types over finitely many parameters are realised, and so we may take $b\in M$ realising $p_{\upharpoonright\{a\}}$. Take $O$ to be the orbital of the pair $(a,b)$. Now, given any two elements $c,d$, taking $e$ realising $p_{\upharpoonright\{c,d\}}$, by $\mathrm{Aut}(M)$-invariance of $p$, we get that $(c,e), (d,e)\in O$, meaning that $O$ witnesses the orbital extension property for $M$. All transitive homogeneous structures with free amalgamation have an invariant $1$-type (consisting of an element not related to any element in the model). Moreover, $(\mathbb{Q}, <)$ has two invariant $1$-types (one asserting $x$ is smaller than all elements in $\mathbb{Q}$, and the other asserting that it is larger).}~\cite[Example 20]{Smoothapp};
   \item any transitive homogeneous structure $B$ where the set of minimal bounds for $\mathrm{Age}(B)$ has no structures of size $3$~\cite[Example 20]{Smoothapp};
   \item any $2$-set-transitive $\omega$-categorical structure~\cite[Lemma 3.7]{bodirsky2020complexity}.
\end{itemize}
\end{remark}

%As mentioned, for $A$ and $B$ countable structures such that  $B$ is $\omega$-categorical, $\mathrm{Aut}(B)\leq\mathrm{Aut}(A)$, and $\mathrm{Aut}(B)$ has the orbital extension property, one can prove that  if $\mathrm{Pol}(A)$ is not essentially unary, it must contain a binary essential operation. 

The structures which will yield a negative answer to Question~\ref{qOEP} are homogeneous $2$\-/multitournaments:

\begin{definition} A $\bm{2}$-\textbf{multitournament} is a relational structure $\mathcal{M}$ in a language with two binary relations $\left\{\xrightarrow[]{1}, \xrightarrow[]{2}\right\}$ such that for any two elements $a,b\in M,$ exactly one of \[a\xrightarrow[]{1}b, b\xrightarrow[]{1}a, a\xrightarrow[]{2}b, \text{ or } b\xrightarrow[]{2}a\]
holds. We can think of a $2$-multitournament as a tournament where every arc is coloured of one of two colours. 
\end{definition}

Cherlin~\cite[Table 18.1]{Cherlinclass} recently classified the family of primitive $3$-constrained homogeneous $2$-multitournaments (i.e.~primitive homogeneous $2$\-/multitournaments whose age is of the form $\mathrm{Forb}^{\mathrm{emb}}(\mathcal{D})$, where every structure in $\mathcal{D}$ has size $3$). We shall see that two such multitournaments do not have the orbital extension property. In order to introduce them we need to set up some notation to denote which substructures we are omitting from the Fra\"{i}ss\'{e} classes of $2$-multitournaments yielding our desired counterexamples.

\begin{notation} For $i,j,k\in\{1,2\}$ we let  $C(i,j,k)$ denote the $2$-multitournament of size $3$ consisting of an oriented $3$-cycle given by three vertices $a,b,c$ such that 
\[a\xrightarrow[]{i}b, b\xrightarrow[]{j}c, \text{ and }c\xrightarrow[]{k}a.\]
Similarly, $L(i,j,k)$ denotes the $2$-multitournament of size $3$ where the three vertices $a,b,c$ are such that 
\[a\xrightarrow[]{i}b, b\xrightarrow[]{j}c, \text{ and }a\xrightarrow[]{k}c.\]
When drawing these graphs, such as in Figure~\ref{triangles}, we use the colour \textcolor{Blue}{blue} to denote the relation $\xrightarrow[]{1}$ and the colour \textcolor{orange}{orange} to denote the relation $\xrightarrow[]{2}$.
\end{notation}

\newcommand{\midarrow}{\tikz \draw[-triangle 90] (0,0) -- +(.1,0);}

\begin{figure}
    \begin{center}
   \begin{minipage}{0.4\textwidth}
\centering
         \begin{tikzpicture}
    \begin{scope}[very thick, every node/.style={sloped,allow upside down}]
        \draw[very thick, Blue] (2,0)--node{\midarrow} (0,0);
        \draw[very thick, Blue] (1,1.5)--node{\midarrow} (2,0);
        \draw[very thick, orange] (0,0)--node{\midarrow} (1,1.5);
      
\end{scope}

\filldraw (0,0) circle (2pt) node[anchor=north] {$a$};
\filldraw (1,1.5) circle (2pt) node[anchor=south] {$b$};
\filldraw (2,0) circle (2pt) node[anchor=north] {$c$};
 \node[below] at (1,-0.1) {$1$};
 \node[left] at (0.5,0.9) {$2$};
 \node[right] at (1.5,0.9) {$1$};
  \node[below] at (1,-0.5) {$C(2,1,1)$};
 
    \end{tikzpicture}
    \end{minipage}
    \begin{minipage}{0.4\textwidth}
\centering

         \begin{tikzpicture}
    \begin{scope}[very thick, every node/.style={sloped,allow upside down}]
        \draw[very thick, orange] (0,0)--node{\midarrow} (2,0);
        \draw[very thick, Blue] (1,1.5)--node{\midarrow} (2,0);
        \draw[very thick, orange] (0,0)--node{\midarrow} (1,1.5);

\end{scope}

\filldraw (0,0) circle (2pt) node[anchor=north] {$a$};
\filldraw (1,1.5) circle (2pt) node[anchor=south] {$b$};
\filldraw (2,0) circle (2pt) node[anchor=north] {$c$};
 \node[below] at (1,-0.1) {$2$};
 \node[left] at (0.5,0.9) {$2$};
 \node[right] at (1.5,0.9) {$1$};

 \node[below] at (1,-0.5) {$L(2,1,2)$};
 
    \end{tikzpicture}
\end{minipage}
    \end{center}
    \caption[]{An example of the $2$-multitournaments given by the relations $C(2,1,1)$ and $L(2,1,2)$. To make the illustration clearer, we use the colour \textcolor{Blue}{blue} to denote the relation $\xrightarrow[]{1}$ and the colour \textcolor{orange}{orange} to denote the relation $\xrightarrow[]{2}$. These would appear as dark-grey and light-grey respectively in black \& white print.}
\label{triangles}
\end{figure}

\begin{definition}
    The homogeneous $2$-multitournaments $\widetilde{\mathbb{S}(3)}$ and $\mathbb{S}(4)$ are given by the Fra\"{i}ss\'e limit of the classes of $2$\-/multitournaments omitting embeddings of the $2$\-/multitournaments on three vertices described by Table~\ref{S3S4}.

\setlength{\tabcolsep}{18pt}
\renewcommand{\arraystretch}{2}
\begin{table}
    \begin{center}
\begin{tabular}{|c ||c | c||} 
 \hline
 \textbf{Structure} & omitted instances of $C(i,j,k)$ & omitted instances of $L(i,j,k)$\\  
 \hline \hline 
 $\widetilde{\mathbb{S}(3)}$ & 111, 222 & 111, 122, 212 \\
 \hline
 $\mathbb{S}(4)$ & 111, 112 & 121, 211, 221, 222 \\
  \hline

\end{tabular}
\end{center}
  \caption{In the row of a given structure $S$ we list the $2$-multitournaments of size $3$ such that the class of $2$-multitournaments omitting them yields $S$ as a Fra\"{i}ss\'{e} limit. By having "111" written in the "omitted instances of $C(i,j,k)$" column on the $\widetilde{\mathbb{S}(3)}$-row, we mean that the $2$-multitournament of size $3$ given by $C(1,1,1)$ is omitted from the class of finite $2$-multitournaments yielding $\widetilde{\mathbb{S}(3)}$ as a Fra\"{i}ss\'{e} limit. In particular, in Figure~\ref{S3} we represent the omitted triangles in $\widetilde{\mathbb{S}(3)}$ of the form $L(i,j,k)$.}
  \label{S3S4}
\end{table}

\end{definition}

\begin{figure}
    \begin{center}
 \begin{minipage}{0.3\textwidth}
\centering

         \begin{tikzpicture}
    \begin{scope}[very thick, every node/.style={sloped,allow upside down}]
        \draw[very thick, Blue] (0,0)--node{\midarrow} (2,0);
        \draw[very thick, Blue] (1,1.5)--node{\midarrow} (2,0);
        \draw[very thick, Blue] (0,0)--node{\midarrow} (1,1.5);
       
\end{scope}

\filldraw (0,0) circle (2pt) node[anchor=north] {$a$};
\filldraw (1,1.5) circle (2pt) node[anchor=south] {$b$};
\filldraw (2,0) circle (2pt) node[anchor=north] {$c$};
 \node[below] at (1,-0.1) {$1$};
 \node[left] at (0.5,0.9) {$1$};
 \node[right] at (1.5,0.9) {$1$};

 \node[below] at (1,-0.5) {$L(1,1,1)$};
 
    \end{tikzpicture}
\end{minipage}
 \begin{minipage}{0.3\textwidth}
\centering

         \begin{tikzpicture}
    \begin{scope}[very thick, every node/.style={sloped,allow upside down}]
        \draw[very thick, orange] (0,0)--node{\midarrow} (2,0);
        \draw[very thick, orange] (1,1.5)--node{\midarrow} (2,0);
        \draw[very thick, Blue] (0,0)--node{\midarrow} (1,1.5);

\end{scope}

\filldraw (0,0) circle (2pt) node[anchor=north] {$a$};
\filldraw (1,1.5) circle (2pt) node[anchor=south] {$b$};
\filldraw (2,0) circle (2pt) node[anchor=north] {$c$};
 \node[below] at (1,-0.1) {$2$};
 \node[left] at (0.5,0.9) {$1$};
 \node[right] at (1.5,0.9) {$2$};

 \node[below] at (1,-0.5) {$L(1,2,2)$};
 
    \end{tikzpicture}
\end{minipage}
    \begin{minipage}{0.3\textwidth}
\centering

         \begin{tikzpicture}
    \begin{scope}[very thick, every node/.style={sloped,allow upside down}]
        \draw[very thick, orange] (0,0)--node{\midarrow} (2,0);
        \draw[very thick, Blue] (1,1.5)--node{\midarrow} (2,0);
        \draw[very thick, orange] (0,0)--node{\midarrow} (1,1.5);

\end{scope}

\filldraw (0,0) circle (2pt) node[anchor=north] {$a$};
\filldraw (1,1.5) circle (2pt) node[anchor=south] {$b$};
\filldraw (2,0) circle (2pt) node[anchor=north] {$c$};
 \node[below] at (1,-0.1) {$2$};
 \node[left] at (0.5,0.9) {$2$};
 \node[right] at (1.5,0.9) {$1$};

 \node[below] at (1,-0.5) {$L(2,1,2)$};
 
    \end{tikzpicture}
\end{minipage}
    \end{center}
    \caption[]{pictorial representation of the $L(i,j,k)$-omitted triangles in $\widetilde{\mathbb{S}(3)}$.}
\label{S3}
\end{figure}

\begin{prop}\label{prop:noOEP} The automorphism groups of the homogeneous $2$-multitournaments  $\widetilde{\mathbb{S}(3)}$ and  $\mathbb{S}(4)$ do not have the orbital extension property.
\end{prop}
\begin{proof} We run the proof for $\widetilde{\mathbb{S}(3)}$ since the proof for $\mathbb{S}(4)$ is virtually identical. There are four possible orbits of pairs for which the orbital extension property could hold. These orbitals of $\widetilde{\mathbb{S}(3)}$ are described by the formulas:
\[x\xrightarrow[]{1}y, y\xrightarrow[]{1}x, x\xrightarrow[]{2}y, \text{ and } y\xrightarrow[]{2}x.\]
Suppose the orbital extension property holds with respect to the orbital $O$. Let $u,v\in \widetilde{\mathbb{S}(3)}$ be such that $u\xrightarrow[]{1}v$. Then, $O$ can be neither of the form $x\xrightarrow[]{1}y$ nor $y\xrightarrow[]{1}x$ since then there would be some $z\in \widetilde{\mathbb{S}(3)}$ such that $u,v,z$ forms a copy of $L(1,1,1)$, which is forbidden. This can be seen by inspecting Figure~\ref{S3}. If $O$ is of the form $x\xrightarrow[]{2}y$, then $\widetilde{\mathbb{S}(3)}$ would contain a copy of $L(1,2,2)$, which is also forbidden. Finally, if $O$ is of the form $y\xrightarrow[]{2}x$, $\widetilde{\mathbb{S}(3)}$ would contain a copy of $L(2, 1, 2)$. Hence, there is no orbital for which the orbital extension property holds.
\end{proof}
As pointed out earlier, $\widetilde{\mathbb{S}(3)}$ and  $\mathbb{S}(4)$ are primitive binary finitely bounded homogeneous structures~\cite{Cherlinclass}. This implies that they are $\omega$-categorical and thus yield a counterexample to Question~\ref{qOEP}; moreover, they fall into the scope of the Bodirsky-Pinsker conjecture.

\section{Almost minimal operations above a Boolean group}\label{sec:boolalmost}
The main result of this section will be Theorem~\ref{Boolcase}, in which we classify almost minimal operations above a Boolean group $G\acts B$ of size $>2$. Compared to the three types Theorem~\ref{thm:threetypes}, we have an extra kind of almost minimal operation: $G$-quasi-minorities (Definition~\ref{deforbitmin}). The case of $|G|=2$ will be treated later in Section~\ref{sec:z2}.

\begin{remark}\label{rem:free} Note that if $G\acts B$ freely, then $G$ is topologically closed in $\mathcal{O}^{(1)}$; that is, in our notation we have  $\overline{G}=G$. This is  because by freeness of the action, for any $\alpha, \beta\in G, a\in B$, we have that  $\alpha a=\beta a$ implies that $\alpha=\beta$, so the only convergent sequences from $G$ are those which are eventually constant. An easy induction on terms then shows that moreover,  $\langle G\rangle=\overline{\langle G\rangle}$.
\end{remark}

\begin{lemma} \label{Bnomaj} Let $G\acts B$ where $|G|>2$. Then, there is no quasi-majority almost minimal above $\overline{\langle G\rangle}$. 
\end{lemma}
\begin{proof}
Suppose by contradiction that $m$ is a quasi-majority almost minimal above $\overline{\langle G\rangle}$. Since $|G|>2$, take $\alpha, \beta\in G\setminus \{1\}$ distinct. Lemma~\ref{majtoMal} implies $m(x,\gamma x,z)\approx m(z,z,z)\approx m(z,\gamma x,x)$ for any $\gamma \in G\setminus\{1\}$. Let $u$ be a new variable. Setting $x=u, \gamma=\alpha, z=\beta u$ in the first equation yields $m(u,\alpha u,\beta u)\approx m(\beta u,\beta u,\beta u)$; on the other hand, setting $x=\beta u, z=u, \gamma=\alpha\circ \beta^{-1}$ in the second equation yields $m(u,\alpha u,\beta u)\approx m(u,u,u)$, contradicting injectivity of the map  $x\mapsto m(x,x,x)$.
\end{proof}

\ignore{
\begin{proof} Suppose by contradiction that $m$ is a quasi-majority almost minimal above $\overline{\langle G\rangle}$. Since $|G|>2$, take $\alpha, \beta\in G\setminus \{1\}$ distinct. Since $\alpha$ and $\beta$ are distinct and the action of $G$ is faithful, there is $a\in B$ such that $\alpha a=b\neq c=\beta a$. Let $h(x,y,z)=m(x,\alpha y, z)$ and $l(x,y,z)=m(x,z,\beta y)$.\\

From Lemma~\ref{majtoMal}, we can deduce that
\begin{equation}\label{fun1}
    m(x,\alpha x, y)\approx m(y,y,y)
\end{equation}
 since
\[  m(x,\alpha x, y)\approx h(x,x,y)\approx h(y,y,y)\approx m(y,y,y).\]
We can also deduce that \begin{equation}\label{fun2}
    m(x, z, \beta x)\approx m(z,z,z)
\end{equation}
from Remark~\ref{alsol} since
\[m(x,z,\beta x)\approx l(x,x,z)\approx l(z,z,z)\approx m(z,z,z).\]

Consider $a,b,c$ as above. We have that
\[m(c,c,c)=m(a, \alpha a, c)= m(a,b,c)= m(a, b, \beta a)=m(b,b,b),\]
where the first identity holds by~\ref{fun1} and the last by~\ref{fun2}. However, since $b\neq c$, this contradicts injectivity of $m(x,x,x)$ and thus almost minimality of $m$.
\end{proof}
}

\begin{lemma}\label{Malnoproj} Let $G\acts B$ where $G$ is a Boolean group.  
Let $\alpha\in G\setminus\{1\}$. Suppose that $M(x,y,z)$ is quasi-Malcev and almost minimal above ${\langle G \rangle}$. Then,  $h(x,y,z):=M(x,\alpha y, z)$ is not a quasi-semiprojection. 
\end{lemma}

\begin{proof} The operation $h$ is also almost minimal above ${\langle G\rangle}$. Suppose by contradiction $h(x,y,z)$ is a quasi-semiprojection. Suppose that when two variables are the same $h$ depends on the first coordinate and equals the unary function $g\in{\langle G \rangle}$. Then, 
\[g(y)\approx h(y, \alpha y, y)\approx M(y,y,y)\approx M(\alpha y, \alpha y, y)\approx h(\alpha y, y, y)\approx g(\alpha y)\]
By injectivity of $g$, this cannot happen. The same argument works for showing that $h$ cannot depend on the third coordinate. So, suppose that $h(x,y,z)$ depends on the second coordinate when two variables are identified. Using the fact that $G$ is Boolean, 
we get that 
\[g(y)\approx h(\alpha y, y, \alpha y)\approx M(\alpha y, \alpha y, \alpha y)\approx M(\alpha y, y, y)\approx h(\alpha y, \alpha y, y)\approx g(\alpha y),\]
again contradicting injectivity of $g\in{\langle G\rangle}$. Hence, $h$ cannot be a quasi-semiprojection.

\end{proof}

\ignore{
\begin{lemma}\label{lotsMal} Let $G\acts B$ be Boolean.  
Let  $\alpha\in G\setminus\{1\}$. Suppose that $M(x,y,z)$ is quasi-Malcev and almost minimal above ${\langle G \rangle}$. If  $h(x,y,z)=M(x,\alpha y, z)$  is not a quasi-majority, then it is quasi-Malcev.
\end{lemma}
}
\begin{lemma}\label{lotsMal} 
Let $G\acts B$ and  
let  $\alpha\in G\setminus\{1\}$. Suppose that $M(x,y,z)$ is quasi-Malcev and almost minimal above ${\langle G \rangle}$. If some permutation of the variables of  $h(x,y,z)=M(x,\alpha y, z)$  is quasi-Malcev, then so is $h$ itself.
\end{lemma}
\begin{proof}
Suppose $h$ is not quasi-Malcev itself, but some function obtained by permuting its variables is.   Then  $h(x,y,x)\approx h(y,y,y)$, and hence
\[h(y,y,y)\approx h(\alpha y, y, \alpha y)\approx M(\alpha y, \alpha y, \alpha y)\approx M(\alpha y, y, y)\approx h(\alpha y, \alpha y, y)\; .\]
By almost minimality,  the function $h(x,x,y)$ is essentially unary, and depends on one of its variables injectively. By the above, this variable has to be $y$, and we moreover see $h(x,x,y)\approx h(y,y,y)$. The same argument shows that $h(y,x,x)\approx h(y,y,y)$, so $h$ is quasi-Malcev, a contradiction.
\end{proof}

\ignore{
\begin{proof} 
Since $h$ is almost minimal above  ${\langle G \rangle}$, by the five-types theorem~\ref{fivetypes}, it must be either a quasi-semiprojection,  or a quasi-majority, or quasi-Malcev (up to permutation of variables). From Lemma~\ref{Malnoproj}, it cannot be a quasi-semiprojection. It cannot be a quasi-majority either by hypothesis. Hence, it has to be quasi-Malcev up to permutation of variables.

Suppose $h$ is not quasi-Malcev itself, but some function obtained by permuting its variables is.   Then  $h(x,y,x)\approx h(y,y,y)$, and hence
\[h(y,y,y)\approx h(\alpha y, y, \alpha y)\approx M(\alpha y, \alpha y, \alpha y)\approx M(\alpha y, y, y)\approx h(\alpha y, \alpha y, y)\; .\]
By almost minimality,  the function $h(x,x,y)$ is essentially unary, and depends on one of its variables injectively. By the above, this variable has to be $y$, and we moreover see $h(x,x,y)\approx h(y,y,y)$. The same argument shows that $h(y,x,x)\approx h(y,y,y)$, so $h$ is quasi-Malcev, a contradiction.

\end{proof}
 }

Recall the notion of a $G$-quasi-minority from Definition~\ref{deforbitmin}.

\begin{lemma}\label{orbitmin} 
Let $G\acts B$ where $G$ is a Boolean group. Let $M(x,y,z)$ be a quasi-Malcev operation almost minimal above ${\langle G \rangle}$. Let $\beta\in G\setminus\{1\}$, and suppose that $M(x, \beta y, z)$ is quasi-Malcev. Then,
\begin{align}
    M(x,\beta x, y) & \approx M(\beta y, \beta y, \beta y)\approx M(y,\beta x, x) & \label{eq1}\\
    M(x,y,\beta x)& \approx M(\beta y, \beta y, \beta y) \label{eq2}
\end{align}
In particular, $M(x,y,z)$ is a $G$-quasi-minority.
\end{lemma}
\begin{proof}
By assumption, the function  $h(x,y,z)=M(x,\beta y, z)$ is quasi-Malcev. We thus  obtain the first equation in  (\ref{eq1}):
\[M(x,\beta x, y)\approx h(x,x,y)\approx h(y,y,y)\approx h(y, \beta y, \beta y)\approx M(y,y,\beta y)\approx M(\beta y, \beta y, \beta y).\]
By symmetry, the same argument also yields the second equation in (\ref{eq1}).

For equation (\ref{eq2}), note that by almost minimality of $M(x,y,z), \ l(x,y)=M(x,y,\beta x)$ is in ${\langle G \rangle}$ and so it can only depend on either the first or the second coordinate. We have
\[
l(x,x)\approx M(x,x,\beta x)\approx \beta x\qquad  \text{and} \qquad
l(x,\beta x)\approx M(x,\beta x, \beta x)\approx x\; ,
\]
hence that coordinate is the second one. This implies identity (\ref{eq2}) since 
\[
M(x,y,\beta x)\approx l(x,y)\approx l(y,y)\approx M(y,y,\beta y)\approx M(\beta y, \beta y, \beta y)\; .
\]
In particular, we can conclude that $M$ is a $G$-quasi-minority since it satisfies equation (\ref{orbitmindef}).
\end{proof}

\begin{corollary}\label{cororbitmin} Let $G\acts B$ where $G$ is a Boolean group 
 with $|G|>2$. Let $M(x,y,z)$ be a quasi-Malcev operation almost minimal above ${\langle G \rangle}$. Then, $M(x,y,z)$ is a $G$-quasi-minority.
\end{corollary}
\begin{proof}
Pick an arbitrary $\beta\in G\setminus\{1\}$. The function $M(x,\beta y, z)$ is almost minimal above $\langle G\rangle$. By Lemma~\ref{Bnomaj} and since $|G|>2$, this function cannot be a quasi-majority, and by Lemma~\ref{Malnoproj}, it cannot be a quasi-semiprojection. Hence, some permutation of its variables is quasi-Malcev, and so it is quasi-Malcev itself by Lemma~\ref{lotsMal}. Thus, Lemma~\ref{orbitmin} applies, and we see that $M(x,y,z)$ is a $G$-quasi-minority.
\end{proof}

\begin{theorem}[Theorem~\ref{thmd}, Boolean case]\label{Boolcase} Let $G\acts B$ be a Boolean group acting freely on $B$ with $s$-many orbits (where $s$ is possibly infinite) and assume that $|G|>2$. Let $f$ be an almost minimal operation above $\langle G \rangle$. Then,  $f$ is of one of the following types:
\begin{enumerate}
    \item a unary operation;
    \item a binary operation;
    \item a ternary $G$-quasi-minority;
    \item a $k$-ary orbit\-/semiprojection for $3\leq k \leq s$.
\end{enumerate} 
\end{theorem}
\begin{proof} It suffices to examine $f$ in the light of 
 Theorem~\ref{fivetypes} and apply our additional results for Boolean groups of order greater than two. Lemma~\ref{Bnomaj} tells us that $f$ cannot be a quasi-majority. Corollary~\ref{cororbitmin} tells us that if $f$ is quasi-Malcev, then it must be a $G$-quasi-minority.
\end{proof}

We conclude this section by showing that whenever $G\acts B$ is a Boolean group acting freely on $B$, then there exist $G$-quasi-minorities almost minimal above $\langle G\rangle$.

\begin{prop}\label{exmin} 
Let $G\acts B$ where $G$ is a Boolean group acting freely on $B$. Then, there exists a $G$-quasi-minority almost minimal above $\langle G\rangle$.
\end{prop}
\begin{proof}
For any triple $a=(a_1,a_2,a_3)$ of elements with at least two elements in the same orbit, we consider $(i,j)$  smallest lexicographically  such that  $a_i, a_j$ are in the same orbit, and let $k$ be the remaining coordinate. Since the action of $G$ is free, there exists a unique $\alpha\in G$ such that $a_j=\alpha a_i$; we set $f(a_1,a_2,a_3)=\alpha a_k$ and extend $f$ to $B^3$ arbitrarily. Note that by definition, $f(x,x,x)\approx x$. To see that $f$ is a $G$-quasi-minority, consider first a triple   $(a_1,a_2,a_3)$ such that precisely two of its  elements  $a_i, a_j$ belong to the same orbit, i.e.,  $a_i=\alpha a_j$ for a unique  $\alpha \in G$. Then also $a_j=\alpha a_i$ since $G$ is Boolean. By definition, for the remaining coordinate $k\notin\{i,j\}$ we have   $f(a_1,a_2,a_3)=\alpha a_k$  independently of whether $i<j$. Now consider a triple all three of whose  elements belong to the same orbit, i.e., $(a_1,a_2,a_3)=(\gamma_1 a, \gamma_2 a, \gamma_3 a)$ for some $\gamma_i \in G$ and $a\in B$. Then, using that $G$ is Boolean, we have  $f(a_1,a_2,a_3)=\gamma_1\gamma_2 a_3$ by definition. This equals $\gamma_1\gamma_2\gamma_3 a$, and since  $G$ is Boolean, in turn also $\gamma_1\gamma_3a_2$ and $\gamma_2\gamma_3 a_1$, showing that $f$ is indeed a $G$-quasi-minority. Almost minimality follows from Remark~\ref{allweak} and Lemma~\ref{lem:collapse}.
\end{proof}

\section{Minimal \texorpdfstring{$G$}{G}-quasi-minorities}\label{sec:mintwisted}
In this section, we characterise minimal $G$-quasi-minorities as $G$-invariant Boolean Steiner $3$-quasigroups (Definition~\ref{def:boolgtwisted}) in Proposition~\ref{prop:bts3g}. Subsection~\ref{sub:twist}  is dedicated to proving this,  whilst in Subsection~\ref{sub:chartwist} we further characterise the behaviour of these operations to obtain necessary and sufficient conditions for their existence (Corollary~\ref{cor:correspondence}) as well as an explicit understanding of their behaviour (Definitions~\ref{def:gweighted} and~\ref{def:boolgweighted}). This allows us, in Subsection~\ref{sub:counting}, to count the exact number of $G$-invariant Boolean Steiner $3$-quasigroups for a given free action of a Boolean group $G\acts B$ when $B$ is finite. As an intermediate step towards understanding minimal $G$-quasi-minorities we also classify strictly almost minimal $G$-quasi-minorities (Definition~\ref{def:semimin}), i.e., $G$-quasi-minorities which only generate ternary operations that generate them back. These can also be nicely characterised as $G$-invariant Steiner $3$-quasigroups (Definition~\ref{def:gtwistedsteiner}).

\subsection{\texorpdfstring{$G$}{G}-invariant Steiner \texorpdfstring{$3$}{3}-quasigroups}\label{sub:twist}

\begin{definition} A \textbf{Steiner $3$-quasigroup} is an idempotent symmetric minority $\mathfrak{q}$ satisfying the equation
\[\mathfrak{q}(x, y, \mathfrak{q}(x,y,z))\approx z. \tag{\ref{eq:sqslaw}}\]
\end{definition}

\begin{definition} A \textbf{Steiner quadruple system}, shortened to SQS, is a uniform $4$-hypergraph such that any three vertices are contained in a unique $4$-hyperedge. We call a $4$-hyperedge of a Steiner quadruple system a \textbf{block}.
\end{definition}

Steiner $3$-quasigroups on a given domain $B$ naturally correspond to Steiner quadruple systems on $B$: the $4$-ary relation 
\[(\mathfrak{q}(x_1,x_2,x_3)=x_4 )\wedge \bigwedge_{1\leq i< j\leq 4} x_i\neq x_j\]
yields a Steiner quadruple system on $B$ since one can prove by (\ref{eq:sqslaw}) this is a symmetric $4$-ary relation such that every three vertices are in a unique $4$-hyperedge (we essentially prove this again in Lemma~\ref{Boolsymm}). The converse procedure yields a Steiner $3$-quasigroup for any SQS on $B$. The connection between Steiner quadruple systems and Steiner $3$-quasigroups is heavily studied in universal algebra and design theory, starting with the work of~\cite{quackenbush1975algebraic, armanious1980algebraische}. We refer the reader to~\cite{quackenbush1999nilpotent} and~\cite[Section 4]{lindner1978steiner} for discussions of this connection. In earlier literature these operations are frequently called 'Steiner skeins'. \\

An important identity that a Steiner $3$-quasigroup might satisfy is the \textbf{Boolean law}:
\begin{equation}
    \mathfrak{q}(x, y, \mathfrak{q}(z, y, w))\approx \mathfrak{q}(x, z, w). \tag{\ref{boolaw}}
\end{equation}

The Boolean law says that in the Steiner quadruple system associated with $\mathfrak{q}$, whenever two blocks $H_1$ and $H_2$ intersect in two vertices, the remaining four vertices in $H_1\cup H_2$ form another block. 

\begin{definition} A Steiner $3$-quasigroup satisfying the Boolean law (\ref{boolaw}) is called Boolean. 
\end{definition}

\begin{remark}\label{rem:boolrem}
    It is easy to see that given a Boolean Steiner $3$-quasigroup, fixing any vertex $e\in B$, the structure $(B, +, e)$ with operation $a+b:=\mathfrak{q}(a, b, e)$ is a Boolean group and for all $a, b, c\in B$, 
    \[\mathfrak{q}(a, b, c)=a+b+c\]
    with respect to this operation~\cite{quackenbush1999nilpotent}. The associated SQSs are precisely Kirkman's Steiner quadruple systems~\cite{kirkman1847problem}, and we call them \textbf{Boolean Steiner quadruple systems}.
\end{remark}

\begin{definition}\label{def:gtwistedsteiner} Let $G\acts B$ be a Boolean group acting freely on $B$. A $\bm{G}$\textbf{-invariant Steiner} $\bm{3}$-\textbf{quasigroup} is a Steiner $3$-quasigroup satisfying the $\bm{G}$-\textbf{invariance identity}: for all $\alpha, \beta, \gamma\in G$
\begin{equation}
    \mathfrak{q}(\alpha x, \beta y, \gamma z)\approx \alpha\beta\gamma \mathfrak{q}(x, y, z). \tag{\ref{eq:twistid}}
\end{equation}
\end{definition}

\begin{lemma}\label{obstacle1}  Let $G\acts B$ be a Boolean group acting freely on $B$. Suppose that $\mathfrak{m}$ is a $G$-quasi-minority. Let $\alpha, \beta, \gamma\in G$. The operation $\mathfrak{m}'$ given by
\[\mathfrak{m}'(x,y,z):=\mathfrak{m}(x,y,\mathfrak{m}(\alpha x,\beta y,\gamma z))\]
is an orbit\-/semiprojection (not necessarily essential) with respect to its last coordinate.
\end{lemma}
\begin{proof} Say that $\mathfrak \mathfrak{q}(x,x,x)\approx \eta x$ for $\eta\in G$. Since $\mathfrak m$ is almost minimal, $\mathfrak m'$ is either almost minimal or in $\langle G\rangle$. For any $\delta\in G$, we have the following:
\begin{align*}
    \mathfrak{m}'(x,y,\delta x) &:=\mathfrak{m}(x, y, \mathfrak{m}(\alpha x,\beta y, \delta \gamma x))\approx 
    \mathfrak{m}(x, y, \eta  \alpha\beta \gamma \delta y)\approx \alpha\beta \gamma\delta  x;\\
    \mathfrak{m}'(x,\delta x, y) &:=\mathfrak{m}(x, \delta x, \mathfrak{m}(\alpha x, \beta \delta x, \gamma y))\approx
    \mathfrak{m}(x, \delta x, \eta\alpha\beta \gamma \delta y) \approx \alpha\beta \gamma y;
  \\
    \mathfrak{m}'(y, x, \delta x) &:=\mathfrak{m}(y, x, \mathfrak{m}(\alpha y,\beta x, \delta \gamma x))\approx \mathfrak{m}(y, x, \eta\alpha\beta \gamma \delta y)\approx \alpha\beta \gamma \delta x.
\end{align*}
In particular, whenever $(a_1,a_2, a_3)\in B^3$ are such that at least two of them are in the same orbit, $\mathfrak{m}'(a_1, a_2, a_3)=\alpha\beta\gamma a_3$, meaning that $\mathfrak m'$ is an orbit\-/semiprojection with respect to its last coordinate.
\end{proof}

We study a notion which is strictly weaker than minimality and strictly stronger than almost minimality:

\begin{definition}\label{def:semimin} The $k$-ary operation $f:B^k\to B$ is \textbf{strictly almost minimal} above $\overline{\langle G\rangle}$ if it is almost minimal above $\overline{\langle G\rangle}$ and for every essential operation $h\in\overline{\langle G \cup\{f\}\rangle}\cap \mathcal{O}^{(k)}$, we have that $h$  generates back $f$ (together with $G$) in the sense that $f\in\overline{\langle G \cup\{h\}\rangle}\cap \mathcal{O}^{(k)}$.
\end{definition}

Strict almost minimality is a weaker condition than minimality since it is possible that a strictly  almost minimal $f$ generates a minimal operation of strictly higher arity than itself. That this situation can occur  will be seen  when comparing strictly almost minimal $G$-quasi-minorities to minimal $G$-quasi-minorities. 

\begin{remark}\label{rem:emprime} Lemma~\ref{obstacle1} implies that if $\mathfrak{m}$ is a strictly almost minimal $G$-quasi-minority, then $\mathfrak{m}'\in\langle G\rangle$. This is because from Proposition~\ref{orbitgen} an orbit\-/semiprojection only generates more orbit\-/semiprojections in its same arity. 
%Hence, for $\mathfrak{m}$ to be strictly almost minimal, $\mathfrak{m}'$ cannot be an essential orbit\-/semiprojection and must be essentially unary.
This yields some very strong constraints on the behaviour of $\mathfrak{m}$ itself, which we will exploit in the rest of this section.
\end{remark}

\begin{lemma}\label{injectivity} Let $G\acts B$ where $G$ is a Boolean group acting freely on $B$. Suppose that $\mathfrak{m}$ is a $G$-quasi-minority strictly almost minimal above $\langle G\rangle$. Fix $a,b\in B$. Then, the function $\mathfrak{m}_{ab}:B\to B$ given by $c\mapsto \mathfrak{m}(a,b,c)$ is injective.
\end{lemma}
\begin{proof} From Remark~\ref{rem:emprime}, the operation
\[\mathfrak{m}(x,y,\mathfrak{m}(x,y,z))\]
is essentially unary depending on the variable $z$. Hence, we have that
\begin{equation}\label{twistedcalc1}
    \mathfrak{m}(x,y,\mathfrak{m}(x,y,z))=\mathfrak{m}(x,x,\mathfrak{m}(x,x,z)) =\mathfrak{m}(x,x,\mathfrak{m}(z,z,z))=z.
\end{equation}
The first equality uses the fact that this operation is essentially unary depending on its last coordinate. The second equality uses the fact that $\mathfrak{m}$ is a quasi-minority. The final equality uses the fact that $\mathfrak{m}(z,z,z)\in\langle G\rangle$ and that $G$ is Boolean.

In particular, from the identity (\ref{twistedcalc1}) we get
\begin{equation}\label{gettinginj}
\mathfrak{m}(x,y,z)=w \Rightarrow \mathfrak{m}(x,y,w)=z.
\end{equation}
Suppose that 
\[\mathfrak m_{ab}(c)=\mathfrak m_{ab}(c')=d.\] 
From (\ref{gettinginj}), that implies 
\[\mathfrak{m}(a,b,d)=c \text{ and } \mathfrak{m}(a,b,d)=c',\]
yielding that $c=c'$. Hence, $\mathfrak m_{ab}$
 is injective.    
\end{proof}

\begin{lemma}\label{Boolsymm} Let $G\acts B$ be a Boolean group acting freely on $B$. Suppose that $\mathfrak{m}$ is a $G$-quasi-minority strictly almost minimal above $\langle G\rangle$. Then, the $4$-ary relation given by
\[\mathfrak{m}(x,y,z)=w\]
is symmetric.
\end{lemma}
\begin{proof} 
From Lemma~\ref{obstacle1} and the same calculations as equation (\ref{twistedcalc1}), for any permutation of $\{1,2,3\}, \sigma\in S_3$,
\[\mathfrak{m}(x_1, x_2,  \mathfrak{m}(x_{\sigma 1}, x_{\sigma 2}, x_{\sigma 3}))=x_3.\]
Hence, by injectivity of $\mathfrak{m}_{ab}$, $\mathfrak{m}(x,y,z)$ is a symmetric operation. Finally, from this and (\ref{gettinginj}), we get that the relation 
$\mathfrak{m}(x,y,z)=w$ is symmetric.
\end{proof}

\begin{remark} From the above two lemmas one can also deduce that minorities strictly almost minimal above the trivial group are Steiner $3$-quasigroups and the well-known fact~\cite{lindner1978steiner} that the latter induce a Steiner quadruple system on the domain.
\end{remark}

\begin{lemma}\label{twistingaround} Let $G\acts B$ be a Boolean group acting freely on $B$. Suppose that $\mathfrak{m}$ is a $G$-quasi-minority strictly almost minimal above $\langle G\rangle$. Then for all $\alpha, \beta, \gamma\in G$,
\[\mathfrak{m}(\alpha x, \beta y, \gamma z)=\alpha\beta\gamma \mathfrak{m}(x,y,z).\]    
\end{lemma}
\begin{proof} From Remark~\ref{rem:emprime}, we know that for $\alpha\in G$,
\[\mathfrak{m}(x,y,\mathfrak{m}(\alpha x, \alpha y, z))\]
is an essentially unary operation depending on its last coordinate. In particular, for $\mathfrak{m}(x,x,x)=\gamma x$ for $\gamma\in G$
\begin{align*}
\mathfrak{m}(x,y,\mathfrak{m}(\alpha x, \alpha y, z)) &\approx \mathfrak{m}(x,x,\mathfrak{m}(\alpha x, \alpha x, z))\\
& \approx \mathfrak{m}(x,x,\gamma z)\\
& \approx z\\
&\approx  \mathfrak{m}(x,y,\mathfrak{m}(x, y, z)). 
\end{align*}
Here the first equality follows from the operation being essentially unary. The second and third equalities follow from $\mathfrak{m}$ being a quasi-minority with $\mathfrak{m}(x,x,x)=\gamma x$ and $\gamma^2=1$. The final equality is obtained by performing the previous steps in reverse order. 
From Lemma~\ref{injectivity} and the above identity we can deduce that
\begin{equation}\label{applysym}
    \mathfrak{m}(\alpha x, \alpha y, z)\approx \mathfrak{m}(x, y, z).
\end{equation}
More generally, by the same argument, for $\tau_1, \tau_2, \tau_3\in G$ such that for $i, j, k\in\{1,2,3\}$ distinct $\tau_i=\tau_j=\alpha$ and $\tau_k=1$,
\begin{equation}\label{doubling}
   \mathfrak{m}(\tau_1 x, \tau_2 y, \tau_3 z)\approx \mathfrak{m}(x, y, z).
\end{equation}

By symmetry (i.e. Lemma~\ref{Boolsymm}) applied to equation (\ref{applysym}), 
\[\mathfrak{m}(\alpha x, \alpha y, \mathfrak{m}(x, y, z))\approx z.\]
However, by equation (\ref{doubling}),
\[\mathfrak{m}(\alpha x, y, \alpha \mathfrak{m}(x, y, z))\approx z.\]
and so, using symmetry again,
\[\mathfrak{m}(\alpha x, y, z)\approx \alpha \mathfrak{m}(x, y, z).\]
Applying the same reasoning for each variable we obtain the desired equation.  
\end{proof}

From the preceding lemmas we obtain:

\begin{corollary}\label{semimistein} Let $G\acts B$ be a Boolean group acting freely on $B$. Let $\mathfrak{m}$ be a $G$-minority strictly almost minimal above $\langle G\rangle$. Then, $\mathfrak{m}$ is a $G$-invariant Steiner $3$-quasigroup.
\end{corollary}
\begin{proof}
    By assumption, $\mathfrak{m}$ is an idempotent minority. By Lemma~\ref{Boolsymm}, it is symmetric. By Lemma~\ref{obstacle1}, $\mathfrak{m}$ satisfies (\ref{eq:sqslaw}). By Lemma~\ref{twistingaround} it satisfies the $G$-invariance identity (\ref{eq:twistid}). Hence, it is a $G$-invariant Steiner $3$-quasigroup.
\end{proof}

Conversely, below we show that $G$-invariant Steiner $3$-quasigroups are strictly almost minimal above $\langle G\rangle$. Indeed, the space of ternary operations they generate together with $G$ is particularly trivial.

\begin{lemma}\label{steinsemi} Let $G\acts B$ be a Boolean group acting freely on $B$. Let $\mathfrak{q}$ be a $G$-invariant Steiner $3$-quasigroup. Then, 
\[\left((\overline{\langle G\cup\{\mathfrak{q}\}\rangle})\setminus \langle G\rangle \right)\cap\mathcal{O}^{(3)}=\{\eta \mathfrak{q} \vert \eta\in G\}.\]
In particular, $\mathfrak{q}$ is strictly almost minimal above $\langle G\rangle$.
\end{lemma}
\begin{proof}
Firstly, note that a $G$-invariant Steiner $3$-quasigroup is always a $G$-minority. This is easy to see since $\mathfrak{q}$ is a minority and for $\beta\in G$,
\[\mathfrak{q}(x, \beta x, y)\approx \beta \mathfrak{q}(x,x,y)\approx \beta y\approx \beta^3 \mathfrak{q}(y,y,y)\approx \mathfrak{q}(\beta y,\beta y,\beta y).\]
In particular, from the equation above and symmetry, we can deduce that $\mathfrak{q}$ satisfies (\ref{orbitmindef}) and so is a $G$-minority.\\
We claim that any operation in 
$\langle G\cup\{\mathfrak{q}\}\rangle \cap\mathcal{O}^{(3)}$ is either essentially unary or of the form $\eta \mathfrak{q}(x,y,z)$ for $\eta\in G$. We prove this by induction on complexity of terms. The base case is trivial since the statement holds for projections and $\mathfrak{q}$. Suppose that $g_1, g_2, g_3\in \langle G\cup\{\mathfrak{q}\}\rangle \cap\mathcal{O}^{(3)}$ satisfy the inductive hypothesis (i.e. are all essentially unary or of the form $\eta\mathfrak{q}$). The only interesting case to consider for the inductive step is that of
\[\mathfrak{q}(g_1(x,y,z), g_2(x,y,z), g_3(x,y,z)).\]
If $g_1=\eta_1\mathfrak{q}$ and $g_2=\eta_2\mathfrak{q}$, 
\begin{align*}
    \mathfrak{q}(\eta_1\mathfrak{q}(x,y,z), \eta_2\mathfrak{q}(x,y,z), g_3(x,y,z))& \approx \eta_1\eta_2\mathfrak{q}(g_3(x,y,z), g_3(x,y,z), g_3(x,y,z))\\
    & \approx \eta_1\eta_2\beta g_3(x,y,z),
\end{align*}

yielding that $\mathfrak{q}(g_1(x,y,z), g_2(x,y,z), g_3(x,y,z))$ is either essentially unary or again of the form $\eta\mathfrak{q}$. Hence, the only non-trivial case that remains to be considered is that of only one of the $g_i$ being of the form $\eta\mathfrak{q}$ whilst the others are essentially unary. For example, say $g_1=\alpha x, g_2=\beta y$ and $g_3=\eta\mathfrak{q}$. Then,
\[\mathfrak{q}(\alpha x, \beta y, \eta\mathfrak{q}(x,y,z))\approx \alpha\beta\eta \mathfrak{q}(x,y,\mathfrak{q}(x,y,z))\approx \alpha\beta\eta z,\]
which is again essentially unary.\\

This completes the proof that every operation in 
$\langle G\cup\{\mathfrak{q}\}\rangle \cap\mathcal{O}^{(3)}$ is either essentially unary or of the form $\eta \mathfrak{q}( x, y, z)$. This also implies that
\[\langle G\cup\{\mathfrak{q}\}\rangle\cap\mathcal{O}^{(3)}=\overline{\langle G\cup\{\mathfrak{q}\}\rangle}\cap\mathcal{O}^{(3)},\]
since $\overline{\langle G\rangle}=\langle G\rangle$ and for $\eta_1\neq\eta_2$, $\eta_1\mathfrak{q}(x,x,x)$ and $\eta_2\mathfrak{q}(x,x,x)$ must disagree on every point, by freeness of $G\acts B$.\\

Finally, we just proved that all operations in $\overline{\langle G\cup\{\mathfrak{q}\}\rangle}\cap\mathcal{O}^{(3)}$ also generate $\mathfrak{q}$, and so $\mathfrak{q}$ is strictly almost minimal as desired. 
\end{proof}

The result that a $G$-quasi-minority $\mathfrak{m}$ strictly almost minimal above $\langle G \rangle$ only generates essentially unary operations and its $G$-conjugates $\eta \mathfrak{m}$ will be important later when considering strictly almost minimal operations above $\langle \mathbb{Z}_2\rangle$.\\

From Corollary~\ref{semimistein} and Lemma~\ref{steinsemi} we obtain the following:

\begin{prop}\label{semichar} Let $G\acts B$ be a Boolean group acting freely on $B$. Then, a ternary operation $\mathfrak{m}:B^3\to B$ is a $G$-minority strictly almost minimal above $\langle G \rangle$ if and only if it is a $G$-invariant Steiner $3$-quasigroup.
\end{prop}

Now that we have characterised strictly almost minimal $G$-minorities, we can proceed to describe minimal $G$-minorities. We just proved these are $G$-invariant Steiner $3$-quasigroups. We will now show that they correspond to the Boolean $G$-invariant Steiner $3$-quasigroups.

\begin{lemma}\label{lem:minBool} Let $G\acts B$ be a Boolean group acting freely on $B$. Let $\mathfrak{m}$ be a $G$-minority minimal above $\langle G \rangle$. Then, $\mathfrak{m}$ is a $G$-invariant Boolean Steiner $3$-quasigroup. 
\end{lemma}
\begin{proof} Since $\mathfrak{m}$ is minimal, and so strictly almost minimal, from Corollary~\ref{semimistein}, it is a $G$-invariant Steiner $3$-quasigroup.  Hence, we just need to show that $\mathfrak{m}$ satisfies the Boolean law (\ref{boolaw}). To prove this, consider 
\[C(x,y,z,w):=\mathfrak{m}(x, \mathfrak{m}(x,z,w), \mathfrak{m}(z,y,w)).\]
One can compute that $C(x,y, z, w)$ is a $4$-ary orbit\-/semiprojection such that for all $a, b, c, d\in B$ with at least two of them in the same orbit,
\[C(a, b, c, d)=b.\]
We show this computation below. We use the basic properties of $G$-invariant Steiner $3$-quasigroups and the fact that $G$ is Boolean. For $\beta\in G$,
\begin{align*}
C(x,\beta x, z,w) := & \mathfrak{m}(x, \mathfrak{m}(x,z,w), \mathfrak{m}(z,\beta x,w)) \approx 
 \beta \mathfrak{m}(x, \mathfrak{m}(x,z,w), \mathfrak{m}(x, z, w))  \approx  \beta x \; ;\\
C(x, y, \beta x, w) :=& \mathfrak{m}(x, \mathfrak{m}(x,\beta x,w), \mathfrak{m}(\beta x,y,w)) \approx \mathfrak{m}(x, \mathfrak{m}(x,x,w), \mathfrak{m}(x,y,w)) \\
\approx & \mathfrak{m}(x, w, \mathfrak{m}(x,y,w)) \approx  y \; ;\\
C(x,y, \beta y, w) :=& \mathfrak{m}(x, \mathfrak{m}(x,\beta y,w), \mathfrak{m}(\beta y,y,w))  \approx  \mathfrak{m}(x, \mathfrak{m}(x,y,w), \mathfrak{m}(y,y,w)) \\
\approx & \mathfrak{m}(x, \mathfrak{m}(x,y,w), w) \approx y \; ;\\
C(x,y, z, \beta z) :=& \mathfrak{m}(x, \mathfrak{m}(x,z,\beta z), \mathfrak{m}(z,y,\beta z)) \approx \mathfrak{m}(x, x, y) \approx  y.
\end{align*}
All calculations used (\ref{eq:twistid}) and the fact that $\mathfrak{m}$ is a symmetric minority. The second and third equation used (\ref{eq:sqslaw}) for the final identity. Note that by symmetry of $\mathfrak{m}$ and the definition of $C(x,y,z,w)$, for $u\in\{x,y\}$, the case of $u$ and $w$ being in the same orbit is identical to that of $u$ and $z$ being in the same orbit both in computation and in output. Hence, we can conclude that for any $G$-invariant Steiner $3$-quasigroup $\mathfrak{m}$, $C$ is an orbit-semiprojection with respect to its second coordinate. By minimality of $\mathfrak{m}$, it must be an actual projection onto its second coordinate. That is, 
\begin{equation}\label{eq:minswitchstein}
    \mathfrak{m}(x, \mathfrak{m}(x,z,w), \mathfrak{m}(z,y,w))\approx y.
\end{equation}
Since $\mathfrak{m}(x,y,z)=w$ is a symmetric $4$-ary relation (Lemma~\ref{Boolsymm}), we can switch $y$ and $\mathfrak{m}(x,z,w)$ in the identity (\ref{eq:minswitchstein}) to obtain the Boolean law.
\end{proof}

\begin{lemma}\label{lem:whatgislike} Let $G\acts B$ be a Boolean group acting freely on $B$. Suppose that $\mathfrak{q}$ is a $G$-invariant Boolean Steiner $3$-quasigroup. Let $(B, +)$ be the Boolean group such that $\mathfrak{q}(x,y,z)=x+y+z$. Then, 
every operation $g\in \overline{\langle\{\mathfrak{q}\}\cup G\rangle}$ depending on all variables is of the form
\begin{equation}\label{eq:formofg}
    g(x_1, \dots, x_r)=\alpha \sum_{i=1}^{r} x_i,
\end{equation}
where $\alpha\in G$ and $r$ is an odd number. In particular,
\[\overline{\langle\{\mathfrak{q}\}\cup G\rangle}=\langle\{\mathfrak{q}\}\cup G\rangle.\]
\end{lemma}
\begin{proof} By Remark~\ref{rem:boolrem}, $\mathfrak{q}(x,y,z)=x+y+z$ for a Boolean group $(B, +)$. We claim that any $g\in\langle\{\mathfrak{q}\}\cup G\rangle$ is (up to permuting the $x_i$) of the form 
\[g(x_1, \dots, x_k)=\alpha \sum_{i=1}^{r} x_i,\]
where $\alpha\in G$, and $r\leq k$ is odd. As usual, we prove this by induction on terms where the base case is trivial. For the inductive step, we need to consider
\[g(x_1, \dots, x_k):=\mathfrak{q}\left(\alpha\sum_{i\in I} x_i, \beta \sum_{j\in J} x_j, \gamma \sum_{l\in L} x_l\right)\; ,\]
where $I, J, L\subseteq\{1, \dots, k\}$ all have odd size. Let $S\subseteq\{1,\dots, k\}$ be the set of $i\leq k$ which appear in either exactly one, or all three of $I, J,$ and $L$. Note that $S$ must have odd size by the inclusion-exclusion principle. By $G$-invariance of $\mathfrak{q}$ and since $(B, +)$ is a Boolean group, we get that
\[g(x_1, \dots, x_k):=\alpha\beta\gamma \sum_{i\in S} x_i\;,\]
proving our inductive hypothesis. Moreover, it is easy to see that each operation of the form (\ref{eq:formofg}) belongs to $\langle\{\mathfrak{q}\}\cup G\rangle$ by noting that for $r\geq 3$ odd
\[\alpha \sum_{i\leq r} x_i=\alpha \mathfrak{q}(x_1, x_1, \mathfrak{q}(x_1, x_2, \mathfrak{q}(\dots,\mathfrak{q}(x_1, x_{r-1}, x_r)\dots))),\]
observing that $x_1$ appears an odd number of times since $r$ is odd.

Now, we prove that $\overline{\langle\{\mathfrak{q}\}\cup G\rangle}=\langle\{\mathfrak{q}\}\cup G\rangle$. Consider $k$-ary $g_1, g_2\in\langle\{\mathfrak{q}\}\cup G\rangle$ of the form
\[g_1(x_1, \dots, x_k)=\alpha \sum_{i\in R} x_i \text{ and } g_2(x_1, \dots, x_k)=\beta \sum_{i\in S} x_i,\]
for $S, R\subseteq \{1, \dots, k\}$ both of odd size. We show that if $g_1$ and $g_2$ agree on a set of two elements, then they are the same operation. Take $a,b\in B$ distinct and suppose that $g_1$ and $g_2$ agree on $\{a,b\}$. Firstly, since $g_1$ and $g_2$ agree on $\{a\}$, 
\[g_1(a, \dots, a)=\alpha a=\beta a=g_2(a, \dots, a).\]
This implies that $\alpha=\beta$ since $G\acts B$ is free. Now, since $g_1$ and $g_2$ agree on $\{a,b\}$, we can show that $R=S$. In fact, suppose that $j\in R$, but $j\not\in S$ for some $j\in \{1, \dots, k\}$. Take $\overline{c}$ to be the $k$-tuple consisting of $a$ in every coordinate except for a $b$ in the $j$th coordinate. Then, since $(B, +)$ is Boolean, $g_1(\overline{c})=b$, but $g_2(\overline{c})=a$. Hence, if $g_1$ and $g_2$ agree on $\{a,b\}$, we have $\alpha=\beta$ and $R=S$, and so $g_1=g_2$. We can conclude that $\overline{\langle\{\mathfrak{q}\}\cup G\rangle}=\langle\{\mathfrak{q}\}\cup G\rangle$, as desired, and the lemma follows. 
\end{proof}

\begin{lemma}\label{lem:boolmin} Let $G\acts B$ be a Boolean group acting freely on $B$. Suppose that $\mathfrak{q}$ is a $G$-invariant Boolean Steiner $3$-quasigroup. Then, $\mathfrak{q}$ is minimal above $\langle G \rangle$.
\end{lemma}
\begin{proof} Consider $g\in\langle\{\mathfrak{q}\}\cup G\rangle$ of arity $k\geq 4$ depending on all variables. By Lemma~\ref{lem:whatgislike}, we know that $k$ is odd and
\[g(x_1, \dots, x_k)=\alpha \sum_{i\leq k} x_i.\]
Then, since $G$ and $(B, +)$ are Boolean and $k$ is odd,
\[\alpha g(x_1, x_2, x_3, x_1, \dots, x_1)\approx \mathfrak{q}(x_1, x_2, x_3),\]
yielding that $\mathfrak{q}\in \langle \{g\} \cup G \rangle$. Hence, by Lemma~\ref{lem:whatgislike}, strict almost minimality and the fact that any operation  $g\in\langle\{\mathfrak{q}\}\cup G\rangle$ of arity $k\geq 4$ generates $\mathfrak{q}$ back together with $G$, we can deduce that $\mathfrak{q}$ is indeed minimal above $\langle G \rangle$.
\end{proof}

\begin{remark} Note that essentially the same proofs also yields that a Boolean Steiner $3$-quasigroup is minimal above $\langle \mathrm{Id} \rangle$ and that any minority minimal above $\langle \mathrm{Id} \rangle$ is a Boolean Steiner $3$-quasigroup, thus recovering the original description of idempotent minimal minorities of Rosenberg~\cite{fivetypes}. 
\end{remark}

Hence, from Lemma~\ref{lem:minBool} and Lemma~\ref{lem:boolmin}, we can deduce that minimal $G$-minorities correspond to $G$-invariant Boolean Steiner $3$-quasigroups:

\begin{prop}\label{prop:bts3g} Let $G\acts B$ be a Boolean group acting freely on $B$. Then, a ternary operation $\mathfrak{m}:B^3\to B$ is a $G$-minority minimal above $\langle G \rangle$ if and only if it is a $G$-invariant Boolean Steiner $3$-quasigroup.
\end{prop}

\subsection{Characterising \texorpdfstring{$G$}{G}-invariant Steiner \texorpdfstring{$3$}{3}-quasigroups} \label{sub:chartwist}
In the previous subsection we proved correspondences between $G$-invariant Steiner $3$-quasigroups and strictly almost minimal $G$-quasi-minorities, and between Boolean $G$\-/invariant Steiner $3$-quasigroups and minimal $G$-quasi-minorities. In this subsection, we try and give a more accurate description of these operations. In particular, we want to obtain necessary and sufficient conditions for their existence. We do so in Remark~\ref{rem:exginvsteiner} and Corollary~\ref{cor:correspondence}, in which we give necessary and sufficient conditions for the existence of $G$-invariant Steiner $3$-quasigroups, and of Boolean $G$-invariant Steiner $3$-quasigroups respectively.\\

We begin by defining $G$-weighted Steiner $3$-quasigroups. We will see that these correspond to $G$-invariant Steiner $3$-quasigroups in the following two lemmas.

\begin{definition}\label{def:gweighted} Let $G\acts B$ be a Boolean group acting freely on $B$. Let $X=\{ x_i \;\vert\; i\in \mathrm{Orb}(G)\}$ be a set of representatives for the orbits of $G$. 
A $\bm{G}$\textbf{-weighted Steiner} $\bm{3}$\textbf{-quasigroup} (with respect to $X$) consists of a triplet 
$(\mathfrak{q}, \mathfrak{q}', f)$, where 
\begin{itemize}
\item The ternary operation $\mathfrak{q}': \mathrm{Orb}(G)^3\to \mathrm{Orb}(G)$ is a Steiner $3$-quasigroup;
\item The set $E$ is given by
\[E:=\{(\Tilde{x}, \Tilde{y}, \Tilde{z}, \Tilde{w})\in\mathrm{Orb}(G)^4\ \vert \ \mathfrak{q}'(\Tilde{x}, \Tilde{y}, \Tilde{z})= \Tilde{w}\}.\]
    \item The function $f:E\to G$ is such that
    \begin{equation}\label{eq:fis0}
        f(\Tilde{a}, \Tilde{b}, \Tilde{c},\Tilde{d})=0
    \end{equation}
whenever $|\{\Tilde{a}, \Tilde{b}, \Tilde{c},\Tilde{d}\}|<4$;
\item The ternary operation $\mathfrak{q}:B^3\to B$ is defined as follows: given $a,b,c\in B$, 
\begin{equation}\label{eq:gweighted}
    \mathfrak{q}(a, b, c)=\alpha \beta \gamma f(\Tilde{a},\Tilde{b}, \Tilde{c}, \mathfrak{q}'(\Tilde{a},\Tilde{b}, \Tilde{c})) x_{\mathfrak{q}'(\Tilde{a},\Tilde{b}, \Tilde{c})} \;, 
\end{equation}
    where $a=\alpha x_{\Tilde{a}}, b=\beta x_{\Tilde{b}}, c=\gamma x_{\Tilde{c}}$.
\end{itemize}
    Note that this is well-defined since the action of $G$ on $B$ is free and so there is a unique $\alpha$ such that  $a=\alpha x_{\Tilde{a}}$.  
\end{definition}

\begin{lemma}\label{misgtwisted} Let $G\acts B$ be a Boolean group acting freely on $B$. Let $X=\{ x_i \vert i\in \mathrm{Orb}(G)\}$ be a set of representatives for the orbits of $G$. Consider a $G$-weighted Steiner $3$-quasigroup $(\mathfrak{q}, \mathfrak{q}', f)$. Then, $\mathfrak{q}$ is a $G$-invariant Steiner $3$-quasigroup.
\end{lemma}
\begin{proof} From the condition on $f$ given by (\ref{eq:fis0}), we can see that $\mathfrak{q}$ is idempotent. In fact, for $a=\alpha x_{\Tilde{a}}$,
\begin{equation*}
    \mathfrak{q}(a, a, a)=\alpha f(\Tilde{a}, \Tilde{a}, \Tilde{a}, \mathfrak{q}'(\Tilde{a}, \Tilde{a}, \Tilde{a})) x_{\Tilde{a}}=\alpha x_{\Tilde{a}}=a,
\end{equation*}
where the first equality is the definition of $\mathfrak{q}$ in  (\ref{eq:gweighted}) and the second is (\ref{eq:fis0}). From the definition (\ref{eq:gweighted}), we can see that $\mathfrak{q}$ is a symmetric minority satisfying the $G$-invariance identity (\ref{eq:twistid}). Hence, we only need to check that $\mathfrak{q}$ is indeed a Steiner $3$-quasigroup. We verify this below. Let $a=\alpha x_{\Tilde{a}}, b=\beta x_{\Tilde{b}}, c=\gamma x_{\Tilde{c}}$
\begin{align*}
    \mathfrak{q}(a, b, \mathfrak{q}(a, b, c))&=\mathfrak{q}(a, b, \alpha \beta \gamma f(\Tilde{a},\Tilde{b}, \Tilde{c}, \mathfrak{q}'(\Tilde{a},\Tilde{b}, \Tilde{c})) x_{\mathfrak{q}'(\Tilde{a},\Tilde{b}, \Tilde{c})})\\
    &=\gamma  f(\Tilde{a},\Tilde{b}, \Tilde{c}, \mathfrak{q}'(\Tilde{a},\Tilde{b}, \Tilde{c})) f(\Tilde{a},\Tilde{b}, \Tilde{c}, \mathfrak{q}'(\Tilde{a},\Tilde{b}, \Tilde{c})) x_{\Tilde{c}}=c.
\end{align*}
The first equality is just (\ref{def:gweighted}). The second is also  (\ref{eq:gweighted}) noting that since $\mathfrak{q}'$ is a Steiner $3$-quasigroup,  $\mathfrak{q}'(\Tilde{a},\Tilde{b},\mathfrak{q}'(\Tilde{a},\Tilde{b}, \Tilde{c}))=\Tilde{c}$. The final equality follows from $G$ being Boolean and the definition of $c$. 
\end{proof}

\begin{definition}\label{def:inducedmin} Let $G\acts B$ be a Boolean group acting freely on $B$. Let $\mathfrak{q}$ be a $G$-invariant Steiner $3$-quasigroup. Then, $\mathfrak{q}$ induces a Steiner $3$-quasigroup on the orbits of $G$, $\Tilde{\mathfrak{q}}:\mathrm{Orb}(G)\to \mathrm{Orb}(G)$, where for $\Tilde{a}, \Tilde{b}, \Tilde{c}\in\mathrm{Orb}(G)$, 
\[\Tilde{\mathfrak{q}}(\Tilde{a}, \Tilde{b}, \Tilde{c})=\Tilde{d},\]
where $\Tilde{d}$ is the unique orbit of $G$ such that for some/all $a\in\Tilde{a}, b\in\Tilde{b}, c\in\Tilde{c}$, 
\[\mathfrak{q}(a, b, c)\in \Tilde{d}.\]
This is a well-defined operation due to the $G$-invariance identity (\ref{eq:twistid}). It is then easy to show that $\Tilde{\mathfrak{q}}$ is a symmetric idempotent minority satisfying the condition to be a Steiner $3$-quasigroup. Moreover, if $\mathfrak{q}$ is Boolean, then so is $\Tilde{\mathfrak{q}}$.
\end{definition}

\begin{lemma}\label{lem:inducedf} Let $G\acts B$ be a Boolean group acting freely on $B$. and let $X=\{ x_i \vert i\in \mathrm{Orb}(G)\}$ be a set of representatives for the orbits of $G$. Let $\mathfrak{q}$ be a $G$-invariant Steiner $3$-quasigroup. Then, $(\mathfrak{q}, \Tilde{\mathfrak{q}}, f)$ is a $G$-weighted Steiner $3$-quasigroup with respect to $X$, where $\Tilde{\mathfrak{q}}$ is the Steiner $3$-quasigroup on the orbits of $G$ induced by $\mathfrak{q}$ defined in Definition~\ref{def:inducedmin}, and $f$ is the function $f:E\to G$ where, 
\[E:=\{(\Tilde{x}, \Tilde{y}, \Tilde{z}, \Tilde{w})\in\mathrm{Orb}(G)^4\ \vert \ \Tilde{\mathfrak{q}}(\Tilde{x}, \Tilde{y}, \Tilde{z})= \Tilde{w}\},\]
and for $\mathfrak{q}(x_{\Tilde{a}}, x_{\Tilde{b}}, x_{\Tilde{c}})=\delta x_{\Tilde{\mathfrak{q}}(\Tilde{a},\Tilde{b}, \Tilde{c})}$,
\[f(\Tilde{a},\Tilde{b}, \Tilde{c}, \Tilde{\mathfrak{q}}(\Tilde{a},\Tilde{b}, \Tilde{c}))=\delta.\] 
\end{lemma}
\begin{proof} We only need to check the identity from (\ref{eq:gweighted}).   For $a=\alpha x_{\Tilde{a}}, b=\beta x_{\Tilde{b}}, c=\gamma x_{\Tilde{c}}$, and $\mathfrak{q}(x_{\Tilde{a}}, x_{\Tilde{b}}, x_{\Tilde{c}})=\delta x_{\Tilde{\mathfrak{q}}(\Tilde{a},\Tilde{b}, \Tilde{c})}$,
\begin{align*}
\mathfrak{q}(a, b, c)&=\alpha\beta\gamma \mathfrak{q}(x_{\Tilde{a}}, x_{\Tilde{b}}, x_{\Tilde{c}})\\
&=\alpha \beta \gamma \delta x_{\Tilde{\mathfrak{q}}(\Tilde{a},\Tilde{b}, \Tilde{c})}\\
&=\alpha \beta \gamma f(\Tilde{a},\Tilde{b}, \Tilde{c}, \Tilde{\mathfrak{q}}(\Tilde{a},\Tilde{b}, \Tilde{c})) x_{\Tilde{\mathfrak{q}}(\Tilde{a},\Tilde{b}, \Tilde{c})}\;. 
\end{align*}
Here the first equality follows from the $G$-invariance identity (\ref{eq:twistid}) and the last equality is just from the definition of $f$.
\end{proof}

\begin{notation} Since we just proved that $G$-invariant Steiner $3$-quasigroups correspond to $G$-weighed Steiner $3$-quasigroups, from now on we write the latter as $(\mathfrak{q}, \Tilde{\mathfrak{q}}, f)$, where $\Tilde{\mathfrak{q}}$ is given in Definition~\ref{def:inducedmin}. This is mainly to avoid overly dense notation.
\end{notation}

\begin{remark}\label{rem:exginvsteiner}
Lemmas~\ref{misgtwisted} and~\ref{lem:inducedf} imply necessary and sufficient conditions for the existence of a $G$-invariant Steiner $3$-quasigroup. In particular, from Hanani~\cite{hanani1960quadruple}, we know that Steiner quadruple systems on $s$ many vertices exist if and only if $s\equiv 2$ or $4 \ \mathrm{mod}\ 6$ (or is infinite). Hence, for $G\acts B$ a Boolean group acting freely on $B$ with $s$ many orbits, there is a $G$-invariant Steiner $3$-quasigroup if and only if
$s\equiv 2$ or $4 \ \mathrm{mod} \ 6$ (or is infinite). Indeed, given the lack of restrictions on $f$ in Definition~\ref{def:gweighted}, we can compute that for $s$-many $G$-orbits the number of $G$-invariant Steiner $3$-quasigroups is
\[\mathrm{nSQS}(s)|G|^{\frac{s(s-1)(s-2)}{24}},\]
where $\mathrm{nSQS}(s)$ is the number of Steiner quadruple systems on $s$ many vertices and $s(s-1)(s-2)/{24}$ is the number of blocks (i.e. $4$-hyperedges) in a Steiner quadruple system of size $s$.
\end{remark}

\begin{definition}\label{def:boolgweighted} A $G$-weighted Steiner $3$-quasigroup $(\mathfrak{q}, \Tilde{\mathfrak{q}}, f)$ is \textbf{Boolean} if $\Tilde{\mathfrak{q}}$ is Boolean and, for all $\Tilde{a}, \Tilde{b}, \Tilde{c}, \Tilde{d},$
\begin{equation}\label{eq:eqnfcond}
    f(\Tilde{a}, \Tilde{b}, \Tilde{c}, \Tilde{a}\oplus \Tilde{b}\oplus \Tilde{c}) f(\Tilde{a}, \Tilde{b}, \Tilde{d}, \Tilde{a}\oplus \Tilde{b}\oplus \Tilde{d})=f(\Tilde{c}, \Tilde{d}, \Tilde{a}\oplus \Tilde{b}\oplus \Tilde{c}, \Tilde{a}\oplus \Tilde{b}\oplus \Tilde{d}),
\end{equation}
  writing $\Tilde{\mathfrak{q}}(x, y, z)=x\oplus y \oplus z$, where $\oplus$ is addition in a Boolean group on $\mathrm{Orb}(G)$. 
\end{definition}

\begin{remark} The condition (\ref{eq:eqnfcond}) tells us that in the Steiner quadruple system induced by $\Tilde{\mathfrak{q}}$ on $\mathrm{Orb}(G)$, $f$ assigns blocks elements of $G$ so that whenever two blocks $H_1$ and $H_2$ intersect in two vertices, the block containing the remaining four vertices in $H_1\cup H_2$ is assigned value $\alpha\beta$ by $f$, were $H_1$ and $H_2$ were assigned values $\alpha$ and $\beta$ respectively. Note that if any two elements in $\{\Tilde{a}, \Tilde{b}, \Tilde{c}, \Tilde{d}\}$ are equal, then both sides of  (\ref{eq:eqnfcond}) are zero. 
\end{remark}

\begin{lemma}\label{lem:equationfcond} Let $G\acts B$ be a Boolean group acting freely on $B$. Let $X=\{ x_i \vert i\in \mathrm{Orb}(G)\}$ be a set of representatives for the orbits of $G$. Let $\mathfrak{q}$ be a $G$-invariant Boolean Steiner $3$-quasigroup and let $(\mathfrak{q}, \Tilde{\mathfrak{q}}, f)$ be the associated $G$-weighted Steiner $3$-quasigroup from Lemma~\ref{lem:inducedf}. Then, $(\mathfrak{q}, \Tilde{\mathfrak{q}}, f)$ is Boolean.
\end{lemma}

\begin{proof} Consider $x_{\Tilde{d}}, x_{\Tilde{c}}, x_{\Tilde{a}}, x_{\Tilde{b}}$. By the Boolean law, 
\begin{equation}\label{eq:fequality}
    \mathfrak{q}(x_{\Tilde{d}}, x_{\Tilde{c}}, \mathfrak{q}(x_{\Tilde{a}}, x_{\Tilde{c}},x_{\Tilde{b}}))=\mathfrak{q}(x_{\Tilde{a}}, x_{\Tilde{b}}, x_{\Tilde{d}}).
\end{equation}
But being a $G$-weighted Steiner $3$-quasigroup, we also have the following for the left-hand side of (\ref{eq:fequality}):
\begin{align*}
    \mathfrak{q}(x_{\Tilde{c}}, x_{\Tilde{d}}, \mathfrak{q}(x_{\Tilde{a}}, x_{\Tilde{b}}, x_{\Tilde{c}}))&= f(\Tilde{a}, \Tilde{b}, \Tilde{c}, \Tilde{a}\oplus\Tilde{c}\oplus\Tilde{b})\mathfrak{q}(x_{\Tilde{d}}, x_{\Tilde{c}}, x_{\Tilde{a}\oplus\Tilde{c}\oplus\Tilde{b}})\\
    &=f(\Tilde{a}, \Tilde{b}, \Tilde{c}, \Tilde{a}\oplus\Tilde{c}\oplus\Tilde{b})f(\Tilde{c}, \Tilde{d}, \Tilde{a}\oplus\Tilde{b}\oplus\Tilde{d}, \Tilde{a}\oplus\Tilde{b}\oplus\Tilde{c}) x_{\Tilde{a}\oplus\Tilde{b}\oplus\Tilde{d}}.
\end{align*}
Meanwhile, for the right-hand side of (\ref{eq:fequality}), we have
\[\mathfrak{q}( x_{\Tilde{a}}, x_{\Tilde{b}}, x_{\Tilde{d}})=f({\Tilde{a}}, {\Tilde{b}}, {\Tilde{d}},\Tilde{d}\oplus\Tilde{a}\oplus\Tilde{b}) x_{\Tilde{a}\oplus\Tilde{b}\oplus\Tilde{d}}.\]
Hence, comparing the two sides of equation (\ref{eq:fequality}) we deduce that
\[f(\Tilde{a}, \Tilde{b}, \Tilde{c}, \Tilde{a}\oplus\Tilde{c}\oplus\Tilde{b})f(\Tilde{d}, \Tilde{c}, \Tilde{d}\oplus\Tilde{a}\oplus\Tilde{b}, \Tilde{c}\oplus\Tilde{a}\oplus\Tilde{b}) = f({\Tilde{a}}, {\Tilde{b}}, {\Tilde{d}}, \Tilde{d}\oplus\Tilde{a}\oplus\Tilde{b}), \]
and since $G$ is Abelian we also get the desired equation.
\end{proof}

\begin{lemma} Let $G\acts B$ be a Boolean group acting freely on $B$. Let $X=\{ x_i \vert i\in \mathrm{Orb}(G)\}$ be a set of representatives for the orbits of $G$. Let $(\mathfrak{q}, \Tilde{\mathfrak{q}}, f)$ be a  Boolean $G$-weighted Steiner $3$-quasigroup. Then, $\mathfrak{q}$ is a $G$-invariant Boolean Steiner $3$-quasigroup.
\end{lemma}
\begin{proof}
We know from Lemma~\ref{misgtwisted} that $\mathfrak{q}$ is a $G$-invariant Steiner $3$-quasigroup, so we only need to prove the Boolean law. We have for $a=\alpha x_{\Tilde{a}}, b=\beta x_{\Tilde{b}}, c=\gamma x_{\Tilde{c}}, d=\delta  x_{\Tilde{d}}$,
\begin{align*}
\mathfrak{q}(a, b, \mathfrak{q}(c, b, d)) &=\gamma\beta\delta f(\Tilde{c}, \Tilde{b}, \Tilde{d}, \Tilde{c}\oplus \Tilde{b} \oplus \Tilde{d})\mathfrak{q}(a, b, x_{\Tilde{c}\oplus \Tilde{b} \oplus \Tilde{d}})\\
&=\alpha\gamma\delta f(\Tilde{c}, \Tilde{b}, \Tilde{d}, \Tilde{c}\oplus \Tilde{b} \oplus \Tilde{d})\mathfrak{q}(x_{\Tilde{a}}, x_{\Tilde{b}}, x_{\Tilde{c}\oplus \Tilde{b} \oplus \Tilde{d}})\\
    &=\alpha\gamma\delta f(\Tilde{c}, \Tilde{b}, \Tilde{d}, \Tilde{c}\oplus \Tilde{b} \oplus \Tilde{d}) f(\Tilde{a}, \Tilde{b}, \Tilde{c}\oplus \Tilde{b} \oplus \Tilde{d}, \Tilde{c}\oplus \Tilde{a} \oplus \Tilde{d})  x_{\Tilde{c}\oplus \Tilde{a} \oplus \Tilde{d}}\\
    &=\alpha\gamma\delta f(\Tilde{c}, \Tilde{a}, \Tilde{d}, \Tilde{c}\oplus \Tilde{a} \oplus \Tilde{d}) x_{\Tilde{c}\oplus \Tilde{a} \oplus \Tilde{d}}\\
    &=\mathfrak{q}(a,c, d).
\end{align*}
The first three identities and the last one are just unravelling the definition. Equation (\ref{eq:eqnfcond}) is used in the fourth equality.  
\end{proof}

\begin{remark}\label{rem:existenceBool} Note that whenever $\mathrm{Orb}(G)$ is the domain of a Boolean group (i.e. $|\mathrm{Orb}(G)|=2^n$ or is infinite), we can construct a Boolean $G$-weighted Steiner $3$-quasigroup $(\mathfrak{q}, \Tilde{\mathfrak{q}}, f)$. The construction is as follows: take $\Tilde{\mathfrak{q}}(x,y,z)=x\oplus y\oplus z$, where $\oplus$ is addition in a Boolean group on $\mathrm{Orb}(G)$. Assign to each point of $\Tilde{a}\in \mathrm{Orb}(G)$ an element of $G$, $g_{\Tilde{a}}$. Thus, for $\Tilde{a}, \Tilde{b}, \Tilde{c}$ in $\mathrm{Orb}(G)$, let
\[f(\Tilde{a}, \Tilde{b}, \Tilde{c}, \Tilde{a}\oplus\Tilde{b}\oplus\Tilde{c})=g_{\Tilde{a}}+g_{\Tilde{b}}+g_{\Tilde{c}}+g_{\Tilde{a}\oplus\Tilde{b}\oplus\Tilde{c}}.\]
It is easy to verify that $f$ satisfies the conditions of (\ref{eq:eqnfcond}) since $G$ is a Boolean group. For $|\mathrm{Orb}(G)|=2^n$ for $n\leq 3$ (with few calculations in the $n=3$ case), all Boolean $G$-weighted Steiner $3$-quasigroups are of this form, though it looks implausible this extends to larger numbers of orbits. We can thus conclude the following:
\end{remark}

\begin{corollary}\label{cor:correspondence} Let $G\acts B$ be a Boolean group acting freely on $B$. There is a one-to-one correspondence between $G$-invariant Boolean Steiner $3$-quasigroups $\mathfrak{q}$ and Boolean $G$-weighted Steiner $3$-quasigroups $(\mathfrak{q}, \Tilde{\mathfrak{q}}, f)$. In particular, there is a $G$-quasi-minority minimal above $\langle G \rangle$ if and only if $|\mathrm{Orb}(G)|=2^n$ for some $n\in\mathbb{N}$ or is infinite.  
\end{corollary}

\subsection{Counting \texorpdfstring{$G$}{G}-invariant Boolean Steiner \texorpdfstring{$3$}{3}-quasigroups}\label{sub:counting}
In Corollary~\ref{cor:correspondence}, we found a correspondence between $G$-invariant Boolean Steiner $3$\-/quasigroups and Boolean $G$-weighted Steiner $3$\-/quasigroups. Whilst this correspondence is sufficient to obtain necessary and sufficient conditions on $G\acts B$ for the existence of a $G$\-/invariant Boolean Steiner $3$-quasigroup, it does not yet give us a precise characterisation of what the $G$\-/invariant Boolean Steiner $3$-quasigroups for $G\acts B$ look like. The careful reader might note that in order to do this, we would need to describe all functions satisfying condition (\ref{eq:eqnfcond}). Given the algebraic description of Boolean Steiner $3$-quasigroups as points and planes in an affine space over $\mathbb{Z}_2$ (Remark~\ref{rem:boolrem}), this boils down to solving a small problem in affine and projective geometry. In this subsection we find a way to describe all finite Boolean Steiner $3$-quasigroups. In particular, we show how any function of the form (\ref{eq:eqnfcond}) may be constructed (Lemmas~\ref{lem:constonpl},~\ref{lem:welldefined}, and~\ref{lem:WDformula}). This allows in Corollary~\ref{cor:counting} to count the exact number of $G$-invariant Boolean Steiner $3$-quasigroup for a given free action of a Boolean group $G\acts B$.\\

This subsection is slightly more technical than other parts of the paper, and the reader who is not interested in the precise behaviour of $G$-invariant Steiner $3$-quasigroups may skip it with no harm to their understanding of later sections. We also expect the reader of this section to have a basic understanding of projective and affine geometry, at least to the extent that projective spaces have well-defined notions of dimension and independence. The various facts we mention below on Steiner triple systems and Steiner quadruple systems can be found, for example, in \cite{lindner1978steiner} and~\cite{dixon1996permutation}.

\begin{definition} A \textbf{Steiner triple system} is a uniform $3$-hypergraph $(P, \mathcal{L})$ where every two vertices are contained in a unique $3$-hyperedge. We call the vertices of a Steiner triple system \textbf{points} and its $3$\-/hyperedges \textbf{lines}. We are interested in the \textbf{Boolean projective Steiner triple systems}. These Steiner triple system have as domain the $n$-dimensional projective space over $\mathbb{Z}_2$, $\mathrm{PG}(n, 2)$ and their lines are triplets of distinct points $\{a,b,c\}$ such that $a+b=c$. It is well-known that being a Boolean projective Steiner triple system is equivalent to every three \textbf{independent} vertices (i.e., vertices not forming a line) being contained in a copy of the Fano plane (see Figure~\ref{fig:Fanoplane}), i.e., the Boolean projective Steiner triple system on $\mathrm{PG}(2, 2)$. 
\end{definition}

\begin{definition} Let $(S, \mathcal{B})$ be a Steiner quadruple system. Fix $a\in S$. The \textbf{induced Steiner triple system} is the Steiner triple system $(S\setminus\{a\}, \mathcal{L})$, where the lines in $\mathcal{\mathcal{L}}$ consist of triplets of elements $b,c,d\in S\setminus\{a\}$ such that $\{a,b,c,d\}\in\mathcal{B}$. It is well-known that if $(S, \mathcal{B})$ is a Boolean Steiner quadruple system of size $2^n$, all of its induced Steiner triple systems are isomorphic the Boolean projective Steiner triple system on $\mathrm{PG}(n-1, 2)$.
\end{definition}

For a Boolean projective Steiner triple system $(P, \mathcal{L})$, given independent $a,b,c\in P$, let $\mathcal{F}(a,b,c)$ be the copy of the Fano plane containing $a,b,c$. 

\begin{definition} Let $G$ be a Boolean group. Let $(S, \mathcal{B})$ be a Boolean Steiner quadruple system. The function $g:\mathcal{B}\to G$ is \textbf{transparent} if for $H_1:=\{a,b,c,d\}, H_2:=\{a,b,e,f\}\in\mathcal{B}$, 
\begin{equation}\label{eq:transp}
   g(H_1)+g(H_2)=g(H_3), 
\end{equation}
where $H_3=\{c,d,e,f\}$. 
\end{definition}

Note that the condition (\ref{eq:eqnfcond}) from the definition of a $G$-weighted Boolean Steiner $3$\-/quasigroup $(\mathfrak{q}, \Tilde{\mathfrak{q}}, f)$ (Definition \ref{def:boolgweighted}) says precisely that the restriction of $f$ on injective $4$-tuples from $E$ is transparent. Hence, if we can give a description of all transparent functions on the Boolean Steiner quadruple system of size $2^n$, we can also understand all $G$-weighted Boolean Steiner $3$\-/quasigroups when $G\acts B$ has $2^n$ many orbits. In order to obtain this description, we first need to reduce the problem of describing transparent functions on a Boolean Steiner quadruple system to that of describing another class of functions on its induced Boolean projective Steiner triple systems. We do so below.

\begin{definition} Let $G$ be a Boolean group.
Let $(P, \mathcal{L})$ be a Boolean projective Steiner triple system. The function $h:\mathcal{L}\to G$ is \textbf{constant on planes} if for any three independent points $a,b,c\in P$, there is some $\theta_{\mathcal{F}(a,b,c)}\in G$ such that for any three lines $l_1, l_2, l_3\in\mathcal{L}\cap\mathcal{F}(a,b,c)$ passing through a common point, we have
\[h(l_1)+h(l_2)+h(l_3)=\theta_{\mathcal{F}(a,b,c)}.\]
\end{definition}

\begin{lemma}\label{lem:constonpl} Let $(S, \mathcal{B})$ be a Boolean Steiner quadruple system and for $a\in S$, let $(S\setminus\{a\}, \mathcal{L})$ be its induced Boolean Steiner triple System. Let $G$ be a Boolean group. Suppose that $g:\mathcal{B}\to G$ is transparent. Consider the function $h:\mathcal{L}\to G$ where for $\{b,c,d\}\in\mathcal{L}$, 
\[h(\{b,c,d\}):=g(\{a,b,c,d\}).\]
Then, $h$ is constant on planes. 
\end{lemma}
\begin{proof} Consider three independent points, $b,c,d\in S\setminus\{a\}$ and the corresponding Fano plane $\mathcal{F}(b,c,d)$. Without loss of generality, we need to prove that $h$ assigns values in $G$ to the lines in $\mathcal{F}(b,c,d)$ so that the sum of the values it assigns to the lines passing through $b$ is the same as that of the values assigned to the lines passing through $c$. Firstly, there is a line passing through $b$ and $c$, which we can therefore ignore for this calculation. The remaining four lines are of the form
\[\{b,d,l\},\{b,f, e\}, \{c,f, l\}, \{c, d, e\}\;,\]
    for $f,e,l\in \mathcal{F}(b,c,d)\setminus\{b,c,d\}$. In particular, note that since in $(S, \mathcal{B})$, $\{a, b,d,l\},\{a, b,f, e\}\in\mathcal{B}$, by the Boolean law, we have that $\{d, e, f, l\}\in\mathcal{B}$ and by transparency of $g$, 
    \begin{align*}
    g(\{d, e, f, l\}) & =g(\{a, b,d,l\})+g(\{a, b,f, e\})\; \text{ and } \\
    g(\{d, e, f, l\})&=g(\{a,c,f,l\})+ g(\{a, c, d, e\})\; .
\end{align*}
The above also means, by definition of $h$, that
\[h(\{b,d,l\})+h(\{b,f, e\})=h(\{c,f, l\})+h(\{c, d, e\})\; ,\]
from which the desired statement that $h$ is constant on planes follows.
\end{proof}

\begin{lemma}\label{lem:welldefined} Let $(S, \mathcal{B})$ be a Boolean Steiner quadruple system and for $a\in S$, let $(S\setminus\{a\}, \mathcal{L})$ be its induced Boolean Steiner triple System. Let $h:S\setminus\{a\}\to G$ be constant on planes. Define $g:\mathcal{B}\to G$ as follows: for $\{a,b,c,d\}\in\mathcal{B}$, let
\[g(\{a,b,c,d\})=h(\{b,c,d\}).\]
For $\{b,c,d,e\}\in\mathcal{B}$ not containing $a$, let $f\in S$ be such that $\{a,b,c,f\}\in\mathcal{B}$. Then, by the Boolean law, $\{a,d,e,f\}\in\mathcal{B}$, and set
\[g(\{b,c,d,e\})=h(\{b,c,f\})+h(\{d,e,f\}).\]
We have that $g$ is well-defined and transparent.
\end{lemma}
\begin{proof} To prove that $g$ is well-defined, given $\{b,c,d,e\}\in\mathcal{B}$ consider $f, f'\in S$ such that $\{a,b,c,f\}, \{a, c, e, f'\}\in\mathcal{B}$. By the Boolean law, $\{a,d,e,f\}, \{a,b,d, f'\}\in\mathcal{B}$. Hence, we need to show that 
\begin{equation}\label{eq:constantonpasch}
   h(\{f, b,c\})+h(\{d,e,f\})=h(\{c,e,f'\})+h(\{d,b,f'\}).
\end{equation}
Since $h$ is constant on planes, in the Fano plane $\mathcal{F}(b, f, f')$ in $(S\setminus\{a\}, \mathcal{L})$, for $l$ such that $\{l, f, f'\}\in\mathcal{L}$, we have that 
\[h(\{f, b,c\})+h(\{d,e,f\})+h(\{l, f, f'\})=h(\{c,e,f'\})+h(\{d,b,f'\})+h(\{l, f, f'\}),\]
which implies (\ref{eq:constantonpasch}).\\

Now we need to prove that $g$ is transparent. From the definition of $g$ it is clear that $g$ satisfies (\ref{eq:transp}) for any triplet of blocks containing $a$. So, we need to consider 
\[\{b,c,d,e\}, \{c,d,f,l\}, \{b,c, f, l\}\in\mathcal{B},
\]
for $b,c,d,e,f\in S\setminus a$. Note that by definition of $g$ (and the fact it is well-defined), we have that
\[g(\{b,c,d,e\})+g(\{e,d,f,l\})= g(\{b,c, f, l\})\]
if and only if, for $r,s,t\in S$ such that $\{a, r, b, c\}, \{a, s, d, e\}, \{a, t, f, e\}\in\mathcal{B}$, 
\begin{multline*}
   (g(\{a, r, b, c\})+g(\{a, s, d, e\})+(g(\{a, s, d, e\})+g(\{a, t, f, e\}))\\
   =(g(\{a, r, b, c\})+g(\{a, t, f, e\}))\; . 
\end{multline*}
    The latter holds since $G$ is Boolean, yielding that $g$ is indeed transparent.
\end{proof}

\begin{definition} Let $(P, \mathcal{L})$ be a Boolean projective Steiner triple system. A set of lines $\mathcal{L}'\subseteq\mathcal{L}$ is \textbf{weight-determining} if any map $h':\mathcal{L}'\to G$ can be extended uniquely to a map $h:\mathcal{L}\to G$ which is constant on planes. We will prove below such sets of lines exist. Note that any two weight-determining sets of lines must have the same cardinality. In particular, for $n\geq 1$ the $n$th weight dimension $\mathrm{WD}(n)$ is the size of a weight-determining set of lines for the projective plane $\mathrm{PG}(n, 2)$. We set $\mathrm{WD}(0)=0$.
\end{definition}

\begin{remark}\label{rem:counting} Note that if $\mathrm{PG}(n, 2)$ has a weight-determining set of lines, the number of maps $h:\mathcal{L}\to G$ which are constant on planes is exactly $|G|^{\mathrm{WD}(n)}$. Indeed, since for a Steiner quadruple system $(S, \mathcal{B})$ with $|S|=2^n$ transparent maps are in one-to-one correspondence with maps on the lines of its induced Steiner triple system which are constant on planes, $|G|^{\mathrm{WD}(n-1)}$ would also be the total number of transparent maps.
\end{remark}

\begin{figure}[t]
\begin{center}
\begin{tikzpicture}[scale=1.8]

  \draw[very thick] (30:1)  -- (210:2)
        (150:1) -- (330:2)
        (270:1) -- (90:2)
        (90:2)  -- (210:2) -- (330:2) -- cycle
        (0:0)   circle (1);

\draw[very thick, Gray] (30:1)  -- (210:2)
        %(150:1) -- (330:2)
     %   (270:1) -- (90:2)
        (90:2)  -- (210:2) -- (330:2) -- cycle;
  \fill (0:0)   circle(2pt) 
        (30:1)  circle(2pt)
        (90:2)  circle(2pt)
        (150:1) circle(2pt)
        (210:2) circle(2pt)
        (270:1) circle(2pt)
        (330:2) circle(2pt);
\node[anchor=west, inner sep=10pt] at (0:0) {$g$};
\node[anchor=south west, inner sep=10pt] at (30:1)   {$b$};
\node[anchor=south, inner sep=10pt] at         (90:2)   {$a$};
\node[anchor=south east, inner sep=10pt] at         (150:1)  {$e$};
\node[anchor=east, inner sep=10pt] at         (210:2) {$d$};
\node[anchor=north, inner sep=10pt] at         (270:1)  {$f$};
\node[anchor=west, inner sep=10pt] at         (330:2) {$c$};
\node at (-1.35, -0.25) [rectangle,fill=white, draw] () {\small $\alpha$};
\node at (-1.1, -0.6) [rectangle,fill=white, draw] () {\small $\beta$};
\node at (-0.8, -1) [rectangle,fill=white, draw] () {\small $\gamma$};
\node at (0.3, 1.5) [rectangle,fill=white, draw] () {\small $\delta$};
\node at (0.3, 1.5) [rectangle,fill=white, draw] () {\small $\delta$};
\node at (0, 1.2) [rectangle,fill=white, draw] () {\small$\beta+\gamma+\delta$};
\node at (-0.5, 0.22) [rectangle,fill=white, draw] () {\small $\alpha+\beta+\delta$};
\node at (0.6, -0.75) [rectangle,fill=white, draw] () {\small $\alpha+\gamma+\delta$};

\end{tikzpicture}
\end{center}
    \caption{Illustration of a weight-determining set $\mathcal{L}'$ for the Fano plane $\mathrm{PG}(2,2)$ and how to extend a map $h':\mathcal{L}'\to G$ (uniquely) to a map $h:\mathcal{L}\to G$ which is constant on planes. The lines in light-gray form a weight-determining set $\mathcal{L}'$. Starting from a map $h':\mathcal{L}'\to G$ such that $h'(\{d, e, a\})=\alpha, h'(\{d, g, b\})=\beta, h'(\{d, f, c\})=\gamma,$ and $h'(\{a,b,c\})=\delta$, we write in boxes the values in $G$ that each line must be assigned in its unique extension $h:\mathcal{L}\to G$ which is constant on planes. In particular, since the values assigned by $h$ to the lines passing through $a$ must sum to $\alpha+\beta+\gamma$, it is easy to see that we must have $h(\{a,g,f\})=\beta+\gamma+\delta$. Similarly, considering the lines passing through $b$ and through $c$, we obtain $h(\{c,g,e\})=\alpha+\beta+\delta$, and $h(\{b,f,e\})=\alpha+\gamma+\delta$. Finally, it is easy to verify that for each of $e, f$, and $g$, the sum of the values that $h$ assigns to the lines passing though them is still $\alpha+\beta+\gamma$, as desired. This yields that $h$ is constant on planes, and that it is the unique such extension of $h'$.
}\label{fig:Fanoplane}
\end{figure}

\begin{lemma}\label{lem:BCWD} We have that $\mathrm{WD}(1)=1$ and $\mathrm{WD}(2)=4$. In particular, a weight-determining set for the Fano plane $\mathrm{PG}(2,2)$ consists of all three lines passing through a fixed point and a fourth line not passing through that point.
\end{lemma}
\begin{proof} Since $\mathrm{PG}(1, 2)$ consists of a single line, it is trivial that $\mathrm{WD}(1)=1$. Now, consider the Fano plane $\mathrm{PG}(2, 2)$. Let $\{a,b,c\}\in\mathcal{L}$ and for $d\in \mathrm{PG}(2, 2)\setminus\{a,b,c\}$, take $\mathcal{L}'$ to be the three lines passing through $d$, and $\{a,b,c\}$. It is easy to compute that any map $h':\mathcal{L}'\to G$ extends uniquely to a map $h:\mathcal{L}\to G$ which is constant on planes. The reader may verify this by consulting Figure~\ref{fig:Fanoplane}. 
\end{proof}

\begin{lemma}\label{lem:WDformula} For $n\geq 0$, we have
\[\mathrm{WD}(n+1)=\mathrm{WD}(n)+2^{n+1}-1.\]
Hence, $\mathrm{WD}(n)=2^{n+1}-n-2$. 
\end{lemma}
\begin{proof} We prove this by induction with Lemma~\ref{lem:BCWD} providing us with the base case. Assume by induction that $\mathrm{WD}(n)$ exists. Consider a vertex $a\in\mathrm{PG}(n+1, 2)$ and take an $n$-dimensional hyperplane $\mathcal{H}$ not containing $a$. By inductive hypothesis $\mathcal{H}$ has a weight-determining set $\mathcal{L}_1$ of size $\mathrm{WD}(n)$. Let $\mathcal{L}_2$ be the set of lines passing through $a$. Note that $|\mathcal{L}_2|=2^{n+1}-1$. Set $\mathcal{L}':=\mathcal{L}_1\cup\mathcal{L}_2$ and consider any map $h':\mathcal{L}'\to G$. We want to prove that $h'$ extends uniquely to a map $h:\mathcal{L}\to G$ which is constant on planes, where $\mathcal{L}$ is the set of lines of $\mathrm{PG}(n+1, 2)$. Firstly, by inductive hypothesis, $h'_{\upharpoonright\mathcal{L}_1}$ has a unique extension to the set of lines of $\mathcal{H}$, $\mathcal{L}_\mathcal{H}$, $h_1:\mathcal{L}_\mathcal{H}\to G$ which is constant on planes. Then, any (Fano) plane $\mathcal{F}$ containing $a$ intersects $\mathcal{H}$ in a line, and so it contains exactly four lines in $\mathcal{L}_2\cup\mathcal{H}$: the line in $\mathcal{H}$ and three lines passing through $a$. By Lemma~\ref{lem:BCWD}, this is a weight-determining set for the lines of $\mathcal{F}$, $\mathcal{L}_\mathcal{F}$, and so there is a unique extension of $h'\cup h_1$ to $\mathcal{L}_\mathcal{F}$. Since every line of $\mathrm{PG}(n+1, 2)$ not in $\mathcal{L}_2\cup\mathcal{H}$ is contained in a unique Fano plane containing $a$, we can extend $h'\cup h_1$ to $\mathcal{L}$ in a well-defined manner. Moreover, by construction, this extension is the unique extension of $h'$ which is constant on planes and $|\mathcal{L}'|=\mathrm{WD}(n)+2^{n+1}-1$. Solving the recurrence relation for $\mathrm{WD}(n)$ with $\mathrm{WD}(1)=1$ yields the desired final formula.
\end{proof}

\begin{corollary}\label{cor:counting} Let $G\acts B$ be the free action of a Boolean group with $2^n$ many orbits. For $n=0$ or $1$, there is a unique $G$-invariant Boolean Steiner $3$-quasigroup. For $n\geq 1$, the number of $G$-invariant Boolean Steiner $3$-quasigroups is
\[\frac{(2^{n}-1)!|G|^{(2^{n}-n-1)}}{\prod_{k=0}^{n-1}(2^n-2^k)} \;.\] 
\end{corollary}
\begin{proof} Firstly, by Corollary~\ref{cor:correspondence}, there is a one-to-one correspondence between $G$-invariant Boolean Steiner $3$-quasigroup and Boolean $G$-weighted Steiner $3$-quasigroups. By the definition of the latter (Definition~\ref{def:boolgweighted}), these are in one-to-one correspondence with transparent maps on the blocks of the possible Steiner quadruple systems on $\mathrm{Orb}(G)$. Given a Boolean Steiner quadruple system $(S, \mathcal{B})$ on $\mathrm{Orb}(G)$, by Remark~\ref{rem:counting} and Lemma~\ref{lem:WDformula}, the total number of transparent maps on its blocks is 
\[|G|^{\mathrm{WD}(n-1)}=|G|^{2^n -n-1}.\]
As mentioned earlier, we can think of a Boolean Steiner quadruple system as the points and planes of the affine space $\mathrm{AG}(n, 2)$. It is well-known that the automorphism group of such a system is the space of affine transformations of $\mathbb{Z}_2^n$, $\mathrm{Aff}(n,2):=\mathbb{Z}_2^n\rtimes\mathrm{GL}(n, 2)$~\cite[Example 6.2.3]{dixon1996permutation}. Hence, by the orbit-stabilizer theorem, the number of distinct Boolean Steiner quadruple systems on $2^n$ many vertices is
\[\frac{2^n!}{2^n\prod_{k=0}^{n-1}(2^n-2^k)}\;.\]
This gives us the final number of $G$-invariant Boolean Steiner $3$-quasigroups.
\end{proof}

\section{Almost minimal and minimal operations above \texorpdfstring{${\langle\mathbb{Z}_2\rangle}$}{Z2}}\label{sec:z2}

In this Section, we classify almost minimal and minimal operations in the remaining case of $\mathbb{Z}_2$ acting freely on $B$. Subsection~\ref{sub:almostminz2} is dedicated to almost minimal operations, which we classify in Theorem~\ref{Z2case}. Two new kinds of operation appear: odd majorities (Definition~\ref{oddmajdef}) and odd Malcev operations (Definition~\ref{oddmaldef}). In Subsection~\ref{sub:noalmostminodd}, we show these operations cannot be strictly almost minimal (or minimal), thus completing the type-classification part of the proof of Theorem~\ref{maintheorem}.

\subsection{Almost minimal operations above \texorpdfstring{${\langle\mathbb{Z}_2\rangle}$}{Z2}}\label{sub:almostminz2}

\begin{remark} Recall the definition of an odd majority $m$ (Definition~\ref{oddmajdef}). Note that this makes sense in the context of $\mathbb{Z}_2$ acting freely on $B$. Given a triplet from $B$, if three elements are all in the same orbit of $\mathbb{Z}_2$, two of them must be equal and so $m$ acts as a majority. However, if a triplet contains two distinct elements in the same orbit and a third element $c$ outside that orbit, $m$ takes the value $m(c,c,c)$ for that triplet.
\end{remark}

\begin{lemma}\label{ood} Let $G\acts B$. Let $m$ be a quasi-majority almost minimal above $\overline{\langle G\rangle}$. Then $m$ is an odd majority.
\end{lemma}
\begin{proof}
By Lemma~\ref{majtoMal} we have that for any $\gamma\in G\setminus\{1\}$,  $m(x,\gamma y,z)$ is quasi-Malcev, and in particular \[m(x,\gamma x,y)\approx m(y,y,y)\; .
\]
By the symmetry of the definition of a  quasi-majority, the same identity holds for any permutation of the variables on the left side.
\end{proof}

\begin{corollary}\label{Maltooddmaj} Let $G\acts B$ be a Boolean group acting freely on $B$. Let $M$ be a quasi-Malcev operation almost minimal above ${\langle G\rangle}$. Suppose that $M$ is not a $G$-quasi-minority. Then, for all $\gamma\in G\setminus\{1\}, \ M(x,\gamma y, z)$ is an odd-majority. 
\end{corollary}
\begin{proof} By the five types theorem (Theorem~\ref{fivetypes}), $M(x,\gamma y, z)$ is either a quasi-semiprojection, or it is quasi-Malcev up to permutation of variables, or it is a quasi-majority. By Lemma~\ref{Malnoproj}, it is not a quasi-semiprojection. 
By Lemmas~\ref{lotsMal} and~\ref{orbitmin}, if any permutation of the variables of $M(x, \gamma y, z)$ is quasi-Malcev, then $M$ is a $G$-quasi-minority. 
Since this is not the case by assumption, we have that $M(x, \gamma y, z)$ must be a quasi-majority. By Lemma~\ref{ood}, $M(x, \gamma y, z)$ is an odd-majority.
\end{proof}

\begin{notation}
    For the rest of this section $\gamma$ denotes the non-identity element in $\mathbb{Z}_2$.
\end{notation}

Recall Definition~\ref{oddmaldef}, defining an odd Malcev operation.

\begin{lemma}\label{getoddMal} Let $\mathbb{Z}_2$ act freely on $B$. Let $\gamma$ denote the non-identity element of $\mathbb{Z}_2$. Let $M$ be a quasi-Malcev operation such that $M(x,\gamma y, z)$ is an odd-majority. Then, $M$ is an odd Malcev. 
\end{lemma}
\begin{proof} By hypothesis, $m(x,y,z)=M(x,\gamma y, z)$ is an odd-majority. Moreover, $m(x, \gamma y, z)=M(x, y, z)$. Firstly, we can obtain  identity (\ref{oddm1}) since
\[M(x,y,x)\approx m(x, \gamma y, x)\approx m(x, \gamma x, x)\approx M(x, x, x),\]
where the intermediate identity follows from $m$ being an odd-majority, and in particular a quasi-majority. For identity (\ref{oddm2}), we have 
\[M(x, \gamma x, y)\approx m(x, x, y)\approx m(x, \gamma x, x)\approx M(x, x, x),\]
where we use the fact that $m$ is a majority. Similarly, we can prove that
\[M(y, \gamma x, x)\approx M(x,x,x),\]
which together with the previous identity yields (\ref{oddm2}).
Finally, since $m(x, \gamma y, z)$ is an odd majority, 
\[M(x, y, \gamma x)\approx m(x, \gamma y, \gamma x)\approx m(\gamma y, \gamma y, \gamma y)\approx m(\gamma y, y, \gamma y)\approx M(\gamma y, \gamma y, \gamma y).\]
The above yields identity~\ref{oddm3}, concluding the argument.
\end{proof}

\begin{lemma} Let $\mathbb{Z}_2$ act on $B$. Let $f$ be a binary operation almost minimal above ${\langle \mathbb{Z}_2 \rangle}$. Then, $f$ is a binary orbit\-/semiprojection.
\end{lemma}
\begin{proof} Clearly, $f(x,x)$ is an orbit\-/semiprojection if and only if $\gamma\circ f$ is; since one of these functions is idempotent, we may assume without loss of generality that $f$ is. Since $f$ is almost minimal, $f(x, \gamma x)=h(x)$ for some $h\in \mathbb{Z}_2$. If $h$ is the identity, then whenever $a_1,a_2$ are elements from the same orbit, then $f(a_1,a_2)=a_1$, proving the lemma. If $h=\gamma$, then in the same situation we always have $f(a_1,a_2)=\gamma a_1$, again proving the statement.
\end{proof}

\begin{theorem}[Theorem~\ref{thme}, $\mathbb{Z}_2 $ case]\label{Z2case} Let $\mathbb{Z}_2$ act freely on $B$ with $s$-many orbits (where $s$ is possibly infinite). Let $f$ be an almost minimal operation above ${\langle \mathbb{Z}_2 \rangle}$. Then,  $f$ is of one of the following types:
\begin{enumerate}
    \item a unary operation;
    \item a ternary $G$-quasi-minority;
    \item an odd majority;
    \item an odd Malcev (up to permutation of variables);
    \item a $k$-ary orbit\-/semiprojection for $2\leq k \leq s$.
\end{enumerate} 
\end{theorem}
\begin{proof} The theorem follows by the five types theorem (Theorem~\ref{fivetypes}) and the results in this section. If $f$ is a quasi-majority, then it must be an odd majority by Lemma~\ref{ood}. Suppose that $f$ is, up to permuting variables, quasi-Malcev. 
Let $M(x,y,z)$ be the corresponding quasi-Malcev operation. If $M(x,y,z)$ is a $G$-quasi-minority, by symmetry, $f$ is also a $G$-quasi-minority. Otherwise, by Corollary~\ref{Maltooddmaj}, $M(x, \gamma y, z)$ is an odd-majority, and so by Lemma~\ref{getoddMal}, $M(x,y,z)$ is an odd Malcev.
Finally, as usual, if $f$ is a quasi-semiprojection it has to be an orbit\-/semiprojection by Lemma~\ref{semiorb}.
\end{proof}

\begin{remark}\label{duality} 
By Lemma~\ref{majtoMal}, Lemma~\ref{ood} and Lemma~\ref{getoddMal}, we have the following: if $M(x,y,z)$ is almost minimal above ${\langle \mathbb{Z}_2\rangle}$, then $M(x,y,z)$ is an odd Malcev if and only if $M(x, \gamma y, z)$ is an odd majority.
\end{remark}

\begin{prop}\label{prop:optimalZ2} Let $\mathbb{Z}_2$ act freely on $B$. Then, there is an odd majority $m(x,y,z)$ almost minimal above ${\langle \mathbb{Z}_2\rangle}$. Moreover, $m(x, \gamma y, z)$ then is an odd Malcev almost minimal above ${\langle \mathbb{Z}_2\rangle}$.
\end{prop}
\begin{proof} Remark~\ref{duality} already explains that if $m(x,y,z)$ is an odd majority almost minimal above ${\langle \mathbb{Z}_2\rangle}$, then, $m(x, \gamma y, z)$ is an odd Malcev almost minimal above ${\langle \mathbb{Z}_2\rangle}$. Hence, we just need to construct an odd-majority almost minimal above ${\langle \mathbb{Z}_2\rangle}$. \\

It is easy to check that the defining identities of an odd majority $m$ are consistent when the group is a free action of $\mathbb Z_2$ (even imposing idempotency on $m$); $m$ can then be extended to triplets from distinct orbits arbitrarily.
\end{proof}

\subsection{Odd majorities and odd Malcevs cannot be minimal}\label{sub:noalmostminodd}
We just showed that in the case of $\mathbb{Z}_2$ acting freely on a set there are two new almost minimal operations: odd majorities and odd Malcev operations. In this subsection, we show that these cannot be strictly almost minimal in Corollary~\ref{nominmaj}, and so they cannot be minimal. This gives us a classification of minimal operations above $G\acts B$ being a Boolean group acting freely on $B$ (Theorem~\ref{Boolminclass}). The key ingredient of our proof is Lemma~\ref{steinsemi}, which gives us a full understanding of the ternary operations generated by a minimal $G$-quasi-minority (i.e. only essentially unary operations and its translates by group elements). Hence, by proving that odd majorities and odd Malcev operations always generate a $G$-quasi-minority (Lemma~\ref{genmin}), we know these operations cannot be minimal.

\begin{lemma}\label{genmin} Let $\mathbb{Z}_2$ act freely on $B$. Suppose that $m$ is an idempotent odd majority almost minimal above $\langle \mathbb{Z}_2\rangle$. Then, 
\[\mathfrak{m}^\star(x,y,z):= m(x, \gamma m(x,y,z), m(\gamma x, y, z))\]
is a $G$-minority.
\end{lemma}
\begin{proof} Firstly, we evaluate $m_1(x,y,z):=m(\gamma x,y,z)$ on tuples where two elements are equal:
\begin{align*}
m_1(x,x,y)&:=m(\gamma x, x,y)\approx y,\\
m_1(x,y, x)&:=m(\gamma x, y,x)\approx y,\\
    m_1(y,x,x)&:=m(\gamma y, x,x)\approx x,\\
      m_1(x,x,x)&:=m(\gamma x, x,x)\approx x.
\end{align*}
    Now, we evaluate $\mathfrak{m}^\star$ on triplets where two elements are equal
  \begin{align*}
\mathfrak{m}^\star(x,x,y)&:= m(x,\gamma x, y)\approx y,\\
\mathfrak{m}^\star(x,y, x)&:= m(x,\gamma x, y)\approx  y,\\
    \mathfrak{m}^\star(y,x,x)&:= m(y, \gamma x, x)\approx  y,\\
    \mathfrak{m}^\star(x,x,x)&:= m(x, \gamma x, x)\approx x,
\end{align*}
From the identities above we see that $\mathfrak{m}^\star$ is a minority operation. 
Since $\mathfrak{m}^\star$ is almost minimal (being generated by an almost minimal operation) and it is a minority, it must be a $G$-minority.
\end{proof}

\begin{corollary}\label{nominmaj} Let $\mathbb{Z}_2$ act freely on $B$. There are no strictly almost minimal odd majorities or odd Malcev operations above $\langle \mathbb{Z}_2\rangle$.
\end{corollary}
\begin{proof} It is sufficient to prove there cannot be any strictly almost minimal idempotent odd majority: any odd Malcev generates an odd majority and any odd majority $m(x,y,z)$ generates an idempotent odd majority since either $m$ or $\gamma m$ is idempotent.\\
From Lemma~\ref{genmin}, we know that $m$ generates a $G$-minority $\mathfrak{m}^\star$. Since $m$ is strictly almost minimal, $\mathfrak{m}^\star$ must also be strictly almost minimal. However, we know from Lemma~\ref{steinsemi} that the only ternary operations in $\langle G \cup\{\mathfrak{m}^\star\}\rangle$ not in $\langle G \rangle$ are of the form $\eta\mathfrak{m}^\star$ for $\eta\in G$, and so other $G$-quasi-minorities. So
\[m\not\in\langle G\cup\{\mathfrak{m}^\star\}\rangle,\]
implying that $m$ is not strictly almost minimal.
\end{proof}

Since we know that odd majorities and odd Malcev operations are never minimal, we can deduce the classification of minimal operations above a Boolean group acting freely on a set:

\begin{theorem}[Theorem~\ref{thmf}, minimal operations, Boolean case]\label{Boolminclass} Let $G\acts B$ be a Boolean group acting freely on $B$ with $s$-many orbits (where $s$ is possibly infinite). Let $f$ be a minimal operation above $\langle G \rangle$. Then,  $f$ is of one of the following types:
\begin{enumerate}
    \item a unary operation;
    \item a binary operation;
    \item a ternary minority of the form $\alpha \mathfrak{m}$ where $\mathfrak{m}$ is a Boolean $G$-invariant Steiner $3$-quasigroup and $\alpha\in G$;
    \item a $k$-ary orbit\-/semiprojection for $3\leq k \leq s$.
\end{enumerate} 
Furthermore, all Boolean $G$-invariant Steiner $3$-quasigroups on $B$ are minimal and they exist if and only if $s=2^n$ for some $n\in\mathbb{N}$ or is infinite.
\end{theorem}
\begin{proof} Theorems~\ref{Boolcase} and~\ref{Z2case} classify almost minimal operations above a Boolean group acting freely on $B$. Proposition~\ref{prop:bts3g} tells us that minimal $G$-minorities have to be Boolean $G$-invariant $3$-quasigroups. All non-idempotent ones can be obtained by pre-composing an idempotent one by some $\alpha\in G$, and Corollary~\ref{cor:correspondence} gives us necessary and sufficient conditions for the existence of a Boolean $G$-invariant $3$-quasigroup. Theorem~\ref{nominmaj} tells us that there cannot be any strictly almost minimal (and thus no minimal) odd majorities or odd Malcev operations, thus concluding the proof.
\end{proof}

\section{On the existence of minimal orbit-semiprojections}\label{sec:palfy}
In Corollary~\ref{cor:correspondence}, we gave necessary and sufficient conditions for the existence of a minimal $G$-quasi-minority. In this section, we investigate the existence of $k$-ary orbit\-/semiprojections. For semiprojections, a theorem of P\'{a}lfy~\cite{palfy1986arity} yields that there is a $k$-ary semiprojection minimal above the clone of projections for each $2\leq k\leq s$ where $s$ is the size of the domain. Using a similar construction, we prove an analogue for orbit\-/semiprojections in Theorem~\ref{thm:Palfy}. We mimic the idea of P\'{a}lfy of constructing an essential semiprojection as close as possible to a projection in order to have a good understanding of the other operations it generates. With this, we prove that in a finite context orbit\-/semiprojections exist in each arity $2\leq k\leq s$ where $s$ is the number of $G$-orbits. For $B$ infinite, it sounds plausible that there are group actions $G\acts B$ with $k$ many orbits and no $k$-ary orbit\-/semiprojections. Still, in Theorem~\ref{thm:Palfy}, we prove that our result also holds when $G\acts B$ is oligomorphic. 

\begin{notation} For $a,b\in B$, we write $a\sim b$ to say that $a$ and $b$ are in the same orbit.
\end{notation}

\begin{definition}\label{def:palfy} Let $G\acts B$ with $s$ many orbits. Let $b\in B$ be in a $G$-orbit consisting of more than one element. Let {$\alpha,\beta\in\overline{G}$.} For finite $k\leq s$, we define the $k$-ary orbit\-/semiprojection $f^{(\alpha, \beta)}$ as follows: for $(a_1, \dots, a_k)\in B^k$, 
\begin{equation*}
%\label{eq:palfy}
 f^{(\alpha, \beta)}(a_1, \dots, a_k):=
    \begin{cases}
       \beta b \text{ if } a_1\sim b \text{ and } a_1, \dots, a_k \text{ are all in distinct orbits};\\
        \alpha a_1 \text{ otherwise.}
        \end{cases}
        \end{equation*}
We write $f$ for $f^{(1, 1)}$, where $1$ is the identity in $G$.
\end{definition}

\begin{remark}\label{rem:stuffinstuff} Note that any $f^{(\alpha, \beta)}$ for $\alpha, \beta\in G$ is in $\langle G\cup\{f\}\rangle$, being obtained by
\[\beta f(\beta^{-1}\alpha x_1, \dots, x_k).\] 
Conversely, for any $\alpha, \beta\in G$, $f\in \langle G\cup\{f^{(\alpha, \beta)}\}\rangle$. Also note that for $\gamma\in G$,
\[\gamma f^{(\alpha, \beta)}=f^{(\gamma\alpha, \gamma\beta)}.\]

\end{remark}

\begin{lemma}\label{lem:everythingpalfy} {Let $\alpha, \beta\in G$ and $t\in\langle G \cup \{f^{(\alpha, \beta)}\}\rangle \setminus \langle G\rangle$. Then, $t$ is of the form $f^{(\alpha', \beta')}$ for some $\alpha', \beta'\in G$ up to permuting variables and removing dummy variables.}
\end{lemma}
\begin{proof} We prove this by induction on complexity of $t$, where the base case is trivial. We now carry out the inductive step.
{Consider first the case where $t=\gamma t_1$ for $\gamma\in G$ and $t_1\in\langle G \cup \{f^{(\alpha, \beta)}\}\rangle$ satisfying the induction hypothesis. If $t_1\in\langle G\rangle$ then so is $t$; otherwise,  $t_1=f^{(\alpha',\beta')}$ (up to permuting variables and removing dummy variables), and $t=f^{(\gamma \alpha',\gamma\beta')}$ (with the same caveat) by Remark~\ref{rem:stuffinstuff}.}
%{Firstly, if $t$ is of the form $\gamma\delta$ for some $\gamma, \delta\in G$, then $t\in\langle G\rangle$, and so we are not concerned with this case. Secondly, if $t$ is of the form $\gamma f^{(\alpha', \beta')}$ for some $\gamma, \alpha', \beta'\in G$, from Remark~\ref{rem:stuffinstuff}, we know that $t$ is of the form $f^{(\gamma\alpha', \gamma\beta')}$. }
%For the third case of the inductive step, we need to consider $t$ of the form 
{Secondly, consider $t$ of the form}
$f^{(\alpha, \beta)}(t_1, \dots, t_k)$ where the $t_i\in\langle G \cup \{f^{(\alpha, \beta)}\}\rangle$ {satisfy the induction hypothesis.} %are in $\langle G \cup \{f\}\rangle$. By inductive hypothesis, 
We may suitably permute variables to {assume $t_1$} is %obtain $t'$ 
either of the form %$f^{(\alpha, \beta)}(s_1, \dots, s_k)$, where $t_1$ is either of the form 
$\gamma x_1$ or of the form $f^{(\gamma, \delta)}(x_1, \dots, x_k)$ for some $\gamma, \delta\in G$. Also, %by inductive hypothesis 
 each of the $t_i$ for $i>1$ is of the form $\gamma x_{\tau(i)}$ for $\tau(i)\in\mathbb{N}$ and $\gamma_i\in G$, or of the form $f^{(\gamma_i, \delta_i)}(x_{\sigma^i(1)}, \dots, x_{\sigma^i(k)})$, where $\sigma^i$ is some injection $\{1, \dots, k\}\mapsto\mathbb{N}$, and $\gamma_i, \delta_i\in G$. For $i>1$, let $l(t_i)$ be $\tau(i)$ in the first case and $\sigma^i(1)$ in the second case. The identities in the following claim then tell us that {$t$} is either in $\langle G\rangle$ or again of the form $f^{(\alpha', \beta')}$ for $\alpha', \beta'\in G$, completing the proof.\\

\textbf{Claim:} we have the following identities:
\begin{itemize}
    \item If $l(t_i)=i$ for all $2\leq i\leq k$,
    \begin{equation}\label{eq:palfy1}
        f^{(\alpha, \beta)}(\gamma x_1, t_2, \dots, t_k)\approx f^{(\alpha \gamma, \beta)};
    \end{equation}
    \item If $l(t_i)=1$ or $l(t_i)=l(t_j)$ for some $2\leq i<j\leq k$, 
    \begin{equation}\label{eq:palfy2}
        f^{(\alpha, \beta)}(\gamma x_1, t_2, \dots, t_k)\approx \alpha\gamma x_1;
    \end{equation}
    \item If $l(t_i)=i$ for all $2\leq i\leq k$,
    \begin{equation}\label{eq:palfy3}
        f^{(\alpha, \beta)}(f^{(\gamma, \delta)}, t_2, \dots, t_k)\approx f^{(\alpha \gamma, \beta)};
    \end{equation}
     \item If $l(t_i)=1$ or $l(t_i)=l(t_j)$ for some $2\leq i<j\leq k$, 
    \begin{equation}\label{eq:palfy4}
        f^{(\alpha, \beta)}(f^{(\gamma, \delta)}, t_2, \dots, t_k)\approx \alpha f^{(\gamma, \delta)}.
    \end{equation}
\end{itemize}
\begin{proof}[Proof of claim] The identities follow directly from Definition~\ref{def:palfy}. For example, for (\ref{eq:palfy1}), given $(a_1, \dots, a_k)$, we have that
\[f^{(\alpha, \beta)}(\gamma x_1, t_2, \dots, t_k)(a_1, \dots, a_k)=f^{(\alpha, \beta)}(\gamma a_1, a_2', \dots, a_k'),\]
where $a_i'\sim a_i$. Hence, if $(a_1, \dots, a_k)$ were all in distinct orbits and $a_1\sim b$, we get that 
\[f^{(\alpha, \beta)}(\gamma a_1, a_2', \dots, a_k')=\beta b.\]
Otherwise, we have
\[f^{(\alpha, \beta)}(\gamma a_1, a_2', \dots, a_k')=\alpha\gamma a_1.\]
This yields the desired identity in (\ref{eq:palfy1}). For identity (\ref{eq:palfy2}), we get that 
\[f^{(\alpha, \beta)}(\gamma x_1, t_2, \dots, t_k)(a_1, \dots, a_k)=f^{(\alpha, \beta)}(\gamma a_1, a_{l(t_2)}', \dots, a_{l(t_k)}'),\]
where $a_{l(t_k)}'\sim a_{l(t_k)}$. In particular, since either $l(t_i)=1$ or $l(t_i)=l(t_j)$ for $2\leq i<j\leq k$, we are never in the first condition of Definition~\ref{def:palfy}, and so the identity (\ref{eq:palfy2}) holds. The final two identities (\ref{eq:palfy3}) and (\ref{eq:palfy4}) follow by the same reasoning as the previous two.
\end{proof}
\end{proof}

%\begin{corollary}\label{cor:minpalfynoclosed} Let $G\acts B$ with $s$-many orbits. Then, for each finite $1<k\leq s$ there is a $k$-ary orbit\-/semiprojection $f$ depending on all variables such that for every $h\in\langle G\cup\{f\}\rangle\setminus\langle G\rangle$, $f\in \langle G\cup\{h\}\rangle\setminus\langle G\rangle$ and every $h\in\langle G\cup\{f\}\rangle\setminus\langle G\rangle$ has arity $\geq k$. 
%\end{corollary}
%\begin{proof} Use $f$ from Definition~\ref{def:palfy}. The statement holds by Lemma~\ref{lem:everythingpalfy} and Remark~\ref{rem:stuffinstuff}.
%\end{proof}\remmic{I'd remove this corollary and integrate it into the remark after it.}

\begin{remark}\label{rem:minimalityalways}
{From Lemma~\ref{lem:everythingpalfy} and Remark~\ref{rem:stuffinstuff}, we get that whenever $G\acts B$ has $s$-many orbits for $s\geq 2$, for each $1< k\leq s$, the orbit\-/semiprojection $f$ from Definition~\ref{def:palfy} is such that every $h\in\langle G\cup\{f\}\rangle\setminus\langle G\rangle$ has arity $\geq k$ and is such that $f\in \langle G\cup\{h\}\rangle\setminus\langle G\rangle$. In particular, if $\overline{G}=G$, then $f$ is minimal above $\langle G\rangle$. Thus, Lemma~\ref{lem:everythingpalfy} proves the existence of minimal orbit\-/semiprojections above finite non-transitive permutation groups. Moreover, for every $G\acts B$ with $s$-many orbits and $1<k\leq s$, there is a $k$-ary orbit\-/semiprojection $f$ such that $\langle G\cup\{f\}\rangle$ is minimal in the monoidal interval of all (i.e.~not just the closed ones) clones above $\langle G\rangle$. Thus, for minimal operations above permutation groups in the lattice of all clones, we get a version of Theorem~\ref{maintheorem} where $k$-ary orbit\-/semiprojections always exist as minimal above any $G\acts B$ for all $1<k\leq s$ (cf. Remark~\ref{rem:oknoclosure}).}
%\remmic{One of the last three sentence, all of which state the same thing, has to go... this may sound nitpicky but believe me it is very confusing to the reader as he is searching for new information in each sentence (and can't find it)}
\end{remark}

{Finding} minimal orbit\-/semiprojections becomes substantially harder when looking at closed clones above infinite permutation groups. Below, we prove the existence of orbit\-/semiprojections above oligomorphic permutation groups (Theorem~\ref{thm:Palfy}). Note that if $G$ is the automorphism group of an $\omega$-categorical structure in a finite language, the existence of minimal orbit\-/semiprojections above $\overline{\langle G\rangle}$ can be deduced non-constructively from Lemma~\ref{lem:everythingpalfy} and Fact~\ref{existencemin}. However, to deal with the case of $G\acts B$ being oligomorphic in full generality, we need to understand more deeply the structure of the construction of Definition~\ref{def:palfy}.\\

In order to understand what an operation $g\in\overline{\langle G\cup\{f^{(\alpha, \beta)}\}\rangle}$ can generate, we need a slight generalisation of the notion of a two-sided ideal in semigroup theory. Restricting the standard terminology (cf.~\cite[Chapter 1.4]{kilp2011monoids}) to our context, given a semigroup $\mathcal{S}$ and $\mathcal{T}$ a subsemigroup of $\mathcal{S}$, a $(\mathcal{T}, \mathcal{S})$\textbf{-biact} is a set $\mathcal{I}\subseteq\mathcal{S}$ such that $\mathcal{T}\mathcal{I}\mathcal{S}\subseteq\mathcal{I}$. Below, we prove that the construction of Definition~\ref{def:palfy} is tied to the existence of certain minimal closed biacts.

\begin{lemma}\label{lem:biacts} Let $G\acts B$ with $s\geq 2$ many orbits and $b\in B$. Then, there is an orbit\-/semiprojection  of the form $f^{(\gamma, \eta)}$ for $\gamma,\eta\in\overline{G}$ as in Definition~\ref{def:palfy} that is minimal above $\overline{\langle G\rangle}$ if and only if for some $c\sim b$, $\overline{G}$ contains a minimal closed $(\overline{G_{c}}, \overline{G})$-biact. 
\end{lemma}
\begin{proof}  Before proving the equivalence we make a general observation. Take any $f^{(\gamma, \eta)}$ for $\gamma, \eta\in\overline{G}$. From the identities (\ref{eq:palfy1})-(\ref{eq:palfy4}) in Lemma~\ref{lem:everythingpalfy} and Remark~\ref{rem:stuffinstuff}, we can deduce that every essential operation depending on all variables in $\overline{\langle G\cup\{f^{(\gamma, \eta)}\}\rangle}$ is, up to permuting its variables, of the form $f^{(\alpha\gamma\beta, \alpha\eta)}$ for $\alpha, \beta\in\overline{G}$. {For the right-to-left direction of this lemma, we will need the following claim:}\\

\textbf{Claim 1:} {For $\alpha, \beta, \gamma, \eta\in\overline{G}$, we have that}
\[f^{(\gamma\beta, \eta)}\in \overline{\langle G\cup\{f^{(\alpha\gamma\beta, \alpha\eta)}\}\rangle}\;.\]
\begin{proof}[Proof of Claim 1]
To see this, for each $A\subseteq B$ finite, let $\alpha_A\in G$ agree with $\alpha$ on $\eta b\cup\gamma\beta A$. Then, by construction of $f^{(\gamma, \eta)}$,
    \[\alpha_A^{-1}f^{(\alpha\gamma\beta, \alpha\eta)}\upharpoonright_{A}=f^{(\gamma\beta, \eta)}\upharpoonright_{A}\;.\]
%Thus, let $f_A:=\alpha_A^{-1}f^{(\alpha\gamma\beta, \alpha\eta)}$. Since for each finite $A\subseteq B$, $f^{(\gamma\beta, \eta)}$ agrees with $f_A$ on $A$, 
{By local closure, }
\[f^{(\gamma\beta, \eta)}\in\overline{\langle G\cup\{f^{(\alpha\gamma\beta, \alpha\eta)}\}\rangle}\;,\] yielding our desired claim.
\end{proof}
Now, we begin to prove the right-to-left direction of the lemma. Let $\mathcal{I}$ be a minimal closed $(\overline{G}_c, \overline{G})$-biact for some $c\sim b$. Let $\eta\in G$ be such that $\eta b=c$ and let $\gamma\in\mathcal{I}$. We consider $f^{(\gamma, \eta)}$ and prove it is minimal above $\overline{\langle G\rangle}$ . {As observed above, any  essential operation depending on all variables in $\overline{\langle G\cup\{f^{(\gamma, \eta)}\}\rangle}$ is, up to permuting its variables, of the form $f^{(\alpha\gamma\beta, \alpha\eta)}$ for {$\alpha, \beta\in\overline{G}$}, and we need to show it can generate back $f^{(\gamma, \eta)}$. By the above claim, we may replace this function by $f^{(\gamma\beta, \eta)}$.}
% From the above claim, it is sufficient to prove that for any $\beta\in\overline{G}$, $f^{(\gamma, \eta)}\in\overline{\langle G\cup\{f^{(\gamma\beta, \eta)}\}\rangle}$. 
Since $\mathcal{I}$ is a minimal closed $(\overline{G}_c, \overline{G})$-biact, there are $\delta\in \overline{G}_c$ and $\beta^*\in\overline{G}$ such that $\delta\gamma\beta\beta^*=\gamma$. Hence, \[f^{(\delta\gamma\beta\beta^*, \delta\eta)}=f^{(\gamma, \delta\eta)}=f^{(\gamma, \eta)}\;,\]
where the final equality follows from the fact that $\delta c=c$. This proves that $f^{(\gamma, \eta)}$ is minimal.\\

For the left-to-right direction, we prove its contrapositive. Suppose there is no minimal closed $(\overline{G}_c, \overline{G})$-biact for any $c\sim b$. {For any $f^{(\gamma, \eta)}$, with $\gamma, \eta\in\overline{G}$
set $c=\eta b$.}\\

{\textbf{Claim 2:} There are $\xi\in \overline{G_c}, \chi\in \overline{G}$ such that  $\gamma\not\in\overline{G}_c\xi\gamma\chi\overline{G}$.} 
\begin{proof}[Proof of Claim 2]{ Firstly, it is easy to check that for any $\alpha\in\overline{G}$, $\overline{G}_c\alpha\overline{G}$ is a closed $(\overline{G}_c, \overline{G})$-biact and it is contained in any closed $(\overline{G}_c, \overline{G})$-biact containing $\alpha$. Now, suppose by contradiction that for all $\xi\in \overline{G_c}, \chi\in \overline{G}$ we have $\gamma\in\overline{G}_c\xi\gamma\chi\overline{G}$. This implies that for all $\alpha\in \overline{G}_c\gamma\overline{G}$, $\overline{G}_c\alpha\overline{G}=\overline{G}_c\gamma\overline{G}$, meaning that $\overline{G}_c\gamma\overline{G}$ is minimal, contradicting our assumption.}   
\end{proof}
{From Claim 2, and the general observation at the beginning of our proof, any operation of the form $f^{(\gamma, \theta)}\in \overline{\langle G\cup\{f^{(\xi\gamma\chi, \eta)}\}\rangle}$ for $\theta\in \overline{G}$ is such that $\theta=\alpha\eta$ for $\alpha\in\overline{G}\setminus\overline{G}_c$. In particular, $\theta(b)=\alpha c\neq c=\eta(b)$, and so }  
\[f^{(\gamma, \eta)}\not\in\overline{\langle G\cup\{f^{(\xi\gamma\chi, \eta)}\}\rangle},\]
yielding that $f^{(\gamma, \eta)}$ is not minimal. 
\end{proof}

%As mentioned above, we think of a $(\overline{G_b}, \overline{G})$-biact as a variation of the notion of a two-sided ideal, where in the left action we make sure to fix a specified point. We know from~\cite{barto2017equivalence} that $\overline{G}$ has a minimal closed left-ideal when $G\acts B$ is oligomorphic \remmic{While true, the only minimal closed left ideal is $\overline{G}$, so it does not make too much sense to state this (the statement about minimal closed left ideals is only interesting for larger monoids). The proof is the same as Claim 1 above. (I know it was me who started talking about minimal closed left ideals but that was just my first intuition, now we know what's needed is a biact...) I think the entire pragraph should be deleted.}. Slightly generalising the classical result of Clifford~\cite{clifford1948semigroups} that the existence of a minimal left-ideal implies the existence of a (unique) minimal two-sided ideal, one can see that the existence of a (unique) minimal closed left-ideal in a transformation monoid implies the existence of a minimal closed two-sided ideal, which corresponds to the closure of the union of all of the minimal closed left-ideals. Hence, when $G\acts B$ is oligomorphic, $\overline{G}$ contains a minimal two-sided ideal. Meanwhile, it is easy to see that $\overline{G}$ will not have a minimal right-ideal the moment that it is not a group of permutations.\\

Below, we prove the existence of minimal closed $(\overline{G}_b, \overline{G})$-biacts whenever $G\acts B$ is an oligomorphic permutation group. For our proof we need a {compactness argument} %basic construction 
 %which %was also needed 
 {used} to show the existence of a minimal closed left-ideal in this context in~\cite{barto2017equivalence} (cf.~\cite{Topo-Birk, bodirsky2011n0}): 

\begin{lemma}[{Lemma 4 in~\cite{pinsker2021canonical}}]\label{lem:compactness} Let $G\acts B$ be oligomorphic and $b\in B$. Let $\sim_b$ be the equivalence relation on  $\overline{G}$ where $\alpha\sim_b\beta$ indicates that $\alpha\in\overline{G_b\beta}$. Then, the space $\overline{G}/{\sim_b}$ is compact and Hausdorff.
\end{lemma}

%Hence, we prove the existence of minimal closed $(\overline{G}_b, \overline{G})$-biacts in oligomorphic permutation groups:

\begin{lemma}\label{lem:existencebiacts} Let $G\acts B$ be oligomorphic. Then, there is a minimal closed $(\overline{G}_b, \overline{G})$-biact $\mathcal{I}$.
\end{lemma}
\begin{proof}
Let $\gamma\in\overline{G}$ {and $n\geq 1$}. By the $(n,b)$\textbf{-orbit range} of $\gamma$ we mean the set of $G_b$-orbits of $n$-tuples  occurring in the range of $\gamma$. The $b$\textbf{-orbit range} of $\gamma$ is the union over all $n\in\mathbb{N}$ of its $(n,b)$-orbit ranges. Note that if $\alpha\sim_b\beta$, then $\alpha$ and $\beta$ have the same $b$-orbit range.\\

\textbf{Claim 1:} There is $\gamma\in\overline{G}$ with minimal $b$-orbit range, in the sense that there is no $\delta\in \overline{G}$ with $b$-orbit range strictly contained in that of $\gamma$.
\begin{proof}[Proof of Claim 1] 
{By oligomorphicity of $G_b\acts B$, for each $n\in\mathbb{N}$, there is $\gamma_n\in\overline{G}$ with minimal $(n,b)$-orbit range in the sense that there is no $\delta\in\overline{G}$ with $(n,b)$-orbit range strictly contained in that of $\gamma_n$. This is just because for each $n$ there are only finitely many possible $(n,b)$-orbit ranges. Note that {this also implies that} each $\gamma_n$ has minimal $(l,b)$-orbit range for each $l\leq n$. Consider the sequence $(\gamma_n)_{n\in\mathbb{N}}$. By compactness of $\overline{G}/{\sim_b}$  (Lemma~\ref{lem:compactness}), the sequence of $\sim_b$-equivalence classes $([\gamma_n]_{\sim_b})_{n\in\mathbb{N}}$ has an accumulation point $[\gamma]_{\sim_b}$ in  $\overline{G}/{\sim_b}$. We claim that $\gamma$ has minimal $b$-orbit range. To see this, suppose by contradiction that for some $n$, $\gamma$ does not have minimal $(n,b)$-orbit range. Then, by oligomorphicity of $G_b$, there are finite tuples $\overline{a}$ and $\overline{c}$ such that $\gamma(\overline{a})=\overline{c}$ and there is some $\delta\in\overline{G}$ whose range contains fewer $G_b$-orbits of $n$-tuples than $\overline{c}$. Consider the following open {set} %sets 
in $\overline{G}$:}
\[ \mathcal{W}:=\bigcup_{\overline{d}\in G_b\cdot  \overline{c}}\{\xi\in\overline{G} \ \vert \xi(\overline{a})=\overline{d}\}\;.\]

%\[\mathcal{U}_{(\overline{a}, \overline{c})}:=\{\xi\in\overline{G} \ \vert \xi(\overline{a})=\overline{c}\} \text{ and } \mathcal{W}:=\bigcup_{\alpha\in \overline{G_b}}\alpha\cdot \mathcal{U}_{(\overline{a}, \overline{c})}\;.\]
{Notice that $[\mathcal{W}]_{\sim_b}$ is an open neighbourhood of $\gamma$ in $\overline{G}/{\sim_b}$ since its preimage under the quotient map $\overline{G}\mapsto\overline{G}/\sim_b$ is $\mathcal{W}$. However, since the $(n,b)$-orbit range of $\gamma_m$ for $m\geq n$ is minimal, none of the $([\gamma_m]_{\sim_b})_{m\geq n}$ are in $[\mathcal{W}]_{\sim_b}$, contradicting  that $[\gamma]_{\sim_b}$ is an accumulation point of $([\gamma_n]_{\sim_b})_{n\in\mathbb{N}}$.  Thus, $\gamma$ must have minimal $b$-orbit range.}
\end{proof}

\textbf{Claim 2:} Suppose that $\gamma\in\overline{G}$ has minimal $b$-orbit range. Then, $\overline{G}_b\gamma\overline{G}$ is a minimal closed $(\overline{G}_b, \overline{G})$-biact. 
\begin{proof}[Proof of Claim 2] It is easy to see that $\overline{G}_b\gamma\overline{G}$ is a closed $(\overline{G}_b, \overline{G})$-biact. So, we need to prove minimality. To do this, it is sufficient to prove that for any $\alpha\gamma\beta$ for $\alpha\in\overline{G_b}$, $\beta\in\overline{G}$, we have that $\gamma\in\overline{G}_b\alpha\gamma\beta \overline{G}$. Firstly, by essentially the same argument as Claim 1 of Lemma~\ref{lem:biacts}, $\gamma\beta\in \overline{G}_b\alpha\gamma\beta \overline{G}$. Let $(b_i)_{i\in\mathbb{N}}$ be an enumeration of $B$, {and let $n\geq 1$}. Since we chose $\gamma$ to have minimal $b$-orbit range and $\gamma\beta(B)\subseteq\gamma(B)$, %for $(b_i)_{i\leq n}$ 
 there is some tuple $(b_i')_{i\leq n}$ such that $(\gamma \beta b_i')_{i\leq n}$ is in the same $G_b$-orbit as $(\gamma b_i)_{i\leq n}$.  {Since $\gamma,\beta\in\overline{G}$ {we have that} $(b_i)_{i\leq n}$ and $(b_i')_{i\leq n}$ are in the same $G$-orbit and so we can choose $\beta^*_n\in G$ such that $\beta_n^*b_i=b_i'$ for all $i\leq n$.} Also, since $(\gamma \beta b_i')_{i\leq n}$ is in the same $G_b$-orbit as $(\gamma b_i)_{i\leq n}$, there is $\delta_n\in G_b$ such that  $\delta_n\gamma\beta b_i'=\gamma b_i$ for all $i\leq n$. %Now, consider $\delta_n\gamma\beta\beta^*_n$. Note that for each $n\in\mathbb{N}$ we have that
{Then,} 
    \[{\delta_n\gamma\beta\beta^*_n}\upharpoonright_{(b_i)_{i\leq n}}= \gamma \upharpoonright_{(b_i)_{i\leq n}},\]
    meaning that {$(\delta_n\gamma\beta\beta^*_n)_{n\geq 1}$ } converges to $\gamma$, yielding $\gamma\in\overline{G}_b\gamma\beta\overline{G}\subseteq \overline{G}_b\alpha\gamma\beta\overline{G}$.
\end{proof}
\end{proof}

Hence, from Lemmas~\ref{lem:biacts} and~\ref{lem:existencebiacts} (and more easily from Lemma~\ref{lem:everythingpalfy} in the finite context, cf. Remark~\ref{rem:minimalityalways}), we can deduce the following version of P\'{a}lfy's Theorem for orbit\-/semiprojections:

\begin{theorem}[Theorem~\ref{thmg}, P\'{a}lfy's Theorem for orbit\-/semiprojections]\label{thm:Palfy} Let $G\acts B$ with $s$-many orbits be finite or oligomorphic. Then, for all finite $1<k\leq s$, there is a $k$-ary orbit\-/semiprojection minimal above $\overline{\langle G\rangle}$.
\end{theorem}

It is evident from Lemma~\ref{lem:biacts} that for arbitrary group actions $G\acts B$ our construction in Definition~\ref{def:palfy} may fail to yield any orbit\-/semiprojection minimal above $\overline{\langle G\rangle}$. However, this does not exclude the existence of other minimal orbit\-/semiprojections, perhaps less close in behaviour to an actual projection than the ones built from Definition~\ref{def:palfy}. Other constructions of minimal semiprojections over finite domains (such as that of~\cite{lengvarszky1986note}) do not seem to adapt as well as 
P\'{a}lfy's~\cite{palfy1986arity} to infinite settings since they tend to construct a strictly almost minimal operation above $\langle\mathrm{Id}\rangle$ and then deduce the existence of a minimal operation from the fact these exist over finite domains. Hence, it seems that new ideas would be needed for more general constructions than the one of Definition~\ref{def:palfy}. Nevertheless, it sounds plausible that there should be groups whose behaviour is wild enough that there are no orbit\-/semiprojections minimal above $\overline{\langle G\rangle}$. Hence, we ask:

\begin{prob} Give an example of $G\acts B$ with three orbits such that there is no ternary orbit\-/semiprojection minimal above $\overline{\langle G\rangle}$. 
\end{prob}

This does not sound like an easy problem to answer: for specific choices of $f$, one can show that there is no orbit\-/semiprojection in $\overline{\langle G\cup\{f\}\rangle}$ minimal above $\overline{\langle G\rangle}$ like we did in Lemma~\ref{lem:biacts}. However, it is unclear how on an infinite domain one may run an argument that any possible orbit\-/semiprojection is not minimal.

\section{Solvability in Datalog without binary injections}\label{sec:datalog}

From Theorem~\ref{thm:findbin}, we know that whenever $B$ is an $\omega$-categorical model-complete core whose $\mathrm{CSP}$ is not $\mathrm{NP}$-hard, $\mathrm{Pol}(B)$ has a binary essential polymorphism. Frequently in the literature on infinite-domain $\mathrm{CSP}$s, the existence of a binary essential polymorphism is used to deduce the existence of a binary injective polymorphism under some additional assumptions~\cite{BodKara, bodirsky2010complexity, minrandom, Smoothapp, bodirsky2020complexity}. Binary injective polymorphisms are  helpful when reducing $\mathrm{CSP}(B)$ to its injective variant~\cite{Smoothapp}, and more generally in most complexity classifications on an infinite domain~\cite{bodirsky2010complexity, minrandom, Smoothapp}. Moreover, they naturally appear as one of the Nelson-Opper criteria~\cite{nelson1979simplification} for deducing that the generic combination of two $\omega$-structures with no algebraicity and trivial $\mathrm{CSP}$ also has trivial $\mathrm{CSP}$~\cite{bodirsky2020complexity, bodirsky2022tractable}.\\

We conclude by giving a counterexample to the following question of Bodirsky on binary injective polymorphisms:

\begin{customq}{\ref{q:datalog}}(Question 14.2.6 (27) in {\cite{BodCSP}})
   Does every $\omega$-categorical structure without algebraicity that can be solved
by Datalog also have a binary injective polymorphism?
\end{customq}

  Datalog is a logic programming language which captures many standard polynomial-time algorithms used in the context of CSPs, such as path consistency. Indeed, solvability in Datalog of a CSP corresponds to solvability by $k$-consistency algorithms for some $k$~\cite{feder1998computational, bodirsky2013datalog}. We refer the reader to~\cite{bodirsky2013datalog} for a discussion of Datalog in an $\omega$-categorical setting and for the definition of a Datalog program. In this section, we use an algebraic characterisation of solvability in Datalog for a particular class of $\omega$-categorical structures~\cite{mottet2021symmetries}, thus avoiding the need to formally introduce the syntax and semantics of Datalog. Its significance for Question~\ref{q:datalog} is the idea that binary injective polymorphisms would always appear for $\omega$-categorical structures with easy enough $\mathrm{CSP}$s. Without the restriction to Datalog, some first-order reducts of $(\mathbb{Q}, <)$ have $\mathrm{CSP}$ solvable in polynomial time (though not in Datalog) but have no binary injective polymorphisms (cf. Theorem 1.4 in~\cite{bodirsky2022tractable} and Claim 4.8 in~\cite{bodirsky2022descriptive}).\\

Our counterexample (Definition~\ref{def:defD}) is a finitely bounded homogeneous structure; it is also %and 
 a first-order reduct of the unary structure consisting of two infinite disjoint unary predicates partitioning a countable domain. The $\mathrm{CSP}$s of infinite unary structures (and their reducts) have been studied in detail with a complexity dichotomy proved in~\cite[Theorems 1.1 \& 1.2]{bodirsky2018dichotomy} and a descriptive complexity analysis given in~\cite[Theorem 5.1]{mottet2021symmetries}. The latter characterises solvability in Datalog for reducts of unary structures as being witnessed by a particular class of (canonical) polymorphisms known as pseudo-WNU operations (Definition~\ref{def:WNU}). We will use this description for our proof of solvability in Datalog since it clarifies the intuition behind the counterexample.\\ %However, it is also easy to prove that our counterexample is solvable in Datalog by providing an explicit program, and we do so in Appendix~\ref{sec:appdatalog}. \\

The starting point of our proof is that any structure on the Boolean domain $\{0,1\}$ with a minimum operation among its polymorphisms has $\mathrm{CSP}$ solvable in Datalog (because it contains WNU operations in all arities). So, the idea of our counterexample is to ``blow up'' the points of a structure on $\{0,1\}$ with a minimum operation to countable sets $P_0$ and $P_1$,  and then destroy all binary injective polymorphisms by adding relations  simple enough to keep the $\mathrm{CSP}$ of the enriched  structure in Datalog. %to ensure this 
%which %whilst %trying to have that the $\mathrm{CSP}$ of the resulting structure is still solvable in Datalog.} 

%At this point we want to consider a function $\mathrm{m}$ acting as a minimum on these predicates (i.e. always selecting the element in $P_0$ if its arguments are one from $P_0$ and one from $P_1$). The challenge is balancing the requirement that $\mathrm{m}$ should generate a very rich clone (so that structures that have it as a polymorphism are solvable in Datalog), whilst also ensuring it does not generate any non-injective unary operation or any injective binary operation (which would break the requirements for a counterexample). This leads us to the set of operations $\mathfrak{M}$ (Definition~\ref{def:mdef}). After showing that any model-complete core with automorphism group $\mathrm{Aut}(D)$ and containing some operation in $\mathfrak{M}$ is solvable in Datalog (Corollary~\ref{cor:solvabilitycriteriondat}), we give an example of a structure in a finite relational language for which this is the case, i.e. $D'$ in Definition~\ref{def:defD}. We then show that $D'$ contains no binary injective operations, concluding the proof and yielding an answer to Question~\ref{q:datalog} (Corollary~\ref{cor:conclusionq3}).

\begin{definition} Let $D$ be the first-order structure with domain the disjoint union of two copies of $\mathbb{N}$ in the language with two unary predicates $\{P_0,P_1\}$, where each predicate names one of the two disjoint copies (which we also call $P_0$ and $P_1$). 
\end{definition}

\begin{definition}\label{def:mdef}
For $i\in\{0,1\}$, let $\mathrm{Inj}_i$ be the set of injections $P_i\to P_i$ and $\mathrm{BInj}_i$ be the set of binary injections $P_i^2\to P_i$. Now, define $\mathfrak{M}'$ to be the set of binary operations $\mathrm{m}:D^2\to D$ such that, for some $g_0\in \mathrm{BInj}_0, g_1\in \mathrm{BInj}_1$, and $\alpha, \beta\in \mathrm{Inj}_0$, we have
\[\mathrm{m}(x,y):=\begin{cases} g_1(x,y) \text{ if } x,y\in P_1;\\
g_0(x,y) \text{ if } x,y\in P_0;\\
\alpha x \text{ if } x\in P_0, y\in P_1;\\
\beta y \text{ if } y\in P_0, x\in P_1.
    \end{cases}
    \]
So, $\mathrm{m}$ acts as a binary injection on each $\mathrm{Aut}(D)$-orbit and if its inputs are from different orbits, it selects the one in $P_0$ (and moves it by an appropriate injection of $P_0$). Sometimes we may write $\mathrm{m}_{(\alpha, \beta, g_0, g_1)}$ when we want to specify the functions used to construct $\mathrm{m}$.\\

We write $\mathfrak{M}$ for the operations in $\mathfrak{M}'$ such that $g_0, \alpha,$ and $\beta$ have disjoint images.
\end{definition}

As mentioned earlier, solvability in Datalog for reducts of unary structures is witnessed by canonical pseudo-WNU operations. Below we define each of these terms.

\begin{definition}[{\cite{pinsker2021canonical}}] Let $G\acts B$. We say that an $n$-ary operation $f:B^n\to B$ is \textbf{canonical} with respect to $G$ (or $G$-canonical) if for all $k\geq 1$, all $k$-tuples $\overline{a}_1, \dots, \overline{a}_n\in B^k$, and all $\alpha_1, \dots, \alpha_n\in G$, there is $\beta\in G$ such that
\begin{equation}\label{eq:canondef}
    f(\alpha_1 \overline{a}_1, \dots, \alpha_n \overline{a}_n)=\beta f(\overline{a}_1, \dots, \overline{a}_n),
\end{equation}
where $f$ and $\alpha$ are applied componentwise. 

%That is, writing $\overline{a}_i=(a_i^1, \dots, a_i^k)$ for $i\leq n$, 
%\[f(\alpha_1 \overline{a}_1, \dots, \alpha_n \overline{a}_n)=\begin{pmatrix}
 %   f(\alpha_1 a^1_1, \dots, \alpha_na_n^1)\\
%    \vdots\\
%    f(\alpha_1a_1^k, \dots, \alpha_n a_n^k)
%\end{pmatrix}\;.\]
Let $\mathfrak{C}^G$ denote the clone of $G$-canonical functions on $B$. Note that the definition of a canonical function in (\ref{eq:canondef}) says that the $G$-orbit of the $k$-tuple $f(\overline{a}_1, \dots, \overline{a}_n)$ is entirely determined by the $G$-orbits of the $k$-tuples $\overline{a}_1, \dots, \overline{a}_n$. In particular, writing $\mathcal{T}_k$ for the space of $G$-orbits of $k$-tuples on $B$, the $G$-canonical function $f$ induces an $n$-ary operation $\xi_k^G(f)$ on $\mathcal{T}_k$ defined as follows (cf.~\cite{BPP-projective-homomorphisms}): for $O_1, \dots, O_n\in\mathcal{T}_k$, let $\xi_k^G(f)(O_1, \dots, O_n)$ be given by the orbit of $f(\overline{a}_1, \dots, \overline{a}_n)$ for some/all $\overline{a}_i\in O_i$. This is well-defined by canonicity. It is easy to see (and well-known) that the map $\xi_k^G:\mathfrak{C}^G\to \mathcal{O}_{\mathcal{T}_k}$ is a continuous clone homomorphism.  For any closed clone $\mathcal{C}\subseteq \mathfrak{C}^G$, we denote by $\xi_k^G(\mathcal{C})$ its image with respect to $\xi_k^G$.
%which is a closed clone on $\mathcal{T}_k$.
\end{definition}

\begin{definition}\label{def:WNU} An $n$-ary operation $w:B^n\to B$ is called a \textbf{weak near-unanimity operation (WNU)} if for any two tuples $\overline{z}_1, \overline{z}_2\in\{x,y\}^n$ containing exactly one instance of $y$, we have
\[w(\overline{z}_1)\approx w(\overline{z}_2).\]
So, for example, a ternary WNU operation $w$ satisfies 
\[w(x,x,y)\approx w(x,y,x)\approx w(y,x,x).\]
\end{definition}

\begin{definition} 
%Let $\tau$ be a functional language. Let $\Sigma$ be a set of identities in the language $\tau$. Let $\mathcal{C}$ be a clone on $B$ and $G\acts B$. We say that $\mathcal{C}$ satisfies \textbf{pseudo-}$\bm{\Sigma}$ \textbf{modulo} $\bm{\overline{G}}$ if the function-symbols from $\tau$ can be assigned functions in $\mathcal{C}$ so that, for each identity $t_1(\overline{x})\approx t_2(\overline{y})\in \Sigma$, where $t_1$ and $t_2$ are $\tau$-terms, there are $\alpha_1, \alpha_2\in\overline{G}$ such that $\mathcal{C}$ satisfies
%\[\alpha_1 t_1(\overline{x})\approx \alpha_2 t_2(\overline{y})\;.\]
%For example, 
We say that $w:B^n\to B$ is a \textbf{pseudo-WNU} operation (modulo $\overline{G}$) if for any $\overline{z}_1, \overline{z}_2\in\{x,y\}^n$ containing exactly one instance of $y$ there are $\alpha_{\overline{z}_1}, \beta_{\overline{z}_2}\in\overline{G}$ such that $\alpha_{\overline{z}_1}w(\overline{z}_1)\approx\beta_{\overline{z}_2} w(\overline{z}_2)$.
\end{definition}

\begin{theorem}[Theorem 5.1 in~\cite{mottet2021symmetries}]\label{thm:symmetriesmott} Let $B$ be a model complete core and such that $\mathrm{Aut}(B)$ contains $\mathrm{Aut}(C)$, where $C$ is a structure in a finite unary language. Then, $\mathrm{CSP}(B)$ is solvable in Datalog if and only if the clone of $\mathrm{Aut}(C)$-canonical polymorphisms for $B$ has pseudo-WNU operations  with respect to $\overline{\mathrm{Aut}(C)}$ in all arities $n\geq 3.$
\end{theorem}

To check whether $\mathrm{Pol}(B)$ has canonical pseudo-WNU operations in our context (where $G=\mathrm{Aut}(D)$), it is sufficient to focus on 
the image of $\mathrm{Pol}(B)\cap\mathfrak{C}^G$ (i.e.~the canonical polymorphisms of $B$) under $\xi_2^G$. 
This is due to the following instantiation of Proposition 6.6 in~\cite{BPP-projective-homomorphisms}:

\begin{lemma}[{Proposition 6.6 in~\cite{BPP-projective-homomorphisms}}]\label{lem:pclonhom} Let $B$ be a homogeneous structure in a finite relational language $\mathcal{L}$ and with automorphism group $G$. Let $\mathcal{C}\supseteq G$ be a closed clone of canonical functions with respect to $G$. Let $l:=\mathrm{max}\{2, |\mathcal{L}|\}$. Suppose that $\xi_l^G(\mathcal{C})$ satisfies WNUs in all arities $n\geq 3$. Then, $\mathcal{C}$ satisfies pseudo-WNUs modulo $\overline{G}$ in all arities $n\geq 3$. 
\end{lemma}

\begin{lemma}\label{lem:WNUsobtain} Let $\mathrm{m}\in\mathfrak{M}$. Then $\mathrm{m}$ is $\mathrm{Aut}(D)$-canonical. Moreover, $g:=\xi_2^G(\mathrm{m})$ satisfies $g(x,g(x,y))\approx g(x,y)$, and it generates WNU operations of all arities $\geq 3$ on $\mathcal{T}_2$.
\end{lemma}
\begin{proof} We first prove canonicity of $\mathrm{m}$. Note that it is sufficient to check canonicity of $\mathrm{m}$ on orbits of pairs. This is because the $\mathrm{Aut}(D)$-orbit of an $n$-tuple is entirely determined by its restrictions to the orbits of all of its pairs (cf. p.156 in~\cite{BPP-projective-homomorphisms}). We have six different orbitals to consider:
\begin{itemize}
    \item for $i\in\{0,1\}$, $O_i^=$, the orbit of a pair of equal elements both in $P_i$;
    \item for $i\in\{0,1\}$, $O_i^{\neq}$, the orbit of a pair of distinct elements both in $P_i$;
    \item the orbitals $O_{(0,1)}$ and $O_{(1,0)}$, denoting, respectively, for $a_0\in P_0, b_1\in P_1$, the orbits of $(a_0, b_1)$ and of $(b_1, a_0)$.
\end{itemize}
We now show that for $\overline{a}, \overline{b}\in D^2$, the orbital of $\mathrm{m}(\overline{a}, \overline{b})$ is entirely determined by the orbitals of $\overline{a}$ and $\overline{b}$. We can compute the following, exploiting the fact that the domains of $\alpha, \beta,$ and $g_0$ are disjoint and injections:
\begin{enumerate}
    \item for $o_0^{{\neq}}\in O_0^{\neq}$ and $q\in D^2$, we always have $\mathrm{m}(o^{\neq}_0, q)\in O^{\neq}_0$;
    \item for $o_1^{\neq}\in O_1^{\neq}$ and $q\in D^2$, $\mathrm{m}(o^{\neq}_1, q)$ is in the same orbital as $q$;
    \item for $o_0^=\in O_0^=$ and $q\not\in O_0^=$, $\mathrm{m}(o_0^=, q)\in O_0^{\neq}$;
    \item for $o_0^=, q_0^=\in O_0^=$, $\mathrm{m}(o_0^=, q_0^=)\in O_0^=$;
    \item for $o_1^{=}\in O_1^{=}$ and $q\in D^2$, $\mathrm{m}(o^=_1, q)$ is in the same orbital as $q$;
    \item for $J\in \{(0,1), (1,0)\}$ and $o_J, q_J\in O_J$, $\mathfrak{m}(o_J, q_J)\in O_J$;
    \item for $I, J\in \{(0,1), (1,0)\}$ distinct, and $o_I\in O_I, o_J\in O_J$, $\mathfrak{m}(o_I, o_J)\in O_0^{\neq}$.
\end{enumerate}
Hence, $\mathrm{m}$ is canonical. We can also see from the list above that $g(x,g(x,y))\approx g(x,y)$ (since $g$ is just the operation induced by $\mathfrak{m}$ on the orbitals). Firstly, note that $g$ is idempotent. In cases (1), (3), (4), and (6), $g(x,y)=x$, and so $g(x,g(x,y))=g(x,y)$ by idempotency. In cases (2) and (4), $g(x,y)=y$, yielding again $g(x,g(x,y))=g(x,y)$. In case (7), $g(O_I, O_J)=O^{\neq}_0$. By the same reasoning as in (1), $g(O_I, O^{\neq}_0)=O^{\neq}_0$. Hence, also in this case, the identity holds, yielding that $g(x,g(x,y))\approx g(x,y)$.

For $n\geq 3$, consider the operation on $\mathcal{T}_2^n$ given by
\[w_n(x_1, \dots, x_n):=g(x_1, g(x_2, \dots ,g(x_{n-1}, x_n)\dots )).\]
It follows readily from the identity $g(x,g(x,y))\approx g(x,y)$ and the idempotency and commutativity of $g$ that $w_n(\overline{z})\approx g(x,y)$ for all non-constant $\overline{z}\in\{x,y\}^n$; in particular $w_n$ is a weak near unanimity operation.
%It is easy to see that $w_n$ is a weak near unanimity operation from the identity $g(x,g(x,y))\approx g(x,y)$, and idempotency and commutativity of $g$.
\end{proof}

\begin{remark} Whilst this is unnecessary for our purposes, one can also prove that $\mathfrak{M}$ is precisely the set of $\mathrm{Aut}(D)$-canonical operations in $\mathfrak{M}'$.
\end{remark}

\begin{corollary}\label{cor:solvabilitycriteriondat} Let $C$ be a model complete core in a finite language with automorphism group $\mathrm{Aut}(D)$ and with $\mathrm{Pol}(C)\cap\mathfrak{M}\neq\emptyset$. Then, $\mathrm{CSP}(C)$ is solvable in Datalog. 
\end{corollary}
\begin{proof} We are assuming that $\mathrm{Pol}(C)$ contains a  polymorphism in $\mathfrak{M}$, which is canonical by Lemma~\ref{lem:WNUsobtain}. By  Lemma~\ref{lem:WNUsobtain}, this means that $\xi_2^{\mathrm{Aut}(D)}(\mathrm{Pol}(C))$ contains WNU operations of all arities $n\geq 3$. By Lemma~\ref{lem:pclonhom}, $\mathrm{Pol}(C)$ contains pseudo-WNUs of all arities $n\geq 3$. Theorem~\ref{thm:symmetriesmott} tells us this implies solvability in Datalog of $\mathrm{CSP}(C)$.
\end{proof}

\begin{definition}\label{def:defD} Let $(F; R_1, \dots, R_m)$ be any finite relational structure on $\{0,1\}$ with the minimum operation being its only binary essential polymorphisms\footnote{Such structures exist. Indeed, any structure on $\{0,1\}$ with constants for $0$ and $1$ and such all of its relations are definable by a Horn formula, but one of its relations is not definable by a dual-Horn formula will work (c.f.~\cite{schaefer1978complexity}, see~\cite[Theorem 6.2.7]{BodCSP}). We say that a formula in conjunctive normal form is Horn where each clause has at most one positive literal. We say it is dual-Horn if each clause has at most one negative literal. So, an example of a ternary relation which is Horn but not dual-Horn is $\{0,1\}^3\setminus(1,1,0)$.}. For each relation $R_i$ of $F$, %of arity $l_i$, 
let $R_i^*$ be its ``blow-up'' to $D$, i.e.~the preimage of $R_i$ under the map which sends $P_j$ onto $j$ for $j\in\{0,1\}$. %That is, for $(a_1, \dots, a_{l_i})\in D$ with $a_i\in P_{j(i)}$ for $j(i)\in\{0,1\}$, 
%\[R_i^*(a_1, \dots, a_{l_i}) \text{ holds if and only if } R_i(j(1), \dots, j(l_i)) \text{ does in } F.\]
Note that each of the $R_i^*$ is first-order definable in $D$ since the original $R_i$ were first-order definable %definable in $F$ by making use of 
using only the constants for $0$ and $1$.

We define two additional relations on $D$. Firstly, the $4$-ary relation $E$ denotes orbital equivalence of pairs; it will ensure all polymorphisms of our structure will be canonical. Secondly, the following relation will help destroy all binary injective polymorphisms:
\[ p(x,y) := (x=y\wedge P_0(x))\vee(P_1(x)\wedge P_1(y))\; .\]
Now, consider the first-order structure
\[D':=(D; P_0, P_1, R^*_1, \dots, R^*_m, E, p, \neq)\;.\]
\end{definition}

\begin{lemma}\label{lem:Datacore} The structure $D'$ is a model-complete core. 
\end{lemma}
\begin{proof} We show that $\overline{\mathrm{Aut}(D')}=\mathrm{End}(D')$, which is given by 
\[\mathcal{G}:=\{g_0\cup g_1\vert g_0\in \mathrm{Inj}(P_0), g_1\in\mathrm{Inj}(P_1)\} \;.\]
Let $g\in\mathrm{End}(D')$. We know that $g$ preserves $P_0$ and $P_1$ and that on each set it must act as an injection since it preserves $\neq$. In particular, $g\in\mathcal{G}$. Conversely, any function in $\mathcal{G}$ preserves all relations of $D'$. Finally, we know that $\mathrm{Aut}(D')=\mathrm{Aut}(D)$ and that the latter is
\[\{\sigma_0\cup \sigma_1\vert \sigma_0\in\mathrm{Bij}(P_0), \sigma_1\in\mathrm{Bij}(P_1)\}\; ,\]
where $\mathrm{Bij}(P_i)$ is the set of bijections of the countably infinite set denoted by $P_i$.
Hence, $\overline{\mathrm{Aut}(D')}=\mathrm{End}(D')$.
\end{proof}

\begin{lemma}\label{lem:containment} We have that $\mathfrak{M}\subseteq \mathrm{Pol}(D')$. 
\end{lemma}
\begin{proof} Since any polymorphism of $D'$ preserves the relation of orbit-equivalence on pairs, all polymorphisms in $\mathrm{Pol}(D')$ are $\mathrm{Aut}(D)$-canonical. We then see that  $\mathrm{Pol}(D')$ consists precisely of those $\mathrm{Aut}(D)$-canonical operations for which $\xi_2(f)$ is a polymorphism of $F$ and which additionally preserve $\neq$ and $p$. 
For any $\mathrm{m}\in\mathfrak{M}$, $\xi_2(\mathrm{m})$ is the operation of minimum on $F$, which is a polymorphism of $F$.  From the proof of  Lemma~\ref{lem:WNUsobtain}, it is easy to verify that all operations in $\mathfrak{M}$ preserve inequality. 
Since we know that $\mathrm{m}$ preserves $P_1$ and $O_{0}^=$ (by Lemma~\ref{lem:WNUsobtain}), to see that $p$ is preserved by $\mathrm{m}$, we only need to consider whether $p$ holds of $(\mathrm{m}(a_0, b_1), \mathrm{m}(a_0, c_1))$, where $a_0\in P_0$, and $b_1, c_1\in P_1$. This is the case because $\mathrm{m}(a_0, b_1)=\alpha a_0=\mathrm{m}(a_0, c_1)$.
\end{proof}

\begin{lemma}\label{lem:bininj}
  Let $g$ be a binary polymorphism of $D'$. Then, $g$ is not injective.
\end{lemma}
\begin{proof} If $g$ is an essential polymorphism of $D'$, then $\xi_2(g)$ must be the minimum operation on $\{0,1\}$. Hence, for $a_0\in P_0$ and $b_1\in P_1$, $g(a_0, b_1)\in P_0$.
Consider $c_1\in P_1$ and $(g(a_0, b_1), g(a_0, c_1))$. Note that $(a_0, a_0)\in p$ and $(b_1, c_1)\in p$ and $g(a_0, b_1)\in P_0$ and $g(d_0, b_1)\in P_0$. But since $g$ preserves $p$, we must have $g(a_0, b_1)=g(a_0, c_1)$, which breaks injectivity.
\end{proof}

\begin{corollary} \label{cor:WNUprogram} The structure $D'$ has $\mathrm{CSP}$ solvable in Datalog.
\end{corollary}
\begin{proof} We know that $D'$ is a core from Lemma~\ref{lem:Datacore}. Also, $\mathrm{Aut}(D')=\mathrm{Aut}(D)$. From Lemma~\ref{lem:containment}, $\mathfrak{M}\subseteq\mathrm{Pol}(D')$. Hence, $\mathrm{CSP}(D')$ is solvable in Datalog by Corollary~\ref{cor:solvabilitycriteriondat}.
\end{proof}

We say that an $\omega$-categorical structure $B$ has \textbf{algebraicity} if for some finite $A\subseteq B$, some orbit of the stabilizer of $A$ %, %$\mathrm{Aut}(B/A)$, 
acting on $B\setminus A$ is finite. 

\begin{corollary}\label{cor:conclusionq3} There is a finitely bounded homogeneous model-complete core without algebraicity whose $\mathrm{CSP}$ can be solved by a Datalog program but which has no binary injective polymorphism.
\end{corollary}
\begin{proof} The structure $D$ is clearly homogeneous and finitely bounded without algebraicity. Since $D'$ is obtained by expanding $D$ by definable relations, it also has all of these properties. We proved that $D'$ is a model-complete core in Lemma ~\ref{lem:Datacore}. We know that $D'$ is solvable by a Datalog program by Lemma~\ref{lem:containment} and Corollary~\ref{cor:WNUprogram}. Finally, we know that $D'$ has no binary injections from Lemma~\ref{lem:bininj}.
\end{proof}

%\newpage
 \printbibliography

\end{document}